\documentclass[11pt, twoside]{article}

\usepackage{mathrsfs}
\usepackage{amssymb}
\usepackage{amsmath}
\usepackage{mathrsfs}
\usepackage{amsthm}
\usepackage{amsfonts}
\usepackage{color}
\usepackage{latexsym}
\usepackage{txfonts}
\usepackage{indentfirst}
\usepackage{anysize}
\usepackage{tikz}

\usepackage[colorlinks=true,
  linkcolor=red,
  citecolor=blue,
  urlcolor=magenta]{hyperref}

\allowdisplaybreaks

\newtheorem{theorem}{Theorem}[section]
\newtheorem{lemma}[theorem]{Lemma}
\newtheorem{corollary}[theorem]{Corollary}
\newtheorem{proposition}[theorem]{Proposition}
\theoremstyle{definition}
\newtheorem{remark}[theorem]{Remark}

\newtheorem{definition}[theorem]{Definition}

\renewcommand{\appendix}{\par
\setcounter{section}{0}%
\setcounter{subsection}{0}%
\setcounter{subsubsection}{0}%
\gdef\thesection{\@Alph\c@section}%
\gdef\thesubsection{\@Alph\c@section.\@arabic\c@subsection}%
\gdef\theHsection{\@Alph\c@section.}%
\gdef\theHsubsection{\@Alph\c@section.\@arabic\c@subsection}%
\csname appendixmore\endcsname
}

\numberwithin{equation}{section}

\begin{document}

\title{\bf\Large One-Sided
and Parabolic BLO Spaces with Time Lag and Their Applications
to Muckenhoupt $A_1$ Weights
and Doubly Nonlinear Parabolic Equations
\footnotetext{\hspace{-0.35cm} 2020
{\it Mathematics Subject Classification}.
Primary 42B35; Secondary 42B25, 35K55, 28A75.
\endgraf
{\it Key words and phrases.}
one-sided BLO space,
parabolic BLO space with time lag,
one-sided Muckenhoupt weight,
parabolic Muckenhoupt weight with time lag,
parabolic maximal operator with time lag,
doubly nonlinear parabolic equation.
\endgraf
This project is partially supported by
the National Natural Science Foundation of China
(Grant Nos. 12431006 and 12371093),
the Beijing Natural Science Foundation (Grant No. 1262011),
and the Fundamental Research Funds
for the Central Universities (Grant No. 2253200028).}}
\date{}
\author{Weiyi Kong, Dachun Yang\footnote{Corresponding
author, E-mail: \texttt{dcyang@bnu.edu.cn}/{\color{red}
\today}/Final version.}\ \
and Wen Yuan}

\maketitle

\vspace{-0.8cm}

\begin{center}
\begin{minipage}{13cm}
{\small {\bf Abstract}\quad In this article, we first introduce the
one-sided BLO space $\mathrm{BLO}^+(\mathbb{R})$ and characterize it,
respectively, in terms of the one-sided Muckenhoupt class
$A_1^+(\mathbb{R})$ and the one-sided John--Nirenberg inequality.
Using these, we establish the Coifman--Rochberg type decomposition
of $\mathrm{BLO}^+(\mathbb{R})$ functions and show that
$\mathrm{BLO}^+(\mathbb{R})$
is independent of the distance between
the two intervals, which further induces the characterization of
this space in terms of the one-sided BMO space
$\mathrm{BMO}^+(\mathbb{R})$ (the Bennett type lemma).
As applications, we prove that
any $\mathrm{BMO}^+(\mathbb{R})$ function can split into the sum of
two $\mathrm{BLO}^+(\mathbb{R})$ functions and we provide an
explicit description of the distance from $\mathrm{BLO}^+(\mathbb{R})$
functions to $L^\infty(\mathbb{R})$. Finally, as a
higher-dimensional analogue we introduce the parabolic
BLO space $\mathrm{PBLO}_\gamma^-(\mathbb{R}^{n+1})$ with time lag, and
we extend all the above one-dimensional results to $\mathrm{PBLO}_\gamma^-(\mathbb{R}^{n+1})$;
furthermore, as applications, we not only establish the relationships between
$\mathrm{PBLO}_\gamma^-(\mathbb{R}^{n+1})$ and the solutions of doubly
nonlinear parabolic equations, but also provide a necessary condition for the negative
logarithm of the parabolic distance function to belong to
$\mathrm{PBLO}_\gamma^-(\mathbb{R}^{n+1})$ in terms of the weak porosity of the set.
}
\end{minipage}
\end{center}

\vspace{0.2cm}

\tableofcontents

\vspace{0.2cm}

\section{Introduction}
\label{section1}
The space of functions with bounded mean oscillation,
$\mathrm{BMO}(\mathbb{R}^n)$, introduced by John and Nirenberg
\cite{jn(cpam-1961)}, plays an essential role in harmonic analysis and
partial differential equations. Recall that the space
$\mathrm{BMO}(\mathbb{R}^n)$ is defined to be the set of all locally
integrable functions $f$ on $\mathbb{R}^n$ such that
\begin{align*}
\|f\|_{\mathrm{BMO}(\mathbb{R}^n)}:=\sup_{Q\in\mathcal{Q}}\frac{1}{|Q|}
\int_Q\left|f(x)-\frac{1}{|Q|}\int_Qf(y)\,dy\right|\,dx<\infty,
\end{align*}
where $\mathcal{Q}$ denotes the set of all cubes $Q$ in $\mathbb{R}^n$
with edges parallel to the axes. The space $\mathrm{BMO}(\mathbb{R}^n)$
has close relationship with the \emph{Muckenhoupt class
$A_q(\mathbb{R}^n)$}, which consists of all nonnegative locally
integrable functions $\omega$ on $\mathbb{R}^n$ such that
\begin{align*}
[\omega]_{A_q(\mathbb{R}^n)}:=
\begin{cases}
\displaystyle
\sup_{Q\in\mathcal{Q}}\frac{1}{|Q|}\int_Q\omega(x)\,dx
\left[\mathop\mathrm{ess\,inf}_{x\in Q}\omega(x)\right]^{-1}<\infty
&\mathrm{if}\ q=1,\\
\displaystyle
\sup_{Q\in\mathcal{Q}}\frac{1}{|Q|}\int_Q\omega(x)\,dx\left[\frac{1}{|Q|}
\int_Q\omega(x)^\frac{1}{1-q}\,dx\right]^{q-1}<\infty
&\mathrm{if}\ q\in(1,\infty),\\
\displaystyle
\sup_{Q\in\mathcal{Q}}\frac{1}{|Q|}\int_Q\omega(x)\,dx
\exp\left\{\frac{1}{|Q|}\int_Q\ln\frac{1}{\omega(x)}\,dx
\right\}<\infty &\mathrm{if}\ q=\infty.
\end{cases}
\end{align*}
It is well known that, for any given $q\in(1,\infty]$,
\begin{align*}
\mathrm{BMO}(\mathbb{R}^n)=\left\{\lambda\ln\omega:\omega\in
A_q(\mathbb{R}^n)\mbox{\ and\ }\lambda\in\mathbb{R}\right\}
\end{align*}
(see, for instance, \cite[Chapter V, 6.2]{Stein-book}).
This characterization fails to hold for the endpoint case $q=1$.
In order to remedy this defect, Coifman
and Rochberg \cite{cr(pams-1980)} introduced the space of functions with
bounded lower oscillation, $\mathrm{BLO}(\mathbb{R}^n)$, as a proper
subset of $\mathrm{BMO}(\mathbb{R}^n)$. Formally, the space
$\mathrm{BLO}(\mathbb{R}^n)$ is defined to be the set of all
locally integrable functions $f$ on $\mathbb{R}^n$ such that
\begin{align*}
\|f\|_{\mathrm{BLO}(\mathbb{R}^n)}:=\sup_{Q\in\mathcal{Q}}\frac{1}{|Q|}
\int_Q\left|f(x)-\mathop\mathrm{ess\,inf}_{y\in Q}f(y)\right|\,dx<\infty.
\end{align*}
They \cite{cr(pams-1980)} showed that all the BLO functions are precisely
nonnegative multiples of logarithms of all the $A_1$ weights:
\begin{align}\label{20251012.1456}
\mathrm{BLO}(\mathbb{R}^n)=\left\{\lambda\ln\omega:\omega\in
A_1(\mathbb{R}^n)\mbox{\ and\ }\lambda\in[0,\infty)\right\}.
\end{align}
Bennett \cite{b(pams-1982)} established a characterization of
$\mathrm{BLO}(\mathbb{R}^n)$ in terms of the natural maximal operator. We
refer to \cite{a(na-2024), gk(jfa-2023), 2602.00479, lny(jfsa-2011), lny(bsm-2011),
mny(na-2010), ns(tmj-2017)} for more studies on the space
$\mathrm{BLO}(\mathbb{R}^n)$.

On the other hand, over the past few decades, the one-dimensional
one-sided weighted inequalities have developed rapidly. Sawyer
\cite{s(tams-1986)} and Mart\'in-Reyes et al. \cite{mpt(cjm-1993)}
introduced the \emph{one-sided Muckenhoupt class $A_q^+(\mathbb{R})$},
which consists of all nonnegative locally integrable functions $\omega$
on $\mathbb{R}$ such that
\begin{align}\label{20251011.1606}
[\omega]_{A_q^+(\mathbb{R})}:=
\begin{cases}
\displaystyle
\mathop\mathrm{ess\,sup}_{x\in\mathbb{R}}\frac{M^-(\omega)(x)}
{\omega(x)}<\infty&\mathrm{if}\ q=1,\\
\displaystyle
\sup_{I\in\mathcal{I}}\frac{1}{|I^-|}\int_{I^-}\omega(x)\,dx
\left[\frac{1}{|I^+|}\int_{I^+}\omega(x)^\frac{1}{1-q}\,dx\right]
^{q-1}<\infty &\mathrm{if}\ q\in(1,\infty),\\
\displaystyle
\sup_{I\in\mathcal{I}}\frac{1}{|I^-|}\int_{I^-}\omega(x)\,dx
\exp\left\{\frac{1}{|I^+|}\int_{I^+}\ln\frac{1}{\omega(x)}\,dx
\right\}<\infty &\mathrm{if}\ q=\infty.
\end{cases}
\end{align}
In what follows, let $\mathcal{I}$ denote the set of all bounded open
intervals in $\mathbb{R}$ and, for any $I:=(a,b)\in\mathcal{I}$, let
$$I^-:=\left(a,\frac{a+b}{2}\right), \ I^+:=\left(\frac{a+b}{2},b\right),
\ \mathrm{and}\ I^{++}:=\left(b,b+\frac{b-a}{2}\right).$$
In addition, $M^-$ denotes the
\emph{one-sided maximal operator}, defined by setting, for any locally
integrable function $f$ on $\mathbb{R}$ and for any $x\in\mathbb{R}$,
\begin{align*}
M^-(f)(x):=\sup_{h\in(0,\infty)}\frac{1}{h}\int_{(x-h,x)}|f(y)|\,dy.
\end{align*}
If condition \eqref{20251011.1606} holds with the orientation of the
real line reversed, then $\omega\in A_q^-(\mathbb{R})$ (the other
\emph{one-sided Muckenhoupt class} on $\mathbb{R}$). One-sided
Muckenhoupt weights provide a natural framework for
characterizing the weighted boundedness of one-sided operators, including
the one-sided variants of (fractional) maximal operators, singular
integral operators, fractional integral operators, and discrete
square operators. In the 1990s, Mart\'in-Reyes and de la Torre
\cite{mt(jlms-1994)} introduced the \emph{one-sided $\mathrm{BMO}$ space
$\mathrm{BMO}^+(\mathbb{R})$}, which is defined to be the set of all
locally integrable functions $f$ on $\mathbb{R}$ such that
\begin{align*}
\|f\|_{\mathrm{BMO}^+(\mathbb{R})}:=\sup_{I\in\mathcal{I}}\frac{1}{|I^-|}
\int_{I^-}\left[f(x)-\frac{1}{|I^+|}\int_{I^+}f(y)\,dy\right]_+\,dx<\infty.
\end{align*}
Here, and thereafter, for any function $f$, let
$$f_+:=\max\{f,\,0\}\ \mathrm{and}\ f_-:=-\min\{f,\,0\}.$$
If the condition above holds with the orientation
of the real line reversed, then $f\in\mathrm{BMO}^-(\mathbb{R})$ (the
other \emph{one-sided $\mathrm{BMO}$ space} on $\mathbb{R}$).
By employing the one-sided John--Nirenberg inequality, they
\cite{mt(jlms-1994)} proved that, for any given $q\in(1,\infty]$,
all the $\mathrm{BMO}^+(\mathbb{R})$ functions
are exactly nonnegative multiples of logarithms of all the
$A_q^+(\mathbb{R})$ weights:
\begin{align}\label{20250727.1611}
\mathrm{BMO}^+(\mathbb{R})=\left\{\lambda\ln\omega:\omega\in
A_q^+(\mathbb{R})\mbox{\ \ and\ \ }\lambda\in[0,\infty)\right\}.
\end{align}
One can readily show that \eqref{20250727.1611} fails to hold for the
endpoint case $q=1$ (see Remark \ref{Linfty BLO+ BMO+}). This leads to
the following natural question: \emph{What kind of suitable alternative
of $\mathrm{BMO}^+(\mathbb{R})$ makes \eqref{20250727.1611} hold with
$q=1$}? One of the main targets of this article is to give an affirmative
answer to this question, motivated by \cite{cr(pams-1980)}.

Additionally, in higher dimensions the parabolic BMO space
$\mathrm{PBMO}_\gamma^-(\mathbb{R}^{n+1})$ with time lag serves as the
natural higher-dimensional analogue of the one-sided BMO space. Moreover,
from the perspective of partial differential equations, the parabolic BMO
space with time lag is of vital importance in the regularity theory for
the doubly nonlinear parabolic equation
\begin{align}\label{20251009.2225}
\frac{\partial(|u|^{p-2}u)}{\partial t}
-\mathrm{div}\left(|\nabla u|^{p-2}\nabla u\right)=0,
\end{align}
where $p\in(1,\infty)$ is \emph{fixed} throughout this article (see, for
example, Moser \cite{m(cpam-1964), m(cpam-1967)}, Fabes and Garafalo
\cite{fg(pams-1985)}, and Aimar \cite{a(tams-1988)}). Analogous to the
one-dimensional case, all the functions in the parabolic BMO space with
time lag are precisely nonnegative multiples of logarithms of all the
parabolic Muckenhoupt $A_q^+(\gamma)$ weights with time lag,
$q\in(1,\infty]$, introduced by Kinnunen and Saari
\cite{ks(apde-2016)} and Kinnunen and Myyryl\"ainen
\cite{km(am-2024)} (see Section \ref{section5} for
more details). We refer to \cite{ac(asnsp-1998), lmr(prses-2023),
mr(prses-2000), or(pm-2003), s(ampa-2018), tt(sm-2003)} for further
investigations on the one-sided BMO type space and to \cite{ckyyz-2025,
cm24, km(jam-2025), kmy(ma-2023), kmyz(pa-2023),
ks(na-2016),kyyz(cvpde-2025), 2509.24486, my(mz-2024), s(rmi-2016)} for
more recent studies on the parabolic BMO type space with time lag.

Motivated by the aforementioned works, in this article, we first introduce
the one-sided BLO space $\mathrm{BLO}^+(\mathbb{R})$ and characterize it,
respectively, in terms of $A_1^+(\mathbb{R})$ (and hence answer the aforementioned
question) and the one-sided John--Nirenberg type inequality. Using these,
we establish the Coifman--Rochberg type decomposition
of $\mathrm{BLO}^+(\mathbb{R})$ functions and prove that the space
$\mathrm{BLO}^+(\mathbb{R})$ is independent of
the distance between the two intervals, which further induces
the characterization of $\mathrm{BLO}^+(\mathbb{R})$
in terms of $\mathrm{BMO}^+(\mathbb{R})$
(the Bennett type lemma). As
applications, we show that any $\mathrm{BMO}^+(\mathbb{R})$ function can
split into the sum of two $\mathrm{BLO}^+(\mathbb{R})$ functions and we
provide an explicit description of the distance from
$\mathrm{BLO}^+(\mathbb{R})$ functions to $L^\infty(\mathbb{R})$.
Finally, as a higher-dimensional analogue we introduce the
parabolic BLO space $\mathrm{PBLO}_\gamma^-(\mathbb{R}^{n+1})$
with time lag, and we extend all the above one-dimensional results to
$\mathrm{PBLO}_\gamma^-(\mathbb{R}^{n+1})$; furthermore, as applications, we not only establish
the relationships between $\mathrm{PBLO}_\gamma^-(\mathbb{R}^{n+1})$ and the solutions of doubly
nonlinear parabolic equations, but also provide a necessary condition for the negative
logarithm of the parabolic distance function to belong to
$\mathrm{PBLO}_\gamma^-(\mathbb{R}^{n+1})$ in terms of the weak porosity of the set.

The organization of the remainder of this article is as follows.

In Section \ref{section2}, we introduce the one-dimensional one-sided BLO
space $\mathrm{BLO}^+(\mathbb{R})$ (see Definition \ref{BLO+}) as a proper
subset of $\mathrm{BMO}^+(\mathbb{R})$. Then we prove that all the
functions in $\mathrm{BLO}^+(\mathbb{R})$ are exactly nonnegative multiples
of logarithms of all the $A_1^+(\mathbb{R})$ weights (see Theorem
\ref{BLO+ and A1+}) and we characterize $\mathrm{BLO}^+(\mathbb{R})$ in
terms of the one-sided John--Nirenberg inequality (see Theorem
\ref{JN BLO+ 2}) by using a one-dimensional one-sided Calder\'on--Zygmund
decomposition which is a slight variant of the one used in the proof of
\cite[Lemma 1]{mt(jlms-1994)}. We further show that one can take two
intervals that are not contiguous in the definition of
$\mathrm{BLO}^+(\mathbb{R})$ (see Theorem \ref{BLO+ time lag}).

In Section \ref{section3}, applying Theorem \ref{BLO+ time lag}, a one-sided
Calder\'on--Zygmund decomposition from the proof of \cite[Theorem
5.1]{s(ampa-2018)} and an iterative method adapted from the one used in
the proof of \cite[Theorem 5.1]{s(ampa-2018)}, we prove that
the uncentered one-sided natural maximal operator $N^-$ is bounded from
$\mathrm{BMO}^+(\mathbb{R})$ to $\mathrm{BLO}^+(\mathbb{R})$ (see
Proposition \ref{N-:BMO+ to BLO+}), which improves
\cite[Theorem 5.1]{s(ampa-2018)} via replacing the target space
$\mathrm{BMO}^+(\mathbb{R})$ by its proper subspace
$\mathrm{BLO}^+(\mathbb{R})$ [see Remark \ref{N-:BMO+ to BLO+ remark}(ii)
for the details]. As a consequence, we establish the
Bennett type characterization of $\mathrm{BLO}^+(\mathbb{R})$, which in
turn represents $\mathrm{BLO}^+(\mathbb{R})$ as the image of
$\mathrm{BMO}^+(\mathbb{R})$ under $N^-$, modulo bounded functions (see
Theorem \ref{N-:BMO+ to BLO+ cor}).

In Section \ref{section4}, using
Theorem \ref{BLO+ and A1+}, we establish the Coifman--Rochberg type
decomposition of functions in $\mathrm{BLO}^+(\mathbb{R})$ (see Theorem
\ref{C-R BLO+}). As corollaries, we show that any
$\mathrm{BMO}^+(\mathbb{R})$ function can split into the sum of two
$\mathrm{BLO}^+(\mathbb{R})$ functions (see Corollary
\ref{BMO+=BLO+-BLO-}) and we obtain an explicit quantification
of the distance from $\mathrm{BLO}^+(\mathbb{R})$ functions to
$L^\infty(\mathbb{R})$ (see Corollary \ref{distance BLO+ to Linfty}).

Section \ref{section5}  is structured as follows. In Subsection
\ref{subsection5.1}, we introduce the parabolic BLO space
$\mathrm{PBLO}_\gamma^-(\mathbb{R}^{n+1})$ with time lag (see Definition
\ref{PBLO+}) and we prove that all the
$\mathrm{PBLO}_\gamma^-(\mathbb{R}^{n+1})$ functions are precisely
nonnegative multiples of logarithms of all the parabolic $A_1^+(\gamma)$ weights with
time lag (see Theorem \ref{PBLO+ and A1+}), which answers the
aforementioned question on the endpoint case $q=1$ of
\eqref{20250727.1611} with $\mathrm{BMO}^+(\mathbb{R})$ replaced by
$\mathrm{PBMO}_\gamma^-(\mathbb{R}^{n+1})$. Applying this, we establish a
characterization of $\mathrm{PBLO}_\gamma^-(\mathbb{R}^{n+1})$ in terms
of the parabolic John--Nirenberg inequality (see Theorem
\ref{J-N PBLO+ 2}) and we show that
$\mathrm{PBLO}_\gamma^-(\mathbb{R}^{n+1})$ is independent of the
time lag and the distance between the three rectangles (see Theorem
\ref{PBLO+ time lag}). In Subsection \ref{subsection5.2}, by using
Theorem \ref{PBLO+ time lag} and a parabolic chaining lemma in
\cite[Remark 3.4]{s(ampa-2018)} and by establishing a parabolic
Calder\'on--Zygmund decomposition, we establish the boundedness
of the uncentered parabolic natural maximal operator $N^{\gamma-}$ from
$\mathrm{PBMO}_\gamma^-(\mathbb{R}^{n+1})$ to
$\mathrm{PBLO}_\gamma^-(\mathbb{R}^{n+1})$
(see Proposition \ref{N-:PBMO- to PBLO-}), which improves \cite[Lemma
4.1]{s(ampa-2018)} via replacing the target space
$\mathrm{PBMO}_\gamma^-(\mathbb{R}^{n+1})$ by its proper subspace
$\mathrm{PBLO}_\gamma^-(\mathbb{R}^{n+1})$
[see Remark \ref{r5.11}(ii) for the details]. Applying Proposition
\ref{N-:PBMO- to PBLO-}, we further establish the Bennett type
characterization of $\mathrm{PBLO}_\gamma^-(\mathbb{R}^{n+1})$ (see
Theorem \ref{N-:PBMO- to PBLO- cor}). In Subsection \ref{subsection5.3},
using Theorem \ref{PBLO+ and A1+}, we give the Coifman--Rochberg
type decomposition of functions in
$\mathrm{PBLO}_\gamma^-(\mathbb{R}^{n+1})$ (see Theorem \ref{C-R PBLO-}).
As corollaries, we show that any
$\mathrm{PBMO}_\gamma^-(\mathbb{R}^{n+1})$ function can be represented as
the sum of two $\mathrm{PBLO}_\gamma^-(\mathbb{R}^{n+1})$ functions (see
Corollary \ref{PBMO-=PBLO--PBLO+}) and we obtain an explicit
quantification of the distance from
$\mathrm{PBLO}_\gamma^-(\mathbb{R}^{n+1})$ functions to
$L^\infty(\mathbb{R}^{n+1})$ (see Corollary
\ref{distance PBLO- to Linfty}), which is a refinement of \cite[Theorem
5.1]{kyyz(cvpde-2025)} by characterizing the equivalence therein via a
strict equality. In Subsection \ref{subsection5.4}, as applications of
all these results obtained in this section, we establish the
relationships between $\mathrm{PBLO}_\gamma^-(\mathbb{R}^{n+1})$ and the
solutions of doubly nonlinear parabolic equations (see Proposition
\ref{PBLO- and PDE} and Corollaries \ref{PBLO- and PDE cor} and
\ref{PBLO- and PDE cor 2}), which is the second one of the two main targets of this article.
In Subsection \ref{subsection5.5}, as an application of Theorem \ref{PBLO+ and A1+}, we provide
a necessary condition for the negative logarithm of the parabolic distance function $-\log
d_p(\cdot,E)$ to belong to $\mathrm{PBLO}_\gamma^-(\mathbb{R}^{n+1})$ in terms of the parabolic
$\gamma$-FIT weak porosity of the set $E\subset\mathbb{R}^{n+1}$
(see Theorem \ref{PBLO- imply porous}).

We end this introduction by making some conventions on
symbols. Throughout this article, let $\mathbb{N}:=\{1,2,\ldots\}$
and $\mathbb{Z}_+:=\{0\}\cup\mathbb{N}$. For any $s\in\mathbb{R}$, the
symbol $\lceil s\rceil$ denotes the smallest integer not less than $s$
and the symbol $\lfloor s\rfloor$ denotes the largest integer not greater
than $s$. Let $\mathcal{X}\in\{\mathbb{R},\,\mathbb{R}^{n+1}\}$. For any
measurable set $A\subset\mathcal{X}$, we denote by $|A|$ its Lebesgue
measure. Let $L_{\mathrm{loc}}^1(\mathcal{X})$ be the set of all locally
integrable functions on $\mathcal{X}$. For any $f\in
L_{\mathrm{loc}}^1(\mathcal{X})$ and any measurable set
$A\subset\mathcal{X}$ with $|A|\in(0,\infty)$, let
$$f_A:=\fint_Af:=\frac{1}{|A|}\int_Af.$$
If $f\in L_{\mathrm{loc}}^1(\mathcal{X})$ is positive almost everywhere,
then $f$ is called a \emph{weight}. For any given $q\in[1,\infty]$, the
\emph{Lebesgue space $L^q(\mathcal{X}$)} is defined to be the space of
all measurable functions $f$ on $\mathcal{X}$ such that
\begin{align*}
\|f\|_{L^q(\mathcal{X})}:=
\begin{cases}
\displaystyle
\left[\int_\mathcal{X}|f(x)|^q\,dx\right]^\frac{1}{q}
&\mathrm{if}\ q\in[1,\infty),\\
\displaystyle\mathop\mathrm{ess\,sup}_{x\in\mathcal{X}}|f(x)|&\mathrm{if}\ q=\infty\\
\end{cases}
\end{align*}
is finite. If there exists no confusion, then
we \emph{always omit} the differential in all integral representations.
In addition, we \emph{always suppress} the variable in the notation
and, for example, for any $A\subset\mathbb{R}$, any function $f$ on
$\mathbb{R}$, and any $\lambda\in\mathbb{R}$, we simply write
$A\cap\{f>\lambda\}:=\{x\in A:f(x)>\lambda\}$.  The symbol $f\lesssim g$
means that there exists a positive constant $C$ such that $f\le Cg$ and,
if $f\lesssim g\lesssim f$, we then write $f\sim g$. Finally, in all
proofs, we consistently retain the symbols introduced in the original
theorem (or related statement).

\section{Characterizations of $\mathrm{BLO}^+(\mathbb{R})$
in Terms of $A_1^+(\mathbb{R})$ and\\John--Nirenberg
Inequality}
\label{section2}

In this section, we first demonstrate that the
one-sided Muckenhoupt class $A_1^+(\mathbb{R})$ remains invariant under
arbitrary separation between the two defining intervals. Subsequently, we
introduce the one-sided BLO space $\mathrm{BLO}^+(\mathbb{R})$ and
characterize it in terms of both the one-sided John--Nirenberg inequality
and $A_1^+(\mathbb{R})$. Furthermore, we prove that
the space $\mathrm{BLO}^+(\mathbb{R})$ is independent of the
distance between the two intervals.

We begin with listing two basic properties of one-sided Muckenhoupt
classes. Firstly, if $q\in(1,\infty)$, then
$A_q^+(\mathbb{R})$ has the following connection with
$A_\infty^+(\mathbb{R})$ and $A_\infty^-(\mathbb{R})$, which can be
deduced from the proof of \cite[Theorem 2]{mpt(cjm-1993)}.

\begin{lemma}\label{Ap+ and Ainfty +}
Let $q\in(1,\infty)$ and $\omega$ be a weight. Then $\omega\in
A_q^+(\mathbb{R})$ if and only if $\omega\in A_\infty^+(\mathbb{R})$ and
$\omega^\frac{1}{1-q}\in A_\infty^-(\mathbb{R})$. Moreover,
\begin{align*}
\max\left\{[\omega]_{A_\infty^+(\mathbb{R})},\,
\left[\omega^\frac{1}{1-q}\right]_{A_\infty^-(\mathbb{R})}^{q-1}\right\}
\leq[\omega]_{A_q^+(\mathbb{R})}\leq[\omega]_{A_\infty^+(\mathbb{R})}
\left[\omega^\frac{1}{1-q}\right]_{A_\infty^-(\mathbb{R})}^{q-1}.
\end{align*}
\end{lemma}

Secondly, we point out that the distance between the two intervals
is inessential in the definition of $A_q^+(\mathbb{R})$, where
$q\in[1,\infty]$. We mention that the range $(0,\frac{1}{2}]$ of $\gamma$
in the following lemma is the natural and the best range.

\begin{lemma}\label{Ap+ time lag}
Let $\gamma\in(0,\frac{1}{2}]$. Then the following statements hold.
\begin{enumerate}
\item[\rm(i)] $\omega\in A_1^+(\mathbb{R})$ if and only if
\begin{align}\label{20250626.1622}
\langle\omega\rangle_{A_1^+(\mathbb{R})}:=
\sup_{\genfrac{}{}{0pt}{}{-\infty<a<b\leq c<d<\infty}
{b-a=d-c=\gamma(d-a)}}\fint_{(a,b)}\omega
\left[\mathop\mathrm{ess\,inf}_{(c,d)}\omega\right]^{-1}<\infty.
\end{align}
Moreover, $[\omega]_{A_1^+(\mathbb{R})}\leq
\langle\omega\rangle_{A_1^+(\mathbb{R})}\leq
\frac{1}{\gamma}[\omega]_{A_1^+(\mathbb{R})}$.

\item[\rm(ii)] For any given $q\in(1,\infty]$, $\omega\in
A_q^+(\mathbb{R})$ if and only if
\begin{align*}
\langle\omega\rangle_{A_q^+(\mathbb{R})}:=
\begin{cases}
\displaystyle
\sup_{\genfrac{}{}{0pt}{}{-\infty<a<b\leq c<d<\infty}
{b-a=d-c=\gamma(d-a)}}\fint_{(a,b)}\omega
\left[\fint_{(c,d)}\omega^\frac{1}{1-q}\right]^{q-1}<\infty &\mbox{if\ }
q\in(1,\infty),\\
\displaystyle
\sup_{\genfrac{}{}{0pt}{}{-\infty<a<b\leq c<d<\infty}
{b-a=d-c=\gamma(d-a)}}\fint_{(a,b)}\omega
\exp\left\{\fint_{(c,d)}\ln\frac{1}{\omega}
\right\}<\infty &\mbox{if\ }q=\infty.
\end{cases}
\end{align*}
Moreover, $[\omega]_{A_q^+(\mathbb{R})}$ and
$\langle\omega\rangle_{A_q^+(\mathbb{R})}$ only depend on each other and $\gamma$.
\end{enumerate}
\end{lemma}

\begin{proof}
(ii) was given in \cite[Lemma 2.6]{rt(cmj-2001)} for the case
$q\in(1,\infty)$ and in \cite[Theorem 1]{mpt(cjm-1993)} for the case
$q=\infty$. Now, we show (i). Define the \emph{uncentered one-sided
maximal operator} $\mathscr{M}^-$ by setting, for any $f\in
L_{\mathrm{loc}}^1(\mathbb{R})$ and $x\in\mathbb{R}$,
\begin{align*}
\mathscr{M}^-(f)(x):=\sup_{\genfrac{}{}{0pt}{}{-\infty<a<b\leq
c<d<\infty}{b-a=d-c=\gamma(d-a),\,x\in(c,d)}}\fint_{(a,b)}|f|.
\end{align*}
We can easily prove that, for any $f\in L_{\mathrm{loc}}^1(\mathbb{R})$,
$M^-(f)\leq\mathscr{M}^-(f)\leq\frac{1}{\gamma}M^-(f)$. Thus,
$\omega\in A_1^+(\mathbb{R})$ if and only if
\begin{align}\label{20250614.1712}
[\omega]^*_{A_1^+(\mathbb{R})}:=
\left\|\frac{\mathscr{M}^-(\omega)}{\omega}\right\|
_{L^\infty(\mathbb{R})}<\infty.
\end{align}
Moreover, $[\omega]_{A_1^+(\mathbb{R})}\leq[\omega]^*_{A_1^+(\mathbb{R})}
\leq\frac{1}{\gamma}[\omega]_{A_1^+(\mathbb{R})}$. We further claim that
$\langle\omega\rangle_{A_1^+(\mathbb{R})}=[\omega]^*_{A_1^+(\mathbb{R})}$.
Indeed, suppose that \eqref{20250614.1712} holds. For any
$-\infty<a<b\leq c<d<\infty$ with $b-a=d-c=\gamma(d-a)$ and for
any $x\in(c,d)$,
\begin{align*}
\fint_{(a,b)}\omega\leq\mathscr{M}^-(\omega)(x)
\leq[\omega]^*_{A_1^+(\mathbb{R})}\omega(x).
\end{align*}
Taking the essential infimum over all $x\in(c,d)$, we obtain
\eqref{20250626.1622} and $\langle\omega\rangle_{A_1^+(\mathbb{R})}
\leq[\omega]^*_{A_1^+(\mathbb{R})}$. Conversely, assume that
\eqref{20250626.1622} holds. Let
\begin{align*}
E:=\left\{x\in\mathbb{R}:\mathscr{M}^-(\omega)(x)>
\langle\omega\rangle_{A_1^+(\mathbb{R})}\omega(x)\right\}.
\end{align*}
For any $x\in E$, there exist $-\infty<a_x<b_x\leq
c_x<d_x<\infty$, with $b_x-a_x=d_x-c_x=\gamma(d_x-a_x)$, such that
$x\in(c_x,d_x)$, $c_x,d_x\in\mathbb{Q}$, and
\begin{align*}
\fint_{(a_x,b_x)}\omega>\langle\omega\rangle_{A_1^+(\mathbb{R})}\omega(x).
\end{align*}
This, together with \eqref{20250626.1622}, further implies that
$\omega(x)<\mathrm{ess\,inf}_{y\in(c_x,d_x)}\omega(y)$. Therefore,
\begin{align*}
|E|\leq\left|\bigcup_{i\in\mathbb{N}}\left\{x\in(c_i,d_i):c_i,
d_i\in\mathbb{Q}\mbox{\ \ and\ \ }
\omega(x)<\mathop\mathrm{ess\,inf}
_{y\in(c_i,d_i)}\omega(y)\right\}\right|=0
\end{align*}
and hence \eqref{20250614.1712} holds with
$[\omega]^*_{A_1^+(\mathbb{R})}\leq
\langle\omega\rangle_{A_1^+(\mathbb{R})}$. Therefore, the claim
$\langle\omega\rangle_{A_1^+(\mathbb{R})}=[\omega]^*_{A_1^+(\mathbb{R})}$
holds and hence $[\omega]_{A_1^+(\mathbb{R})}\leq\langle\omega\rangle
_{A_1^+(\mathbb{R})}\leq\frac{1}{\gamma}[\omega]_{A_1^+(\mathbb{R})}$.
This finishes the proof of (i) and hence Lemma \ref{Ap+ time lag}.
\end{proof}

Next, we introduce the one-sided BLO space $\mathrm{BLO}^+(\mathbb{R})$
which is a proper subset of $\mathrm{BMO}^+(\mathbb{R})$.

\begin{definition}\label{BLO+}
The \emph{one-sided $\mathrm{BLO}$ space $\mathrm{BLO}^+(\mathbb{R})$} is
defined to be the set of all $f\in L_{\mathrm{loc}}^1(\mathbb{R})$ such that
\begin{align*}
\|f\|_{\mathrm{BLO}^+(\mathbb{R})}:=
\sup_{I\in\mathcal{I}}\fint_{I^-}\left(f-\mathop\mathrm{ess\,inf}_{I^+}
f\right)_+<\infty.
\end{align*}
If the condition above holds with the orientation of the real
line reversed, then $f\in\mathrm{BLO}^-(\mathbb{R})$ (the other
\emph{one-sided $\mathrm{BLO}$ space}).
\end{definition}

Recall that \eqref{20250727.1611} fails to hold for the endpoint
case $q=1$. However, we find that all the functions in
$\mathrm{BLO}^+(\mathbb{R})$ are exactly nonnegative multiples of
logarithms of all the weights in $A_1^+(\mathbb{R})$, which is an
analogue of \eqref{20251012.1456} in the one-sided weight case and hence
answer the question mentioned in the introduction.

\begin{theorem}\label{BLO+ and A1+}
There exists a positive constant $B$ such that, for any
$f\in\mathrm{BLO}^+(\mathbb{R})$ and $\epsilon\in(0,\frac{B}
{\|f\|_{\mathrm{BLO}^+(\mathbb{R})}})$, $e^{\epsilon f}\in
A_1^+(\mathbb{R})$. Conversely, for any $\omega\in A_1^+(\mathbb{R})$,
$\ln\omega\in\mathrm{BLO}^+(\mathbb{R})$ and
\begin{align*}
\|\ln\omega\|_{\mathrm{BLO}^+(\mathbb{R})}
\leq\ln\left(1+2[\omega]_{A_1^+(\mathbb{R})}\right).
\end{align*}
In particular,
\begin{align*}
\mathrm{BLO}^+(\mathbb{R})=\left\{\lambda\ln\omega:\omega\in
A_1^+(\mathbb{R})\mbox{\ \ and\ \ }\lambda\in[0,\infty)\right\}.
\end{align*}
\end{theorem}

To show Theorem \ref{BLO+ and A1+}, we first prove the following
one-sided John--Nirenberg inequality for $\mathrm{BLO}^+(\mathbb{R})$ by
using a one-sided Calder\'on--Zygmund decomposition which
is a slight variant of the one used in the proof of
\cite[Lemma 1]{mt(jlms-1994)}.

\begin{lemma}\label{JN BLO+}
Let $f\in\mathrm{BLO}^+(\mathbb{R})$. Then there exist positive constants
$A$ and $B$ such that, for any $I\in\mathcal{I}$ and
$\lambda\in(0,\infty)$,
\begin{align}\label{20250714.2247}
\left|I^-\cap\left\{\left(f-\mathop\mathrm{ess\,inf}_{I^{++}}f
\right)_+>\lambda\right\}\right|\leq Ae^{-\frac{B\lambda}
{\|f\|_{\mathrm{BLO}^+(\mathbb{R})}}}\left|I^-\right|.
\end{align}
\end{lemma}

\begin{proof}
If $\|f\|_{\mathrm{BLO}^+(\mathbb{R})}=0$, then $f(x)\leq f(y)$ for
almost every $-\infty<x\leq y<\infty$. In this case,
\eqref{20250714.2247} holds automatically for any $I\in\mathcal{I}$ and
$\lambda\in(0,\infty)$. Therefore, we only need to consider the case that
$\|f\|_{\mathrm{BLO}^+(\mathbb{R})}\in(0,\infty)$. Without loss of
generality, we may assume that $\|f\|_{\mathrm{BLO}^+(\mathbb{R})}=1$.
Fix $I\in\mathcal{I}$. From Definition \ref{BLO+} and the assumption that
$\|f\|_{\mathrm{BLO}^+(\mathbb{R})}=1$, we infer that
\begin{align}\label{20250715.2123}
\fint_{I^-\cup I^+}\left(f-\mathop\mathrm{ess\,inf}_{I^{++}}f
\right)_+&\leq\frac{1}{2}\fint_{I^-}
\left(f-\mathop\mathrm{ess\,inf}_{I^+}f\right)_++
\frac{1}{2}\fint_{I^-}\left(\mathop\mathrm{ess\,inf}_{I^+}f
-\mathop\mathrm{ess\,inf}_{I^{++}}f\right)_+\notag\\
&\quad+\frac{1}{2}\fint_{I^+}\left(f-\mathop\mathrm{ess\,inf}
_{I^{++}}f\right)_+\notag\\
&\leq\|f\|_{\mathrm{BLO}^+(\mathbb{R})}+\frac{1}{2}\left(f_{I^+}-
\mathop\mathrm{ess\,inf}_{I^{++}}f\right)_+\notag\\
&\leq\|f\|_{\mathrm{BLO}^+(\mathbb{R})}+\frac{1}{2}\fint_{I^+}
\left(f-\mathop\mathrm{ess\,inf}_{I^{++}}f\right)_+\leq\frac{3}{2}
\|f\|_{\mathrm{BLO}^+(\mathbb{R})}=\frac{3}{2}.
\end{align}
Let
\begin{align*}
\Omega:=\left\{x\in\mathbb{R}:M^-\left(\left[f-\mathop\mathrm{ess\,inf}
_{I^{++}}f\right]_+\boldsymbol{1}_{I^-\cup I^+}\right)(x)>6\right\}.
\end{align*}
Since $M^-$ is lower semicontinuous, it follows that
$\Omega\subset\mathbb{R}$ is open and hence there exists a sequence of
pairwise disjoint open intervals $\{(a_i,b_i)\}_{i\in\mathbb{N}}$ such
that $\Omega=\bigcup_{i\in\mathbb{N}}(a_i,b_i)$. Applying the fact that
$b_i\notin\Omega$ for any $i\in\mathbb{N}$ and \cite[Lemma
2.1]{s(tams-1986)}, we conclude that, for any $i\in\mathbb{N}$ and
$x\in(a_i,b_i)$,
\begin{align}\label{20250715.2107}
\fint_{(x,b_i)}\left(f-\mathop\mathrm{ess\,inf}
_{I^{++}}f\right)_+\boldsymbol{1}_{I^-\cup I^+}\leq6\leq
\fint_{(a_i,x)}\left(f-\mathop\mathrm{ess\,inf}
_{I^{++}}f\right)_+\boldsymbol{1}_{I^-\cup I^+}.
\end{align}
Collect all $(a_i,b_i)$ such that $(a_i,b_i)\cap I^-\neq\emptyset$ and
denote them by $\{J_i\}_{i\in\mathbb{N}}$. We claim that,
for any $i\in\mathbb{N}$, $J_i\cap I^{++}=\emptyset$. Indeed, if there
exists $i_0\in\mathbb{N}$ such that $J_{i_0}\cap I^{++}\neq\emptyset$,
then $I^+\subset J_{i_0}$. This, together with \eqref{20250715.2107} and
\eqref{20250715.2123}, further implies that
\begin{align*}
\left|I^+\right|\leq\left|J_{i_0}\right|=\frac{1}{6}
\int_{J_{i_0}}\left(f-\mathop\mathrm{ess\,inf}_{I^{++}}f\right)
_+\boldsymbol{1}_{I^-\cup I^+}\leq\frac{1}{6}
\int_{I^-\cup  I^+}\left(f-\mathop\mathrm{ess\,inf}_{I^{++}}f\right)
_+\leq\frac{1}{2}\left|I^+\right|,
\end{align*}
which is impossible. Therefore, for any $i\in\mathbb{N}$, $J_i\cap
I^{++}=\emptyset$. This finishes the proof of the above claim.

Now, we show that
\begin{align}\label{20250715.2226}
\sum_{i\in\mathbb{N}}|J_i|\leq\frac{1}{2}\left|I^-\right|.
\end{align}
Indeed, from the definitions of both $M^-$ and $\Omega$, it follows that
$\bigcup_{i\in\mathbb{N}}J_i\subset I^-\cup I^+$. This, together with
\eqref{20250715.2107}, the fact that $\{J_i\}_{i\in\mathbb{N}}$ are
pairwise disjoint, and \eqref{20250715.2123}, further implies that
\begin{align*}
\sum_{i\in\mathbb{N}}|J_i|=\frac{1}{6}\sum_{i\in\mathbb{N}}\int_{J_i}
\left(f-\mathop\mathrm{ess\,inf}_{I^{++}}f\right)_+
\boldsymbol{1}_{I^-\cup I^+}\leq\frac{1}{6}\int_{I^-\cup I^+}
\left(f-\mathop\mathrm{ess\,inf}_{I^{++}}f\right)_+
\leq\frac{1}{2}\left|I^-\right|
\end{align*}
and hence \eqref{20250715.2226} holds.

For any $i\in\mathbb{N}$, let $J_i=:(c_i,d_i)$ and define
$\{x_i^{(k)}\}_{k\in\mathbb{N}},\{y_i^{(k)}\}_{k\in\mathbb{N}}\subset
\mathbb{R}$ by setting, for any $k\in\mathbb{N}$,
\begin{align}\label{20250715.2143}
d_i-x_i^{(k)}=2\left[d_i-y_i^{(k)}\right]
=\left(\frac{2}{3}\right)^k(d_i-c_i)=\left(\frac{2}{3}\right)^k|J_i|.
\end{align}
In addition, for any $i,k\in\mathbb{N}$, let $J_i^{(k)}\in\mathcal{I}$
be such that $J_i^{(k)+}=(x_i^{(k)},y_i^{(k)})$. From \eqref{20250715.2143}
and the fact that $b_i\notin\Omega$ for any $i\in\mathbb{N}$, we deduce
that, for any $i\in\mathbb{N}$, $J_i=\bigcup_{k\in\mathbb{N}}J_i^{(k)-}$
and, for any $k\in\mathbb{N}$,
\begin{align}\label{20250715.2211}
\fint_{J_i^{(k)++}}\left(f-\mathop\mathrm{ess\,inf}_{I^{++}}f
\right)_+=\fint_{(y_k,d_i)}\left(f-\mathop\mathrm{ess\,inf}_{I^{++}}
f\right)_+\leq6.
\end{align}

Next, we are ready to prove \eqref{20250714.2247}. Using the Lebesgue
differentiation theorem and the definition of $\Omega$,
we find that, for almost every $x\in I^-\setminus\Omega=
I^-\setminus\bigcup_{i\in\mathbb{N}}J_i=
I^-\setminus\bigcup_{i,k\in\mathbb{N}}J_i^{(k)-}$,
\begin{align*}
\left[f(x)-\mathop\mathrm{ess\,inf}_{I^{++}}f\right]_+\leq
M^-\left(\left[f-\mathop\mathrm{ess\,inf}_{I^{++}}f\right]_+
\boldsymbol{1}_{I^-\cup I^+}\right)(x)\leq6.
\end{align*}
From this, the fact that $J_i=\bigcup_{k\in\mathbb{N}}J_i^{(k)-}$
for any $i\in\mathbb{N}$, and \eqref{20250715.2211}, we infer that, for
any $\lambda\in(6,\infty)$,
\begin{align}\label{20250715.2214}
\left|I^-\cap\left\{\left(f-\mathop\mathrm{ess\,inf}_{I^{++}}f\right)
_+>\lambda\right\}\right|&\leq\sum_{i,k\in\mathbb{N}}
\left|J_i^{(k)-}\cap\left\{\left[f-\mathop\mathrm{ess\,inf}
_{J_i^{(k)++}}f\right]_+>\lambda-6\right\}\right|\notag\\
&\quad+\sum_{i,k\in\mathbb{N}}\left|J_i^{(k)-}\cap\left\{
\left[\mathop\mathrm{ess\,inf}_{J_i^{(k)++}}f-
\mathop\mathrm{ess\,inf}_{I^{++}}f\right]_+>6\right\}\right|\notag\\
&\leq\sum_{i,k\in\mathbb{N}}\left|J_i^{(k)-}\cap\left\{
\left[f-\mathop\mathrm{ess\,inf}_{J_i^{(k)++}}f\right]_+
>\lambda-6\right\}\right|\notag\\
&\quad+\sum_{i,k\in\mathbb{N}}\left|J_i^{(k)-}\cap
\left\{\fint_{J_i^{(k)++}}\left(f-\mathop\mathrm{ess\,inf}_{I^{++}}f
\right)_+>6\right\}\right|\notag\\
&\leq\sum_{i,k\in\mathbb{N}}\left|J_i^{(k)-}\cap\left\{
\left[f-\mathop\mathrm{ess\,inf}_{J_i^{(k)++}}f\right]_+
>\lambda-6\right\}\right|.
\end{align}
To proceed, for any $\delta\in(0,\infty)$, let
\begin{align*}
\mathcal{A}(\delta):=\sup_{J\in\mathcal{I}}\frac{|J^-\cap
\{(f-\mathop\mathrm{ess\,inf}\limits_{J^{++}}f)_+>\delta\}|}{|J^-|}.
\end{align*}
Combining this, \eqref{20250715.2214}, the facts that $J_i^{(k)-}$ are
pairwise disjoint for any $i,k\in\mathbb{N}$ and
$J_i=\bigcup_{k\in\mathbb{N}}J_i^{(k)-}$ for any $i\in\mathbb{N}$, and
\eqref{20250715.2226}, we obtain
\begin{align*}
\left|I^-\cap\left\{\left(f-\mathop\mathrm{ess\,inf}_{I^{++}}f\right)_+
>\lambda\right\}\right|&\leq\sum_{i,k\in\mathbb{N}}
\mathcal{A}(\lambda-6)\left|J_i^{(k)-}\right|\\
&=\sum_{i\in\mathbb{N}}\mathcal{A}(\lambda-6)|J_i|
\leq\frac{1}{2}\mathcal{A}(\lambda-6)\left|I^-\right|
\end{align*}
and hence $\mathcal{A}(\lambda)\leq\frac{1}{2}\mathcal{A}(\lambda-6)$.
Moreover, notice that, for any $\lambda\in(6,\infty)$, there exists $m\in
\mathbb{N}$ such that $6m<\lambda\leq6(m+1)$. This, together with the
proven conclusion that $\mathcal{A}(\lambda)\leq
\frac{1}{2}\mathcal{A}(\lambda-6)$ for any $\lambda\in(6,\infty)$ and the
fact that  $A(\delta)\leq1$ for any $\delta\in(0,\infty)$, further implies that
\begin{align*}
\mathcal{A}(\lambda)\leq\frac{1}{2^m}\mathcal{A}(\lambda-6m)\leq
\frac{1}{2^m}\leq2^{1-\frac{\lambda}{6}}=2e^{-\frac{\ln2}{6}\lambda}
=2e^{-\frac{\ln 2}{6}\frac{\lambda}{\|f\|_{\mathrm{BLO}^+(\mathbb{R})}}}.
\end{align*}
and hence \eqref{20250714.2247} holds for any $\lambda\in(6,\infty)$.

Finally, for any $\lambda\in(0,6]$,
\begin{align*}
\left|I^-\cap\left\{\left(f-\mathop\mathrm{ess\,inf}_{I^{++}}f\right)
_+>\lambda\right\}\right|\leq\left|I^-\right|
\leq e^{\ln2-\frac{\ln2}{6}\lambda}\left|I^-\right|\leq
2e^{-\frac{\ln 2}{6}\frac{\lambda}{\|f\|_{\mathrm{BLO}^+(\mathbb{R})}}}
\left|I^-\right|.
\end{align*}
In conclusion, for any $I\in\mathcal{I}$ and $\lambda\in(0,\infty)$,
\eqref{20250714.2247} holds with $A:=2$ and $B:=\frac{\ln2}{6}$.
This finishes the proof of Lemma \ref{JN BLO+}.
\end{proof}

Now, we are ready to show Theorem \ref{BLO+ and A1+}.

\begin{proof}[Proof of Theorem \ref{BLO+ and A1+}]
Let $f\in\mathrm{BLO}^+(\mathbb{R})$. From Lemma \ref{JN BLO+}, we deduce
that there exist positive constants $A$ and $B$ such that,
for any $I\in\mathcal{I}$ and $\lambda\in(0,\infty)$,
\eqref{20250714.2247} holds. This, together with Cavalieri's principle,
further implies that, for any $\epsilon\in(0,\frac{B}
{\|f\|_{\mathrm{BLO}^+(\mathbb{R})}})$ and $I\in\mathcal{I}$,
\begin{align*}
\fint_{I^-}e^{\epsilon f}\leq
e^{\epsilon\mathop\mathrm{ess\,inf}\limits_{I^{++}}f}
\fint_{I^-}e^{\epsilon(f-\mathop\mathrm{ess\,inf}\limits_{I^{++}}f)_+}
\leq\left[1+\frac{A\epsilon\|f\|_{\mathrm{BLO}^+(\mathbb{R})}}
{B-\epsilon\|f\|_{\mathrm{BLO}^+(\mathbb{R})}}\right]
\mathop\mathrm{ess\,inf}_{I^{++}}e^{\epsilon f}.
\end{align*}
By this and Lemma \ref{Ap+ time lag} with $\gamma:=\frac{1}{3}$ therein,
we conclude that, for any $\epsilon\in(0,\frac{B}
{\|f\|_{\mathrm{BLO}^+(\mathbb{R})}})$,
$e^{\epsilon f}\in A_1^+(\mathbb{R})$ and
\begin{align}\label{20250622.2236}
\left[e^{\epsilon f}\right]_{A_1^+(\mathbb{R})}\leq
1+\frac{A\epsilon\|f\|_{\mathrm{BLO}^+(\mathbb{R})}}
{B-\epsilon\|f\|_{\mathrm{BLO}^+(\mathbb{R})}}.
\end{align}

Conversely, let $\omega\in A_1^+(\mathbb{R})$. From Jensen's inequality
and Lemma \ref{Ap+ time lag}(i) with $\gamma:=\frac{1}{2}$ therein,
we infer that, for any $I\in\mathcal{I}$,
\begin{align}\label{20250622.2254}
\exp\left\{\fint_{I^-}\left(\ln\omega-\mathop\mathrm{ess\,inf}
_{I^+}\ln\omega\right)_+\right\}&\leq
\fint_{I^-}\exp\left\{\left(\ln\omega-
\mathop\mathrm{ess\,inf}_{I^+}\ln\omega\right)_+\right\}\notag\\
&\leq1+\fint_{I^-}\omega\left(\mathop\mathrm{ess\,inf}
_{I^+}\omega\right)^{-1}\leq1+2[\omega]_{A_1^+(\mathbb{R})}.
\end{align}
Taking the supremum over all $I\in\mathcal{I}$, we obtain
$\ln\omega\in\mathrm{BLO}^+(\mathbb{R})$ and
\begin{align}\label{20250727.1548}
\|\ln\omega\|_{\mathrm{BLO}^+(\mathbb{R})}\leq
\ln\left(1+2[\omega]_{A_1^+(\mathbb{R})}\right).
\end{align}
This finishes the proof of Theorem \ref{BLO+ and A1+}.
\end{proof}

\begin{remark}\label{Linfty BLO+ BMO+}
We point out that $L^\infty(\mathbb{R})\subsetneqq
\mathrm{BLO}^+(\mathbb{R})\subsetneqq\mathrm{BMO}^+(\mathbb{R})$. Indeed,
from some simple calculations, we deduce that
$L^\infty(\mathbb{R})\subset\mathrm{BLO}^+(\mathbb{R})\subset
\mathrm{BMO}^+(\mathbb{R})$ and, for any $f\in L^\infty(\mathbb{R})$,
$\|f\|_{\mathrm{BMO}^+(\mathbb{R})}\leq\|f\|_{\mathrm{BLO}^+(\mathbb{R})}
\leq2\|f\|_{L^\infty(\mathbb{R})}$. Notice that, for any $f\in
L_{\mathrm{loc}}^1(\mathbb{R})$, $\|f\|_{\mathrm{BLO}^+(\mathbb{R})}=0$
if and only if $f(x)\leq f(y)$ for almost every $x,y\in\mathbb{R}$ with
$x\leq y$. Thus, any unbounded monotonically increasing function on
$\mathbb{R}$ belongs to $\mathrm{BLO}^+(\mathbb{R})\setminus
L^\infty(\mathbb{R})$. We prove $[\mathrm{BMO}^+(\mathbb{R})\setminus
\mathrm{BLO}^+(\mathbb{R})]\neq\emptyset$ by contradiction. If
$\mathrm{BMO}^+(\mathbb{R})=\mathrm{BLO}^+(\mathbb{R})$, then
$\mathrm{BMO}^-(\mathbb{R})=\mathrm{BLO}^-(\mathbb{R})$ and hence
\begin{align*}
\mathrm{BMO}(\mathbb{R})=\mathrm{BMO}^+(\mathbb{R})\cap
\mathrm{BMO}^-(\mathbb{R})=\mathrm{BLO}^+(\mathbb{R})\cap
\mathrm{BLO}^-(\mathbb{R}).
\end{align*}
This, together with Theorem \ref{BLO+ and A1+}, further implies that, for
any $f\in\mathrm{BMO}(\mathbb{R})$, there exists $\epsilon\in(0,\infty)$
such that $e^{\epsilon f}\in A_1^+(\mathbb{R})\cap A_1^-(\mathbb{R})=
A_1(\mathbb{R})$. By this and \cite[Lemma 1]{cr(pams-1980)}, we find that
$f\in\mathrm{BLO}(\mathbb{R})$ and hence $\mathrm{BMO}(\mathbb{R})=
\mathrm{BLO}(\mathbb{R})$. This contradicts the classical result that
$\mathrm{BLO}(\mathbb{R})$ is not a linear space (see, for instance,
\cite[p.\,553]{b(pams-1982)}). Thus, $\mathrm{BLO}^+(\mathbb{R})
\subsetneqq\mathrm{BMO}^+(\mathbb{R})$. This further implies that
\eqref{20250727.1611} fails to hold for the endpoint case $q=1$.
\end{remark}

With the aid of Lemma \ref{Ap+ time lag}(i) and Theorem
\ref{BLO+ and A1+}, we obtain the following characterizations of
$\mathrm{BLO}^+(\mathbb{R})$ in terms of the one-sided John--Nirenberg inequality.

\begin{theorem}\label{JN BLO+ 2}
Let $f\in L_{\mathrm{loc}}^1(\mathbb{R})$. Then the following statements
are mutually equivalent.
\begin{enumerate}
\item[\rm(i)] $f\in\mathrm{BLO}^+(\mathbb{R})$.

\item[\rm(ii)] There exist positive constants $A$ and $B$ such that,
for any $I\in\mathcal{I}$ and $\lambda\in(0,\infty)$,
\begin{align*}
\left|I^-\cap\left\{\left(f-\mathop\mathrm{ess\,inf}_{I^+}f
\right)_+>\lambda\right\}\right|\leq Ae^{-B\lambda}\left|I^-\right|.
\end{align*}

\item[\rm(iii)] For any $q\in[1,\infty)$,
\begin{align*}
\|f\|_{\mathrm{BLO}^+(\mathbb{R}),q}:=
\sup_{I\in\mathcal{I}}\left\{\fint_{I^-}
\left(f-\mathop\mathrm{ess\,inf}_{I^+}f\right)_+^q\right\}
^\frac{1}{q}<\infty.
\end{align*}
\end{enumerate}
Moreover, for any $q\in(1,\infty)$,
$\|f\|_{\mathrm{BLO}^+(\mathbb{R})}\sim
\|f\|_{\mathrm{BLO}^+(\mathbb{R}),q}$,
where the positive equivalence constants are independent of $f$.
\end{theorem}

\begin{proof}
The proof that (iii) $\Longrightarrow$ (i) is obvious and the proof that
(ii) $\Longrightarrow$ (iii) follows from Cavalieri's principle. We only
need to show (i) $\Longrightarrow$ (ii). Indeed, let $B$ be the same as
in \eqref{20250714.2247}. By Lemma \ref{Ap+ time lag}(i) with
$\gamma:=\frac{1}{2}$ therein and by \eqref{20250622.2236}, we find that,
for any $I\in\mathcal{I}$,
\begin{align*}
\fint_{I^-}e^\frac{B(f-\mathop\mathrm{ess\,inf}\limits_{I^+}f)_+}
{2\|f\|_{\mathrm{BLO}^+(\mathbb{R})}}&\leq1+\fint_{I^-}
e^\frac{B(f-\mathop\mathrm{ess\,inf}\limits_{I^+}f)}
{2\|f\|_{\mathrm{BLO}^+(\mathbb{R})}}
=1+\fint_{I^-}e^\frac{Bf}{2\|f\|_{\mathrm{BLO}^+(\mathbb{R})}}
\left[\mathop\mathrm{ess\,inf}_{I^+}
e^\frac{Bf}{2\|f\|_{\mathrm{BLO}^+(\mathbb{R})}}\right]^{-1}\\
&\leq1+2\left[e^\frac{Bf}{2\|f\|_{\mathrm{BLO}^+(\mathbb{R})}}\right]
_{A_1^+(\mathbb{R})}\leq3+2A,
\end{align*}
where $A$ is the same as in \eqref{20250714.2247}. This, together with
Chebyshev's inequality, further implies that, for any $I\in\mathcal{I}$
and $\lambda\in(0,\infty)$,
\begin{align*}
\left|I^-\cap\left\{\left(f-\mathop\mathrm{ess\,inf}_{I^+}f
\right)_+>\lambda\right\}\right|&=\left|I^-\cap
\left\{e^\frac{B(f-\mathop\mathrm{ess\,inf}\limits_{I^+}f)_+}
{2\|f\|_{\mathrm{BLO}^+(\mathbb{R})}}>e^\frac{B\lambda}
{2\|f\|_{\mathrm{BLO}^+(\mathbb{R})}}\right\}\right|\\
&\leq e^{-\frac{B\lambda}{2\|f\|_{\mathrm{BLO}^+(\mathbb{R})}}}\int_{I^-}
e^\frac{B(f-\mathop\mathrm{ess\,inf}\limits_{I^+}f)_+}
{2\|f\|_{\mathrm{BLO}^+(\mathbb{R})}}\\
&\leq(3+2A)e^{-\frac{B\lambda}{2\|f\|_{\mathrm{BLO}^+(\mathbb{R})}}}
\left|I^-\right|.
\end{align*}
This finishes the proof that (i) $\Longrightarrow$ (ii) and hence
the proof of Theorem \ref{JN BLO+ 2}.
\end{proof}

At the end of this section, we prove that one can take two intervals that
are not contiguous in the definition of $\mathrm{BLO}^+(\mathbb{R})$.
This result plays a crucial role in the proof of Proposition
\ref{N-:BMO+ to BLO+}.

\begin{theorem}\label{BLO+ time lag}
Let $\gamma\in(0,\frac{1}{2}]$. Then $f\in\mathrm{BLO}^+(\mathbb{R})$ if
and only if $f\in L_{\mathrm{loc}}^1(\mathbb{R})$ and
\begin{align*}
\||f\||_{\mathrm{BLO}^+(\mathbb{R})}:=\sup_{\genfrac{}{}{0pt}{}
{-\infty<a<b\leq c<d<\infty}
{b-a=d-c=\gamma(d-a)}}\fint_{(a,b)}\left[f-
\mathop\mathrm{ess\,inf}_{(c,d)}f\right]_+<\infty.
\end{align*}
Moreover, $\|f\|_{\mathrm{BLO}^+(\mathbb{R})}\sim
\||f\||_{\mathrm{BLO}^+(\mathbb{R})}$,
where the positive equivalence constants are independent of $f$.
\end{theorem}

To show Theorem \ref{BLO+ time lag}, we observe that
$\mathrm{BMO}^+(\mathbb{R})$ coincides with the one-dimensional parabolic
BMO space with time lag. Moreover, their norms are equivalent.

\begin{lemma}\label{BMO+ time lag}
Let $\gamma\in(0,\frac{1}{2}]$. Then $f\in\mathrm{BMO}^+(\mathbb{R})$ if
and only if $f\in L_{\mathrm{loc}}^1(\mathbb{R})$ and
\begin{align*}
\||f\||_{\mathrm{BMO}^+(\mathbb{R})}:=\sup_{\genfrac{}{}{0pt}{}
{-\infty<a<b\leq c<d<\infty}{b-a=d-c=\gamma(d-a)}}
\inf_{\alpha\in\mathbb{R}}\left[\fint_{(a,b)}(f-\alpha)_++
\fint_{(c,d)}(f-\alpha)_-\right]<\infty.
\end{align*}
Moreover, $\|f\|_{\mathrm{BMO}^+(\mathbb{R})}\sim
\||f\||_{\mathrm{BMO}^+(\mathbb{R})}$,
where the positive equivalence constants are independent of $f$.
\end{lemma}

\begin{proof}
We first prove the necessity. Let $f\in\mathrm{BMO}^+(\mathbb{R})$.
From the proof of the one-sided John--Nirenberg inequality for
$\mathrm{BMO}^+(\mathbb{R})$ (see the proof of \cite[Theorem
3]{mt(jlms-1994)}), we infer that there exist positive constants $\delta$
and $K_1$, independent of $f$, such that $e^\frac{\delta f}
{2\|f\|_{\mathrm{BMO}^+(\mathbb{R})}}\in A_\infty^+(\mathbb{R})$,
$e^{-\frac{\delta f}{2\|f\|_{\mathrm{BMO}^+(\mathbb{R})}}}\in
A_\infty^-(\mathbb{R})$, and
\begin{align*}
\max\left\{\left[e^\frac{\delta f}
{2\|f\|_{\mathrm{BMO}^+(\mathbb{R})}}\right]_{A_\infty^+(\mathbb{R})},
\,\left[e^{-\frac{\delta f}{2\|f\|_{\mathrm{BMO}^+(\mathbb{R})}}}
\right]_{A_\infty^-(\mathbb{R})}\right\}\leq K_1.
\end{align*}
This, together with Lemma \ref{Ap+ and Ainfty +},
further implies that $e^\frac{\delta f}
{2\|f\|_{\mathrm{BMO}^+(\mathbb{R})}}\in A_2^+(\mathbb{R})$
and $[e^\frac{\delta f}{2\|f\|_{\mathrm{BMO}^+(\mathbb{R})}}]
_{A_2^+(\mathbb{R})}\leq K_1^2$.
By this and the Jones factorization theorem for one-sided Muckenhoupt
weights (see, for instance, \cite[Remarks (B)]{s(tams-1986)}), we
conclude that there exist a positive constant $K_2$ (independent of $f$),
$u\in A_1^+(\mathbb{R})$, and $v\in A_1^-(\mathbb{R})$ such that
\begin{align}\label{20250624.1557}
\max\left\{[u]_{A_1^+(\mathbb{R})},\,[v]_{A_1^-(\mathbb{R})}\right\}
\leq K_2
\end{align}
and $e^\frac{\delta f}{2\|f\|_{\mathrm{BMO}^+(\mathbb{R})}}=uv^{-1}$.
From \eqref{20250624.1557}, Jensen's inequality, and Lemma
\ref{Ap+ time lag}(i), we deduce that, for any $-\infty<a<b\leq
c<d<\infty$ with $b-a=d-c=\gamma(d-a)$,
\begin{align*}
\exp\left\{\fint_{(c,d)}(\alpha-\ln u)_+\right\}&\leq\fint_{(c,d)}
e^{(\alpha-\ln u)_+}\leq1+\fint_{(c,d)}e^{\alpha-\ln u}\\
&=1+\fint_{(c,d)}\frac{\mathop\mathrm{ess\,inf}\limits_{(c,d)}u}{u}\leq2
\end{align*}
and
\begin{align*}
\exp\left\{\fint_{(a,b)}(\ln u-\alpha)_+\right\}&\leq\fint_{(a,b)}
e^{(\ln u-\alpha)_+}\leq1+\fint_{(a,b)}e^{\ln u-\alpha}\\
&=1+\fint_{(a,b)}u\left[\mathop\mathrm{ess\,inf}_{(c,d)}u
\right]^{-1}\leq1+\frac{K_2}{\gamma},
\end{align*}
where $a:=\ln(\mathrm{ess\,inf}_{(c,d)}u)$.
This further implies that
\begin{align}\label{20250624.1611}
\||\ln u\||_{\mathrm{BMO}^+(\mathbb{R})}\leq
\ln\left(2\left[1+\frac{K_2}{\gamma}\right]\right).
\end{align}
By an argument similar to that used in
\eqref{20250624.1611}, we obtain
\begin{align*}
\||\ln v\||_{\mathrm{BMO}^+(\mathbb{R})}
\leq\ln\left(2\left[1+\frac{K_2}{\gamma}\right]\right).
\end{align*}
Combining this, \eqref{20250624.1611}, and the proven conclusion
that $e^\frac{\delta f}{2\|f\|_{\mathrm{BMO}^+(\mathbb{R})}}=uv^{-1}$,
we conclude that
\begin{align*}
\||f\||_{\mathrm{BMO}^+(\mathbb{R})}&\leq\frac{2}{\delta}
\left[\||\ln u\||_{\mathrm{BMO}^+(\mathbb{R})}+
\||\ln v\||_{\mathrm{BMO}^+(\mathbb{R})}\right]
\|f\|_{\mathrm{BMO}^+(\mathbb{R})}\\
&\leq\frac{4}{\delta}\ln\left(2\left[1+\frac{K_2}{\gamma}\right]\right)
\|f\|_{\mathrm{BMO}^+(\mathbb{R})},
\end{align*}
thereby completing the proof of the necessity.

Next, we show the sufficiency. Without loss of generality, we may
assume that $\gamma<\frac{1}{2}$. Suppose that
$\||f\||_{\mathrm{BMO}^+(\mathbb{R})}<\infty$. For any $x,t\in\mathbb{R}$,
define $F(x,t):=f(t)$. Then $F$ belongs to the two-dimensional parabolic
BMO space $\mathrm{PBMO}_{1-\gamma}^-(\mathbb{R}^2)$ with time lag and
\begin{align*}
\|F\|_{\mathrm{PBMO}_{1-\gamma}^-(\mathbb{R}^2)}&:=
\sup_{\genfrac{}{}{0pt}{}{-\infty<a<b<\infty}{c\in\mathbb{R}}}
\inf_{\alpha\in\mathbb{R}}\left[\fint_{(a,b)
\times(c,c+\gamma(\frac{b-a}{2})^2)}(f-\alpha)_+\right.\\
&\quad\left.+\fint_{(a,b)\times(c+(2-\gamma)(\frac{b-a}{2})^2,
c+2(\frac{b-a}{2})^2)}(f-\alpha)_-\right]\\
&\ =\||f\||_{\mathrm{BMO}^+(\mathbb{R})}.
\end{align*}
From this and the parabolic John--Nirenberg inequality
for the parabolic BMO space with time lag (see \cite[Theorem
4.1]{kmy(ma-2023)}), we infer that there exist positive constants $D$
and $K$, independent of $f$, such that, for any $-\infty<a<b\leq
c<d<\infty$ with $b-a=d-c=\gamma(d-a)$ and for any
$\lambda\in(0,\infty)$, there exists $\alpha\in\mathbb{R}$ satisfying both
\begin{align*}
|(a,b)\cap\{(f-\alpha)_+>\lambda\}|\leq
De^{-\frac{K\lambda}{\||f\||_{\mathrm{BMO}^+(\mathbb{R})}}}(b-a)
\end{align*}
and
\begin{align*}
|(c,d)\cap\{(f-\alpha)_->\lambda\}|\leq
De^{-\frac{K\lambda}{\||f\||_{\mathrm{BMO}^+(\mathbb{R})}}}(d-c).
\end{align*}
This and Cavalieri's principle yield, for any
$-\infty<a<b\leq c<d<\infty$ with $b-a=d-c=\gamma(d-a)$,
\begin{align*}
\fint_{(a,b)}e^\frac{Kf}{2\||f\||_{\mathrm{BMO}^+(\mathbb{R})}}
\fint_{(c,d)}e^{-\frac{Kf}{2\||f\||_{\mathrm{BMO}^+(\mathbb{R})}}}\leq
\fint_{(a,b)}e^\frac{K(f-\alpha)_+}{2\||f\||_{\mathrm{BMO}^+(\mathbb{R})}}
\fint_{(c,d)}e^\frac{K(f-\alpha)_-}{2\||f\||_{\mathrm{BMO}^+(\mathbb{R})}}
\leq(1+D)^2.
\end{align*}
Combining this and Lemma \ref{Ap+ time lag}(ii), we obtain
$e^\frac{Kf}{2\||f\||_{\mathrm{BMO}^+(\mathbb{R})}}\in A_2^+(\mathbb{R})$
and there exits a positive constant $K_3$, independent of $f$, such that
$[e^\frac{Kf}{2\||f\||_{\mathrm{BMO}^+(\mathbb{R})}}]
_{A_2^+(\mathbb{R})}\leq K_3$. This, together with
\cite[Theorem 1]{mt(jlms-1994)}, further implies that
$f\in\mathrm{BMO}^+(\mathbb{R})$ and
$\|f\|_{\mathrm{BMO}^+(\mathbb{R})}\lesssim
\||f\||_{\mathrm{BMO}^+(\mathbb{R})}$.
\end{proof}

Now, we are ready to prove Theorem \ref{BLO+ time lag}.

\begin{proof}[Proof of Theorem \ref{BLO+ time lag}]
We first show the necessity. Assume that
$f\in\mathrm{BLO}^+(\mathbb{R})$. By \eqref{20250622.2236}, we find that
$e^\frac{Bf}{2\|f\|_{\mathrm{BLO}^+(\mathbb{R})}}\in
A_1^+(\mathbb{R})$ and $[e^\frac{Bf}{2\|f\|_{\mathrm{BLO}^+(\mathbb{R})}}]
_{A_1^+(\mathbb{R})}\leq1+A$, where $A$ and $B$ are the same as in
\eqref{20250714.2247}. This, together with
Lemma \ref{Ap+ time lag}(i), further implies that
\begin{align*}
\left\langle e^\frac{Bf}{2\|f\|_{\mathrm{BLO}^+(\mathbb{R})}}\right\rangle
_{A_1^+(\mathbb{R})}:=\sup_{\genfrac{}{}{0pt}{}
{-\infty<a<b\leq c<d<\infty}{b-a=d-c=\gamma(d-a)}}
\fint_{(a,b)}e^\frac{Bf}{2\|f\|_{\mathrm{BLO}^+(\mathbb{R})}}
\left[\mathop\mathrm{ess\,inf}_{(c,d)}e^\frac{Bf}
{2\|f\|_{\mathrm{BLO}^+(\mathbb{R})}}\right]^{-1}
\leq\frac{1+A}{\gamma}.
\end{align*}
Combining this and an argument similar to that used in
\eqref{20250622.2254} with $\omega$, $I^-$, $I^+$, and
$[\omega]_{A_1^+(\mathbb{R})}$ therein replaced, respectively, by
$e^\frac{Bf}{2\|f\|_{\mathrm{BLO}^+(\mathbb{R})}}$, $(a,b)$, $(c,d)$, and
$\langle e^\frac{Bf}{2\|f\|_{\mathrm{BLO}^+(\mathbb{R})}}\rangle
_{A_1^+(\mathbb{R})}$, we obtain
\begin{align*}
\||f\||_{\mathrm{BLO}^+(\mathbb{R})}\leq\frac{2}{B}\ln\left(1+
\frac{2[1+A]}{\gamma}\right)\|f\|_{\mathrm{BLO}^+(\mathbb{R})},
\end{align*}
thereby completing the proof of the necessity.

Next, we prove the sufficiency. Suppose that
$\||f\||_{\mathrm{BLO}^+(\mathbb{R})}<\infty$.
Using Lemma \ref{BMO+ time lag}, we find that
$f\in\mathrm{BMO}^+(\mathbb{R})$ and there exists a positive constant $K$,
independent of $f$, such that
\begin{align*}
\|f\|_{\mathrm{BMO}^+(\mathbb{R})}\leq
K\||f\||_{\mathrm{BMO}^+(\mathbb{R})}\leq
K\||f\||_{\mathrm{BLO}^+(\mathbb{R})}.
\end{align*}
This, together with the John--Nirenberg inequality for
$\mathrm{BMO}^+(\mathbb{R})$ (see \cite[Theorem 3]{mt(jlms-1994)}),
further implies that there exist positive constants $\delta$ and $C$,
both independent of $f$, such that, for any $-\infty<a<b\leq c<d<\infty$
with $b-a=d-c=\frac{\gamma}{1+\gamma}(d-a)$,
\begin{align*}
\fint_{(a,b)}e^\frac{\delta f}{2K\||f\||_{\mathrm{BLO}^+(\mathbb{R})}}\leq
e^\frac{\delta f_{(b,2b-a)}}{2K\||f\||_{\mathrm{BLO}^+(\mathbb{R})}}
\fint_{(a,b)}e^\frac{\delta[f-f_{(b,2b-a)}]_+}
{2K\||f\||_{\mathrm{BLO}^+(\mathbb{R})}}
\leq Ce^\frac{\delta}{2K}\mathop\mathrm{ess\,inf}_{(c,d)}
e^\frac{\delta f}{2K\||f\||_{\mathrm{BLO}^+(\mathbb{R})}}.
\end{align*}
Combining this and Lemma \ref{Ap+ time lag}(i), we obtain
$e^\frac{\delta f}{2K\||f\||_{\mathrm{BLO}^+(\mathbb{R})}}\in
A_1^+(\mathbb{R})$ and $[e^\frac{\delta f}
{2K\||f\||_{\mathrm{BLO}^+(\mathbb{R})}}]_{A_1^+(\mathbb{R})}\leq
Ce^\frac{\delta}{2K}$. From this and an argument similar to that used in
\eqref{20250622.2254} with $\omega$ and
$[\omega]_{A_1^+(\mathbb{R})}$ therein replaced, respectively, by
$e^\frac{\delta f}{2K\||f\||_{\mathrm{BLO}^+(\mathbb{R})}}$ and
$[e^\frac{\delta f}{2K\||f\||_{\mathrm{BLO}^+(\mathbb{R})}}]
_{A_1^+(\mathbb{R})}$, We further deduce that $f\in\mathrm{BLO}^+$ and
\begin{align*}
\|f\|_{\mathrm{BLO}^+(\mathbb{R})}\leq\frac{2K}{\delta}
\ln\left(1+2Ce^\frac{\delta}{2K}\right)
\||f\||_{\mathrm{BLO}^+(\mathbb{R})}.
\end{align*}
This finishes the proof of the sufficiency and hence
Corollary \ref{BLO+ time lag}.
\end{proof}

\begin{remark}
Consider the following $\mathrm{BMO}^+$-norm. Let
$\gamma\in(0,\frac{1}{2})$. For any $f\in
L_{\mathrm{loc}}^1(\mathbb{R})$, define
\begin{align*}
\langle f\rangle_{\mathrm{BMO}^+(\mathbb{R})}:=\sup_{\genfrac{}{}{0pt}{}
{-\infty<a<b\leq c<d<\infty}{b-a=d-c=\gamma(d-a)}}\fint_{(a,b)}
\left[f-f_{(c,d)}\right]_+<\infty.
\end{align*}
Unlike the BLO case, whether $\|\cdot\|_{\mathrm{BMO}^+(\mathbb{R})}\sim
\langle\cdot\rangle_{\mathrm{BMO}^+(\mathbb{R})}$ is still an open problem.
\end{remark}

\section{Bennett Type Characterization of
$\mathrm{BLO}^+(\mathbb{R})$
in Terms of $\mathrm{BMO}^+(\mathbb{R})$}
\label{section3}

The \emph{uncentered one-sided natural maximal operator $N^-$} is defined
by setting, for any $f\in L_{\mathrm{loc}}^1(\mathbb{R})$ and
$x\in\mathbb{R}$,
\begin{align*}
N^-(f)(x):=\sup_{I\in\mathcal{I},\,x\in I^+}\fint_{I^-}f.
\end{align*}
In this section, we first establish the boundedness of $N^-$ from
$\mathrm{BMO}^+(\mathbb{R})$ to $\mathrm{BLO}^+(\mathbb{R})$. Then we in
turn present the Bennett type characterization of
$\mathrm{BLO}^+(\mathbb{R})$ in terms of $\mathrm{BMO}^+(\mathbb{R})$,
which means that $\mathrm{BLO}^+(\mathbb{R})$ coincides precisely with
\begin{align*}
\left\{N^-(F):F\in\mathrm{BMO}^+(\mathbb{R})\mbox{\ \ and\ \ }
N^-(F)<\infty\mbox{\ almost\ everywhere}\right\}
\end{align*}
modulo bounded functions.

We begin with showing the boundedness of $N^-$ from
$\mathrm{BMO}^+(\mathbb{R})$ to $\mathrm{BLO}^+(\mathbb{R})$, for whose
proof we need to use Theorem \ref{BLO+ time lag}, a one-sided
Calder\'on--Zygmund decomposition from the proof of
\cite[Theorem 5.1]{s(ampa-2018)}, and an iterative method which is a
slight variant of the one used in the proof
of \cite[Theorem 5.1]{s(ampa-2018)}.

\begin{proposition}\label{N-:BMO+ to BLO+}
There exists a positive constant $C$ such that, for any
$f\in\mathrm{BMO}^+(\mathbb{R})$ such that $N^-(f)<\infty$ almost
everywhere, $N^-(f)\in\mathrm{BLO}^+(\mathbb{R})$ and
\begin{align*}
\left\|N^-(f)\right\|_{\mathrm{BLO}^+(\mathbb{R})}\leq
C\|f\|_{\mathrm{BMO}^+(\mathbb{R})}.
\end{align*}
\end{proposition}

\begin{proof}
Let $f\in\mathrm{BMO}^+(\mathbb{R})$ be such that $N^-(f)<\infty$ almost
everywhere and $I\in\mathcal{I}$. For any $x\in I^-$, define
\begin{align*}
N_1^-(f)(x):=\sup\left\{\fint_{J^-}f:J\in\mathcal{I},\ x\in J^+,
\mbox{\ and\ }|J|\leq\frac{|I|}{2}\right\}
\end{align*}
and
\begin{align*}
N_2^-(f)(x):=\sup\left\{\fint_{J^-}f:J\in\mathcal{I},\ x\in J^+,
\mbox{\ and\ }|J|\geq\frac{|I|}{2}\right\}.
\end{align*}
From Theorem \ref{BLO+ time lag}, it follows that
we only need to prove that there exists a positive constant $C$,
independent of $f$ and $I$, such that, for any $j\in\{1,\,2\}$,
\begin{align}\label{20250618.1752}
\fint_{I^-}\left[N_j^-(f)-\mathop\mathrm{ess\,inf}_{I^{+++}}N^-(f)
\right]_+\leq C\|f\|_{\mathrm{BMO}^+(\mathbb{R})}.
\end{align}
Here, and thereafter, for any $I:=(a,b)\in\mathcal{I}$,
let $I^{--}:=(a-\frac{b-a}{2},a)$ and $I^{+++}:=(b+\frac{b-a}{2},2b-a)$.

We first show \eqref{20250618.1752} for the case $j=1$. Notice
that $f_{I^+\cup I^{++}}\leq \mathrm{ess\,inf}_{I^{+++}}N^-(f)$. Thus,
\begin{align}\label{20250616.1805}
\fint_{I^-}\left[N_1^-(f)-\mathop\mathrm{ess\,inf}_{I^{+++}}N^-(f)
\right]_+\leq\fint_{I^-}\left[N_1^-(f)-
\fint_{I^+\cup I^{++}}f\right]_+.
\end{align}
To estimate the right hand side of \eqref{20250616.1805}, we
construct a one-sided Calder\'on--Zygmund decomposition of $f$
in $I^{--}\cup I^-$ as follows. Let $\lambda\in(f_{I^+\cup
I^{++}},\infty)$ and
\begin{align*}
\begin{cases}
\displaystyle
g:=f\boldsymbol{1}_{I^{--}\cup
I^-\setminus\bigcup_{i\in\mathbb{N}}I_i^-}+
\sum\limits_{i\in\mathbb{N}}
f_{\hat{I}_i^+}\boldsymbol{1}_{I_i^-},\\
\displaystyle
h_i:=\left(f-f_{\hat{I}_i^+}\right)\boldsymbol{1}_{I_i^-}
\mbox{\ for\ any\ }i\in\mathbb{N},\\
\displaystyle
h:=\sum\limits_{i\in\mathbb{N}}b_i,
\end{cases}
\end{align*}
where $\{I_i^-\}_{i\in\mathbb{N}}$ are the maximal
dyadic subintervals of $I^{--}\cup I^-$ such that, for any
$i\in\mathbb{N}$, $f_{I_i^+}>\lambda$. Moreover, for any
$i\in\mathbb{N}$, $\hat{I}_i^-$ is the dyadic parent of $I_i^-$ and
$\hat{I}_i^+:=\hat{I}_i^-+|\hat{I}_i^-|$. Then
$\{I_i^-\}_{i\in\mathbb{N}}$ are pairwise disjoint,
$f\boldsymbol{1}_{I^{--}\cup I^-}=h+g$, and $g\leq\lambda$ almost
everywhere in $I^{--}\cup I^-$. In addition, for any $x\in I^-$,
\begin{align*}
N_1^-(f)(x)&\leq\sup\left\{\fint_{J^-}(h_++g):
J\in\mathcal{I},\ x\in J^+, \mbox{\ and\ }|J|\leq
\frac{|I|}{2}\right\}\\
&\leq M(h_+)(x)+\lambda,\notag
\end{align*}
where $M$ denotes the Hardy--Littlewood maximal operator, which is defined
by setting, for any $F\in L_{\mathrm{loc}}^1(\mathbb{R})$ and
$x\in\mathbb{R}$, $M(F)(x):=\sup_{I\in\mathcal{I}}\fint_I|F|$.
Combining this and H\"older's inequality, we conclude that
\begin{align}\label{20250618.1641}
\int_{I^-}\left[N_1^-(f)-\fint_{I^+\cup I^{++}}f\right]_+&=
\int_{I^-\cap\{N_1^-(f)>f_{I^+\cup I^{++}}\}}
\left[N_1^-(f)-\fint_{I^+\cup I^{++}}f\right]\notag\\
&\leq\int_{I^-\cap\{N_1^-(f)>f_{I^+\cup I^{++}}\}}
\left[M(h_+)+\lambda-\fint_{I^+\cup I^{++}}f\right]\notag\\
&\leq|I^-|^\frac{1}{2}\|M(h_+)\|_{L^2(\mathbb{R})}+
\left(\lambda-f_{I^+\cup I^{++}}\right)\left|I^-\right|.
\end{align}

Now, we prove that $h_+\in L^2(\mathbb{R})$ and
$\|h_+\|_{L^2(\mathbb{R})}\lesssim|I^-|^\frac{1}{2}
\|f\|_{\mathrm{BMO}^+(\mathbb{R})}$,
where the implicit positive constant is independent of $f$.
Indeed, from \cite[Theorem 3]{mt(jlms-1994)}, we deduce that,
for any $F\in\mathrm{BMO}^+(\mathbb{R})$,
\begin{align*}
\|F\|_{\mathrm{BMO}^+(\mathbb{R})}\sim\sup_{I\in\mathcal{I}}
\left[\fint_{I^-}(F-F_{I^+})_+^2\right]^\frac{1}{2},
\end{align*}
where the positive equivalence constants are independent of $F$. Combining
this and the fact that $\{I_i^-\}_{i\in\mathbb{N}}$ are pairwise disjoint
and $I_i\subset\hat{I}_i^-$ with $|\hat{I}_i^-|=2|I_i^-|$ for any
$i\in\mathbb{N}$, we obtain
\begin{align*}
\|h_+\|_{L^2(\mathbb{R})}^2&=\int_{\mathbb{R}}\left[\sum_{i\in\mathbb{N}}
\left(f-f_{\hat{I}_i^+}\right)_+\boldsymbol{1}_{I_i^-}\right]^2\\
&=\sum_{i\in\mathbb{N}}\int_{I_i^-}
\left(f-f_{\hat{I}_i^+}\right)_+^2\leq\sum_{i\in\mathbb{N}}
\left|\hat{I}_i^-\right|\fint_{\hat{I}_i^-}
\left(f-f_{\hat{I}_i^+}\right)_+^2
\lesssim|I^-|\,\|f\|_{\mathrm{BMO}^+(\mathbb{R})}^2,
\end{align*}
where the implicit positive constant is independent of $f$.
Therefore, $h_+\in L^2(\mathbb{R})$ and
\begin{align*}
\|h_+\|_{L^2(\mathbb{R})}\lesssim|I^-|^
\frac{1}{2}\|f\|_{\mathrm{BMO}^+(\mathbb{R})}.
\end{align*}
Using this, \eqref{20250616.1805}, the fact that $M$ is bounded on
$L^2(\mathbb{R})$, and \eqref{20250618.1641} and letting $\lambda\to
f_{I^+\cup I^{++}}$, we find that \eqref{20250618.1752} holds for the
case $j=1$.

Next, we show \eqref{20250618.1752} for the case $j=2$. Let $x\in I^-$,
$y\in I^{+++}$, $J\in\mathcal{I}$ with $x\in J^+$ and
$|J|\geq\frac{|I|}{2}$, and $J_y:=(y-|J^-|,y+|J^-|)\in\mathcal{I}$. We
first prove that
\begin{align}\label{20250618.1654}
\left(f_{J^-}-f_{J_y^-}\right)_+\lesssim
\|f\|_{\mathrm{BMO}^+(\mathbb{R})},
\end{align}
where the implicit positive constant is independent of $f$, $J$, and
$J_y$. Indeed, let $a,z\in\mathbb{R}$ be such that $J^-:=(a,z)$ and
$b:=y-|J_y^-|$. We consider the following two cases for $z$.

\emph{Case (1)} $z\notin J_y^-$, namely $z\leq b$. In this case, $J^-\cap
J_y^-=\emptyset$. Let $\widetilde{J}^-:=J^-\cup[z,\frac{z+b}{2})$ and
$\widetilde{J}_y^-:=J_y^-\cup(\frac{z+b}{2},b]$. Then
$|J^-|=|J_y^-|\leq|\widetilde{J}_y^-|=|\widetilde{J}^-|$.
Moreover, we claim that $|\widetilde{J}^-|\leq6|J^-|$. Indeed,
if $b\leq x$, then $b-z\leq x-z\leq|J^+|=|J^-|$ and hence
$|\widetilde{J}^-|=|J^-|+\frac{b-z}{2}\leq2|J^-|$. If $b>x$, then
$\frac{|I^-|}{2}\leq|J^-|=|J_y^-|\leq4|I^-|$, which
further implies that $b-z\leq y-z\leq4|I^-|+|J^+|\leq9|J^-|$ and hence
$|\widetilde{J}^-|=|J^-|+\frac{b-z}{2}\leq6|J^-|$. In conclusion,
$|\widetilde{J}^-|\leq6|J^-|$. Combining this, the definitions of both
$\widetilde{J}^-$ and $\widetilde{J}_y^-$, and the fact that,
for any $F\in\mathrm{BMO}^+(\mathbb{R})$,
\begin{align*}
\|F\|_{\mathrm{BMO}^+(\mathbb{R})}\sim\sup_{I\in\mathcal{I}}
\fint_{I^-}\fint_{I^+}[F(t)-F(s)]_+\,ds\,dt,
\end{align*}
where the positive equivalence constants are independent of $F$ (see, for
instance, \cite[p.\,1493]{s(ampa-2018)}), we conclude that
\begin{align}\label{20250618.1726}
\left(f_{J^-}-f_{J_y^-}\right)_+&\leq\fint_{J^-}\fint_{J_y^-}[f(t)-f(s)]_+
\,ds\,dt\notag\\
&\leq\left(\frac{|\widetilde{J}^-|}{|J^-|}\right)^2\fint_{\widetilde{J}^-}
\fint_{\widetilde{J}_y^-}[f(t)-f(s)]_+\,ds\,dt
\lesssim\|f\|_{\mathrm{BMO}^+(\mathbb{R})},
\end{align}
which hence completes the proof of \eqref{20250618.1654} in this case.

\emph{Case (2)} $z\in J_y^-$, namely $z>b$. In this case,
$J^-\cap J_y^-\neq\emptyset$. Let $k\in\mathbb{N}$ be such that
\begin{align}\label{20250618.1727}
z-(k-1)(b-a)>b,\ z-k(b-a)\in[a,b],\mbox{\ and\ }z-(k+1)(b-a)<a.
\end{align}
If $k=1$, then, from \eqref{20250618.1727} and the fact that $y-z=b-a$, we
easily deduce that
\begin{align}\label{20250618.1732}
\left|J^-\cap J_y^-\right|=
z-b\leq\frac{1}{2}(z-a)=\frac{1}{2}\left|J^-\right|.
\end{align}
Divide $J^-$ into two intervals of equal length, denoted as $J_1^-$ and
$J_2^-$. Analogously, divide $J_y^-$ into two intervals of equal length,
denoted as $J_{y,1}^-$ and $J_{y,2}^-$. Using \eqref{20250618.1732}, we
obtain, for any $j\in\{1,2\}$, $J_j^-\cap J_{y,j}=\emptyset$. This,
together with an argument similar to that used in \eqref{20250618.1726}
with $J^-$ and $J_y^-$ therein replaced, respectively, by
$J_1^-$ and $J_{y,1}^-$ firstly and then by $J_2^-$ and $J_{y,2}^-$,
further implies that
\begin{align}\label{20250618.1734}
\left(f_{J^-}-f_{J_y^-}\right)_+&=2\left(f_{J_1^-}+f_{J_2^-}-
f_{J_{y,1}^-}-f_{J_{y,2}^-}\right)_+\notag\\
&\leq2\left(f_{J_1^-}-f_{J_{y,1}^-}\right)_++
\left(f_{J_2^-}-f_{J_{y,2}^-}\right)_+\lesssim
\|f\|_{\mathrm{BMO}^+(\mathbb{R})}
\end{align}
and hence \eqref{20250618.1654} holds in this case.

If $k>1$, then, let $d:=b-a=y-z$ and, for any $j\in\mathbb{Z}_+\cap[0,k]$,
define $J_j^-:=(z-jd,z-[j-1]d)$. Moreover, define $\widetilde{J}_k^-:=(a,z-kd]$
and $\widetilde{J}_{k-1}^-:=(b,z-(k-1)d]$. Then
\begin{align}\label{20250620.1626}
\left(f_{J^-}-f_{J_y^-}\right)_+&=\frac{1}{|J^-|}\left[\sum_{j=1}^{k-1}
\int_{J_j^-}f+\int_{\widetilde{J}_k^-}f-\sum_{j=1}^{k-1}\int_{J_{j-1}^-}f
-\int_{\widetilde{J}_{k-1}^-}f\right]_+\notag\\
&\leq\frac{|\widetilde{J}_k^-|}{|J^-|}\left(f_{\widetilde{J}_k^-}-
f_{\widetilde{J}_{k-1}^-}\right)+\sum_{j=1}^{k-1}\frac{|J_j^-|}{|J^-|}
\left(f_{J_j^-}-f_{J_{j-1}^-}\right)\notag\\
&\leq\frac{|\widetilde{J}_k^-|}{|J^-|}\left(f_{\widetilde{J}_k^-}-
f_{\widetilde{J}_{k-1}^-}\right)+\frac{|J^-|-|\widetilde{J}_k^-|}{|J^-|}
\|f\|_{\mathrm{BMO}^+(\mathbb{R})}.
\end{align}
We then consider the following two cases for
$|\widetilde{J}_k^-\cap\widetilde{J}_{k-1}^-|$.

\emph{Case (2.1)} $|\widetilde{J}_k^-\cap\widetilde{J}_{k-1}^-|\leq
\frac{1}{2}|\widetilde{J}_k^-|$. In this case, by an argument similar to
that used in \eqref{20250618.1734} with $J^-$ and $J_y^-$
therein replaced, respectively, by $\widetilde{J}_k^-$ and
$\widetilde{J}_{k-1}^-$, we conclude that
\begin{align*}
\left(f_{\widetilde{J}_k^-}-f_{\widetilde{J}_{k-1}^-}\right)\lesssim
\|f\|_{\mathrm{BMO}^+(\mathbb{R})}.
\end{align*}
This, together with \eqref{20250620.1626}, further implies that
\eqref{20250618.1654} holds in this case.

\emph{Case (2.2)} $|\widetilde{J}_k^-\cap\widetilde{J}_{k-1}^-|>
\frac{1}{2}|\widetilde{J}_k^-|$. In this case, we adopt an iterative
approach adapted from \cite{s(ampa-2018)}. Let
$J_{\mathrm{left}}^{(1)}:=J^-$, $J_{\mathrm{right}}^{(1)}:=J_y^-$,
$J_{\mathrm{left}}^{(2)}:=\widetilde{J}_k^-$,
and $J_{\mathrm{right}}^{(2)}:=\widetilde{J}_{k-1}^-$. Then we can rewrite
\eqref{20250620.1626} as
\begin{align}\label{20250620.1648}
\left|J_{\mathrm{left}}^{(1)}\right|\left[f_{J_{\mathrm{left}}^{(1)}}
-f_{J_{\mathrm{right}}^{(1)}}\right]_+\leq
\left|J_{\mathrm{left}}^{(2)}\right|
\left[f_{J_{\mathrm{left}}^{(2)}}-f_{J_{\mathrm{right}}^{(2)}}\right]_++
\left[\left|J_{\mathrm{left}}^{(1)}\right|-
\left|J_{\mathrm{left}}^{(2)}\right|\right]
\|f\|_{\mathrm{BMO}^+(\mathbb{R})}.
\end{align}
For any $i\in\mathbb{N}\cap[2,\infty]$ such that
$|J_{\mathrm{left}}^{(i)}\cap
J_{\mathrm{right}}^{(i)}|>\frac{1}{2}|J_{\mathrm{left}}^{(i)}|$, let
$-\infty<a_i<b_i<z_i<y_i<\infty$ be such that
$J_{\mathrm{left}}^{(i)}=(a_i,z_i)$
and $J_{\mathrm{right}}^{(i)}=(b_i,y_i)$ and let
$k_i\in\mathbb{N}\cap[2,\infty)$ satisfy
\begin{align*}
z_i-k_i(b_i-a_i)>b_i,\ z_i-k_i(b_i-a_i)\in[a_i,b_i],\mbox{\ and\ }
z_i-(k_i+1)(b_i-a_i)<a_i.
\end{align*}
In addition, let
$d_i:=|J_{\mathrm{left}}^{(i)}|-|J_{\mathrm{left}}^{(i)}\cap
J_{\mathrm{right}}^{(i)}|=y_i-z_i=b_i-a_i$. For any
$j\in\mathbb{N}\cap[1,k_i]$,
define $J_{i,j}^-:=(z_i-jd_i,z_i-[j-1]d_i)$ and define
$J_{\mathrm{left}}^{(i+1)}:=(a_i,z_i-k_id_i)$ and
$J_{\mathrm{right}}^{(i+1)}:=(b_i,z_i-[k_i-1]d_i)$.
From an argument similar to that used in \eqref{20250620.1626}
with $J^-$, $J_y^-$, $\widetilde{J}_{k-1}^-$, and
$\widetilde{J}_k^-$ therein replaced, respectively, by
$J_{\mathrm{left}}^{(i)}$, $J_{\mathrm{right}}^{(i)}$,
$J_{\mathrm{right}}^{(i+1)}$, and $J_{\mathrm{left}}^{(i+1)}$, it
follows that
\begin{align*}
\left|J_{\mathrm{left}}^{(i)}\right|\left[f_{J_{\mathrm{left}}^{(i)}}
-f_{J_{\mathrm{right}}^{(i)}}\right]_+\leq
\left|J_{\mathrm{left}}^{(i+1)}\right|
\left[f_{J_{\mathrm{left}}^{(i+1)}}-f_{J_{\mathrm{right}}^{(i+1)}}\right]_++
\left[\left|J_{\mathrm{left}}^{(i)}\right|-
\left|J_{\mathrm{left}}^{(i+1)}\right|\right]
\|f\|_{\mathrm{BMO}^+(\mathbb{R})}.
\end{align*}
Combining this and \eqref{20250620.1648}, we find that, for any
$i\in\mathbb{N}\cap[2,\infty)$ with $|J_{\mathrm{left}}^{(i)}\cap
J_{\mathrm{right}}^{(i)}|>\frac{1}{2}|J_{\mathrm{left}}^{(i)}|$,
\begin{align}\label{20250620.1737}
\left|J_{\mathrm{left}}^{(1)}\right|\left[f_{J_{\mathrm{left}}^{(1)}}
-f_{J_{\mathrm{right}}^{(1)}}\right]_+\leq
\left|J_{\mathrm{left}}^{(i)}\right|
\left[f_{J_{\mathrm{left}}^{(i)}}-f_{J_{\mathrm{right}}^{(i)}}\right]_++
\left[\left|J_{\mathrm{left}}^{(1)}\right|-
\left|J_{\mathrm{left}}^{(i)}\right|\right]
\|f\|_{\mathrm{BMO}^+(\mathbb{R})}.
\end{align}
If $|J_{\mathrm{left}}^{(N)}\cap J_{\mathrm{right}}^{(N)}|\leq
\frac{1}{2}|J_{\mathrm{left}}^{(N)}|$ for some $N\in\mathbb{N}$, then,
from \eqref{20250620.1737} and an argument similar to that used in
\eqref{20250618.1734} with $J^-$ and $J_y^-$ therein
replaced, respectively, by $J_{\mathrm{left}}^{(N)}$ and
$J_{\mathrm{right}}^{(N)}$, we infer that \eqref{20250618.1654} holds.
Otherwise, applying the fact that
\begin{align*}
\left|J_{\mathrm{left}}^{(i)}\right|
\left[f_{J_{\mathrm{left}}^{(i)}}-f_{J_{\mathrm{right}}^{(i)}}\right]_+\to0
\end{align*}
as $i\to\infty$ (which can be deduced from the definitions of both
$\{J_{\mathrm{left}}^{(i)}\}_{i\in\mathbb{N}}$ and
$\{J_{\mathrm{right}}^{(i)}\}_{i\in\mathbb{N}}$ and the
absolute continuity of integrals) and \eqref{20250620.1737},
we obtain \eqref{20250618.1654}.

In conclusion, for any $x\in I^-$, $y\in I^{+++}$, $J\in\mathcal{I}$ with
$x\in J^+$ and $|J|\geq\frac{|I|}{2}$, and
$J_y:=(y-|J^-|,y+|J^-|)\in\mathcal{I}$,
\eqref{20250618.1654} holds. According to this and the definitions of
both $N^-$ and $N_2^-$, we conclude that \eqref{20250618.1752} holds for
the case $j=2$. This finishes the proof of Proposition \ref{N-:BMO+ to BLO+}.
\end{proof}

\begin{remark}\label{N-:BMO+ to BLO+ remark}
\begin{enumerate}
\item[\rm(i)] We give a counterexample to illustrate that $M^-$ is not bounded on
$\mathrm{BMO}^+(\mathbb{R})$, which is quite different from the classical
theory. Indeed, let
\begin{align*}
f(x):=\begin{cases}
2x&\mathrm{if}\ x\in[0,\infty),\\
\displaystyle
\frac{1}{(x-1)^2}-1&\mathrm{if}\ x\in(-\infty,0).
\end{cases}
\end{align*}
Obviously, $f$ is monotonically increasing on $\mathbb{R}$ and hence
$\|f\|_{\mathrm{BMO}^+(\mathbb{R})}=0$. Moreover, by some simple
calculations, we find that
\begin{align*}
M^-(f)(x)=\begin{cases}
1&\mathrm{if}\ x\in(-\infty,0],\\
\displaystyle
x+\frac{2}{x}&\mathrm{if}\ x\in(0,\sqrt{2}),\\
2x&\mathrm{if}\ x\in[\sqrt{2},\infty).
\end{cases}
\end{align*}
Since $M^-(f)$ is not monotonically increasing on $\mathbb{R}$, it
follows that $\|M^-(f)\|_{\mathrm{BMO}^+(\mathbb{R})}\in(0,\infty]$,
which further implies that $M^-$ is not bounded on
$\mathrm{BMO}^+(\mathbb{R})$.

\item[\rm(ii)] From Remark \ref{Linfty BLO+ BMO+}, we infer that
Proposition \ref{N-:BMO+ to BLO+} is indeed a refinement of \cite[Theorem
5.1]{s(ampa-2018)} via replacing the target space
$\mathrm{BMO}^+(\mathbb{R})$ by its proper subspace
$\mathrm{BLO}^+(\mathbb{R})$, which states that there exists a positive
constant $C$ such that, for any positive function
$f\in\mathrm{BMO}^+(\mathbb{R})$ with $M^-(f)\in
L_{\mathrm{loc}}^1(\mathbb{R})$, $\|M^-(f)\|_{\mathrm{BMO}^+(\mathbb{R})}
\leq C\|f\|_{\mathrm{BMO}^+(\mathbb{R})}$.
\end{enumerate}
\end{remark}

As a consequence of Proposition \ref{N-:BMO+ to BLO+}, we can represent
$\mathrm{BLO}^+(\mathbb{R})$ as the image of $\mathrm{BMO}^+(\mathbb{R})$
under $N^-$, modulo bounded functions. This establishes the Bennett type
characterization of $\mathrm{BLO}^+(\mathbb{R})$ in terms of
$\mathrm{BMO}^+(\mathbb{R})$, which is the one-sided analogue of
\cite[Theorem]{b(pams-1982)}.

\begin{theorem}\label{N-:BMO+ to BLO+ cor}
$f\in\mathrm{BLO}^+(\mathbb{R})$ if and only if there exist $h\in
L^\infty(\mathbb{R})$  and $F\in\mathrm{BMO}^+(\mathbb{R})$ with
$N^-(F)<\infty$ almost everywhere on $\mathbb{R}$ such that
\begin{align}\label{20250620.1837}
f=N^-(F)+h.
\end{align}
Moreover, $\|f\|_{\mathrm{BLO}^+(\mathbb{R})}\sim
\inf\{\|F\|_{\mathrm{BMO}^+(\mathbb{R})}+\|h\|_{L^\infty(\mathbb{R})}\}$,
where the infimum is taken over all decompositions of $f$ of the form
\eqref{20250620.1837} and the positive equivalence constants are
independent of $f$.
\end{theorem}

\begin{proof}
We first show the necessity. Let $f\in\mathrm{BLO}^+(\mathbb{R})$, $x$ be
a Lebesgue point of $f$, and $I\in\mathcal{I}$ with $x\in I^+$. Then
\begin{align*}
\fint_{I^-}f-f(x)\leq\fint_{I^-}f-\mathop\mathrm{ess\,inf}_{I^+}f
\leq\fint_{I^-}\left[f-\mathop\mathrm{ess\,inf}_{I^+}f\right]_+\leq
\|f\|_{\mathrm{BLO}^+(\mathbb{R})}.
\end{align*}
Taking the supremum over all $I\in\mathcal{I}$ with $x\in I^+$
and using the Lebesgue differentiation theorem, we obtain
\begin{align*}
0\leq N^-(f)(x)-f(x)\leq\|f\|_{\mathrm{BLO}^+(\mathbb{R})}
\end{align*}
and hence $N^-(f)-f\in L^\infty(\mathbb{R})$ and
$\|N^-(f)-f\|_{L^\infty(\mathbb{R})}\leq
\|f\|_{\mathrm{BLO}^+(\mathbb{R})}$. Therefore, $f=f+[N^-(f)-f]$ is a
decomposition of $f$ of the form \eqref{20250620.1837} and hence
\begin{align*}
\inf\left\{\|F\|_{\mathrm{BMO}^+(\mathbb{R})}
+\|h\|_{L^\infty(\mathbb{R})}\right\}
\leq\|f\|_{\mathrm{BMO}^+(\mathbb{R})}+
\left\|N^-(f)-f\right\|_{L^\infty(\mathbb{R})}\leq
2\|f\|_{\mathrm{BLO}^+(\mathbb{R})}.
\end{align*}

Now, we prove the sufficiency. Let $h\in L^\infty(\mathbb{R})$  and
$F\in\mathrm{BMO}^+(\mathbb{R})$ with $N^-(F)<\infty$ almost everywhere
on $\mathbb{R}$ be such that $f=N^-(F)+h$. From Proposition
\ref{N-:BMO+ to BLO+}, we infer that
$N^-(F)\in\mathrm{BLO}^+(\mathbb{R})$ and
$\|N^-(F)\|_{\mathrm{BLO}^+(\mathbb{R})}\lesssim
\|F\|_{\mathrm{BMO}^+(\mathbb{R})}$, which,
together with Remark \ref{Linfty BLO+ BMO+}, further implies that
\begin{align*}
\|f\|_{\mathrm{BLO}^+(\mathbb{R})}\leq
\|N^-(F)\|_{\mathrm{BLO}^+(\mathbb{R})}+\|h\|_{\mathrm{BLO}^+(\mathbb{R})}
\lesssim\|F\|_{\mathrm{BMO}^+(\mathbb{R})}+\|h\|_{L^\infty(\mathbb{R})},
\end{align*}
where the implicit positive constant is independent of $F$ and $h$.
Taking the infimum over all decompositions of $f$ of the form
\eqref{20250620.1837}, we find that
\begin{align*}
\|f\|_{\mathrm{BLO}^+(\mathbb{R})}\lesssim
\inf\left\{\|F\|_{\mathrm{BMO}^+(\mathbb{R})}+
\|h\|_{L^\infty(\mathbb{R})}\right\}.
\end{align*}
This finishes the proof of the sufficiency and hence Theorem
\ref{N-:BMO+ to BLO+ cor}.
\end{proof}

\section{Coifman--Rochberg Type Decomposition of
$\mathrm{BLO}^+(\mathbb{R})$ Functions\\and Distances of
Functions in $\mathrm{BLO}^+(\mathbb{R})$ to
${L^\infty(\mathbb{R})}$}
\label{section4}

In this section, we establish the Coifman--Rochberg type decomposition
of $\mathrm{BLO}^+(\mathbb{R})$ functions. As applications, we show that
any $\mathrm{BMO}^+(\mathbb{R})$ function can be represent as the sum of
two $\mathrm{BLO}^+(\mathbb{R})$ functions and we provide an explicit
description of the distance from $\mathrm{BLO}^+(\mathbb{R})$ functions
to $L^\infty(\mathbb{R})$. To begin with, we compile the
Coifman--Rochberg type lemma of $A_1^+(\mathbb{R})$ weights as
follows, which is exactly \cite[Lemma 2.3]{ac(asnsp-1998)} and
\cite[Corollary 3]{mot(tams-1990)}.

\begin{lemma}\label{C-R A1+}
\begin{enumerate}
\item[\rm(i)] Let $f\in L_{\mathrm{loc}}^1(\mathbb{R})$ be such that
$M^-(f)<\infty$ almost everywhere on $\mathbb{R}$. Then there exists
$\delta\in(0,1)$, independent of $f$, such that $[M^-(f)]^\delta\in
A_1^+(\mathbb{R})$ and $[\{M^-(f)\}^\delta]_{A_1^+(\mathbb{R})}\leq4$.

\item[\rm(ii)] Let $\omega\in A_1^+(\mathbb{R})$. Then there exists
$\delta\in(0,1)$, depending only on $[\omega]_{A_1^+(\mathbb{R})}$, such
that $\omega=b[M^-(f)]^\delta$, where $f:=\omega^\frac{1}{\delta}\in
L_{\mathrm{loc}}^1(\mathbb{R})$ and
$b:=\frac{\omega}{[M^-(\omega^\frac{1}{\delta})]^\delta}$
satisfies $1\leq\|b\|_{L^\infty(\mathbb{R})}\leq C$ with $C$ being a
positive constant depending only on $[\omega]_{A_1^+(\mathbb{R})}$.
\end{enumerate}
\end{lemma}

Next, using Lemma \ref{C-R A1+} and Theorem \ref{BLO+ and A1+}, we
establish the Coifman--Rochberg type decomposition of functions in
$\mathrm{BLO}^+(\mathbb{R})$, which also provides an equivalent norm of
$\mathrm{BLO}^+(\mathbb{R})$.

\begin{theorem}\label{C-R BLO+}
$f\in\mathrm{BLO}^+(\mathbb{R})$ if and only if there exist
$\alpha\in(0,\infty)$, a nonnegative function $g\in
L_{\mathrm{loc}}^1(\mathbb{R})$, and $b\in L^\infty(\mathbb{R})$ such that
\begin{align}\label{20250616.1638}
f=\alpha\ln\left(M^-(g)\right)+b.
\end{align}
Moreover,
\begin{align*}
\|f\|_{\mathrm{BLO}^+(\mathbb{R})}\sim\inf\left\{\alpha+
\|b\|_{L^\infty(\mathbb{R})}\right\}=:
\|f\|_{\mathrm{BLO}^+(\mathbb{R})}^*,
\end{align*}
where the infimum is taken over all decompositions of $f$ of the form
\eqref{20250616.1638} and the positive equivalence constants are
independent of $f$.
\end{theorem}

\begin{proof}
We first prove the necessity. Let $f\in\mathrm{BLO}^+(\mathbb{R})$. From
\eqref{20250622.2236} and Lemma \ref{C-R A1+}(ii), it follows that
there exist $\delta\in(0,1)$ (which is independent of $f$), a nonnegative
function $g\in L_{\mathrm{loc}}^1(\mathbb{R})$, and $b\in
L^\infty(\mathbb{R})$ (whose $L^\infty(\mathbb{R})$-norm is independent
of $f$) such that $e^\frac{B}{2\|f\|_{\mathrm{BLO}^+(\mathbb{R})}}=
b[M^-(g)]^\delta$, where $B$ is the same as in Lemma \ref{JN BLO+}.
Therefore,
\begin{align*}
f=\frac{2\|f\|_{\mathrm{BLO}^+(\mathbb{R})}}{B}\ln b+
\frac{2\delta\|f\|_{\mathrm{BLO}^+(\mathbb{R})}}{B}\ln\left(M^-(g)\right)
\end{align*}
and hence
\begin{align*}
\|f\|_{\mathrm{BLO}^+(\mathbb{R})}^*\leq
\frac{2}{B}\left[\delta+\|\ln b\|_{L^\infty(\mathbb{R})}\right]
\|f\|_{\mathrm{BLO}^+(\mathbb{R})}.
\end{align*}
This finishes the proof of the necessity.

Now, we show the sufficiency. Suppose that
$f=\alpha\ln(M^-(g))+b$ is a decomposition of $f$ of the form
\eqref{20250616.1638}. By Lemma \ref{C-R A1+}(i), we conclude that there
exists $\delta\in(0,1)$, independent of $g$, such that $[M^-(g)]^\delta
\in A_1^+(\mathbb{R})$ and $[\{M^-(g)\}^\delta]_{A_1^+(\mathbb{R})}
\leq4$. This, together with \eqref{20250727.1548}, further implies that
$\alpha\ln(M^-(g))\in\mathrm{BLO}^+(\mathbb{R})$ and
\begin{align*}
\left\|\alpha\ln\left(M^-(g)\right)\right\|_{\mathrm{BLO}^+(\mathbb{R})}=
\left\|\frac{\alpha}{\delta}\ln\left(\left[M^-(g)\right]^\delta
\right)\right\|_{\mathrm{BLO}^+(\mathbb{R})}\leq\frac{2\alpha\ln3}{\delta}.
\end{align*}
Combining this and Remark \ref{Linfty BLO+ BMO+}, we find that
$f\in\mathrm{BLO}^+(\mathbb{R})$ and
\begin{align*}
\|f\|_{\mathrm{BLO}^+(\mathbb{R})}\leq2\|b\|_{L^\infty(\mathbb{R})}+
\frac{2\alpha\ln3}{\delta}\leq
\frac{2\ln3}{\delta}\left[\alpha+\|b\|_{L^\infty(\mathbb{R})}\right].
\end{align*}
Taking the infimum over all decompositions $f=\alpha\ln(M^-(g))+b$
as in \eqref{20250616.1638}, we obtain
\begin{align*}
\|f\|_{\mathrm{BLO}^+(\mathbb{R})}\leq\frac{2\ln3}{\delta}
\|f\|_{\mathrm{BLO}^+(\mathbb{R})}^*,
\end{align*}
which completes the proof of the sufficiency and hence
Theorem \ref{C-R BLO+}.
\end{proof}

As a result, we establish the Coifman--Rochberg type characterization
of $\mathrm{BMO}^+(\mathbb{R})$ and we prove that any
$\mathrm{BMO}^+(\mathbb{R})$ function admits a decomposition into the sum
of two $\mathrm{BLO}^+(\mathbb{R})$ functions.

\begin{corollary}\label{BMO+=BLO+-BLO-}
The following statements are mutually equivalent.
\begin{itemize}
\item[\rm(i)] $f\in\mathrm{BMO}^+(\mathbb{R})$.

\item[\rm(ii)] There exist $\alpha,\beta\in(0,\infty)$,
nonnegative functions $g,h\in L_{\mathrm{loc}}^1(\mathbb{R})$,
and $b\in L^\infty(\mathbb{R})$ such that
\begin{align*}
f=\alpha\ln\left(M^-(g)\right)-\beta\ln\left(M^+(h)\right)+b.
\end{align*}

\item[\rm(iii)] There exist $g,h\in\mathrm{BLO}^+(\mathbb{R})$
such that $f=g+h$.
\end{itemize}
\end{corollary}

\begin{proof}
We first show (i) $\Longrightarrow$ (ii). Let
$f\in\mathrm{BMO}^+(\mathbb{R})$. From \eqref{20250727.1611}, we infer
that there exist $\lambda\in(0,\infty)$ and $\omega\in A_2^+(\mathbb{R})$
such that $f=\lambda\ln\omega$. This, together with the Jones
factorization theorem for one-sided Muckenhoupt weights and Lemma
\ref{C-R A1+}(ii), further implies that there exist
$\delta,\sigma\in(0,1)$, nonnegative functions $g,h\in
L_{\mathrm{loc}}^1(\mathbb{R})$, and $\widetilde{b}\in
L^\infty(\mathbb{R})$ such that
\begin{align*}
f=\lambda\delta\ln\left(M^-(g)\right)-\lambda\sigma\ln\left(M^+(h)\right)
+\lambda\ln\widetilde{b}
\end{align*}
and hence (ii) holds with $\alpha:=\lambda\delta$, $\beta:=\lambda\sigma$,
and $b:=\lambda\ln\widetilde{b}$. Assertion (ii) $\Longrightarrow$ (iii)
is a direct consequence of Theorem \ref{C-R BLO+} and assertion (iii)
$\Longrightarrow$ (i) can be deduced immediately from Remark
\ref{Linfty BLO+ BMO+} and the subadditivity of
$\mathrm{BMO}^+(\mathbb{R})$. This finishes
the proof of Corollary \ref{BMO+=BLO+-BLO-}.
\end{proof}

Moreover, we present an explicit quantification of the distance from
functions in $\mathrm{BLO}^+(\mathbb{R})$ to $L^\infty(\mathbb{R})$ in
terms of the $\|\cdot\|_{\mathrm{BLO}^+(\mathbb{R})}^*$-norm, which is
the corresponding result for the one-sided case of
\cite[Theorem 9]{a(na-2024)}.

\begin{corollary}\label{distance BLO+ to Linfty}
Let $f\in\mathrm{BLO}^+(\mathbb{R})$. Then
\begin{align*}
\inf\left\{\|f-h\|_{\mathrm{BLO}^+(\mathbb{R})}^*:
h\in L^\infty(\mathbb{R})\right\}=
\sup\left\{\epsilon\in(0,\infty):e^{\epsilon f}
\in A_1^+(\mathbb{R})\right\}.
\end{align*}
\end{corollary}

\begin{proof}
Notice that
\begin{align*}
\inf\left\{\|f-h\|_{\mathrm{BLO}^+(\mathbb{R})}^*:
h\in L^\infty(\mathbb{R})\right\}
&=\inf\left\{\alpha+\|b\|_{L^\infty(\mathbb{R})}:
h\in L^\infty(\mathbb{R})\mbox{\ and\ }\right.\\
&\qquad f-h=\alpha\ln\left(M^-(g)\right)+b\mbox{\ is\ a\ decomposition}\\
&\qquad\left.\mbox{of\ $f-h$\ of\ the\ form\
\eqref{20250616.1638}}\right\}\\
&=\inf\left\{\alpha\in(0,\infty):\mbox{there\ exists\ $g\in
L_{\mathrm{loc}}^1(\mathbb{R})$}\right.\\
&\qquad\left.\mbox{such\ that\ }f-\alpha\ln\left(M^-(g)\right)\in
L^\infty(\mathbb{R})\right\}.
\end{align*}
Combining this and Lemma \ref{C-R A1+}, we find that
\begin{align*}
&\sup\left\{\epsilon\in(0,\infty):
e^{\epsilon f}\in A_1^+(\mathbb{R})\right\}\\
&\quad=\inf\left\{\epsilon\in(0,\infty):\mbox{there\ exist\
$\delta\in(0,1)$,\ $g\in L_{\mathrm{loc}}^1(\mathbb{R})$,}\right.\\
&\qquad\quad\left.\mbox{and\ $b\in L^\infty(\mathbb{R})$\ such\ that\ }
e^\frac{f}{\epsilon}=b\left[M^-(g)\right]^\delta\right\}\\
&\quad=\inf\left\{\epsilon\in(0,\infty):\mbox{there\ exist\
$\delta\in(0,1)$,\ $g\in L_{\mathrm{loc}}^1(\mathbb{R})$,}\right.\\
&\qquad\quad\left.\mbox{and\ $b\in L^\infty(\mathbb{R})$\ such\ that\ }
\frac{f}{\epsilon}=\ln b+\delta\ln\left(M^-(g)\right)\right\}\\
&\quad=\inf\left\{\epsilon\in(0,\infty):\mbox{there\ exist\
$\delta\in(0,1)$\ and\ $g\in L_{\mathrm{loc}}^1(\mathbb{R})$}\right.\\
&\qquad\quad\left.\mbox{such\ that\ }
f-\epsilon\delta\ln\left(M^-(g)\right)\in L^\infty(\mathbb{R})\right\}\\
&\quad=\inf\left\{\epsilon\in(0,\infty):\mbox{there\ exists\ $g\in
L_{\mathrm{loc}}^1(\mathbb{R})$\ such\ that}\right.\\
&\qquad\quad\left.f-\epsilon\ln\left(M^-(g)\right)
\in L^\infty(\mathbb{R})\right\}\\
&\quad=\inf\left\{\|f-h\|_{\mathrm{BLO}^+(\mathbb{R})}^*:
h\in L^\infty(\mathbb{R})\right\},
\end{align*}
which completes the proof of Corollary \ref{distance BLO+ to Linfty}.
\end{proof}

\section{Higher-Dimensional Extensions and Applications to\\
Doubly Nonlinear Parabolic Equations}
\label{section5}

Throughout this section, we \emph{always fix} $p\in(1,\infty)$. As
noted in Section \ref{section1}, the parabolic BMO space with time
lag is deeply connected with the doubly nonlinear parabolic
equation \eqref{20251009.2225}. The solution of
\eqref{20251009.2225} possesses the following scaling property: if
$u(x,t)$ is a solution of \eqref{20251009.2225}, then so is $u(\lambda
x,\lambda^p t)$ for any $\lambda\in(0,\infty)$. It is this property that
necessitates the use of parabolic rectangles in place of Euclidean cubes
in all estimates related to \eqref{20251009.2225}. Let
$(x,t):=(x_1,\dots,x_n,t)\in\mathbb{R}^{n+1}$ and $L\in(0,\infty)$.
A \emph{parabolic rectangle} $R(x,t,L)$ centered at $(x,t)$ with edge length
$l(R):=L$ is defined by setting
\begin{align*}
R:=R(x,t,L):=Q(x,L)\times\left[t-L^p,t+L^p\right),
\end{align*}
where $Q(x,L):=\{y:=(y_1,\dots,y_n)\in\mathbb{R}^n:x_i-L\leq y_i<x_i+L
\mbox{\ for\ any\ }i\in\mathbb{N}\cap[1,n]\}$ is the \emph{cube} in
$\mathbb{R}^n$ centered at $x$ with edge length $2L$. In addition, for any
given $\gamma\in[0,1)$, define $R^-(\gamma)$, $R^+(\gamma)$, and
$R^{++}(\gamma)$, respectively, by
\begin{align*}
R^-(\gamma):=Q(x,L)\times\left[t-L^p,t-\gamma L^p\right),\
R^+(\gamma):=Q(x,L)\times\left[t+\gamma L^p,t+L^p\right),
\end{align*}
and
\begin{align*}
R^{++}(\gamma):=Q(x,L)\times\left[t+(1+2\gamma)L^p,t+(2+\gamma)L^p\right),
\end{align*}
where $\gamma$ is called the \emph{time lag}. Denote by $\mathcal{R}$ the set
of all parabolic rectangles in $\mathbb{R}^{n+1}$. In addition, for any rectangle
$R:=Q(x,L)\times[t,T)\subset\mathbb{R}^{n+1}$ with $x\in\mathbb{R}^n$, $L\in(0,\infty)$, and
$-\infty<t<T<\infty$, let $l_x(R):=L$ and $l_t(R):=T-t$. Based on these, we proceed to
recall the definition of the parabolic BMO space with time lag in
\cite[Definition 2.4]{kmy(ma-2023)}.

\begin{definition}\label{PBMO+}
Let $\gamma\in[0,1)$. The \emph{parabolic $\mathrm{BMO}$ space
$\mathrm{PBMO}_\gamma^-(\mathbb{R}^{n+1})$ with time lag} is defined to
be the set of all $f\in L_{\mathrm{loc}}^1(\mathbb{R}^{n+1})$ such that
\begin{align*}
\|f\|_{\mathrm{PBMO}_\gamma^-(\mathbb{R}^{n+1})}:=\sup_{R\in\mathcal{R}}
\inf_{c\in\mathbb{R}}\left[\fint_{R^-(\gamma)}(f-c)_+
+\fint_{R^+(\gamma)}(f-c)_-\right]<\infty.
\end{align*}
If the condition above holds with the time axis reversed, then
$f\in\mathrm{PBMO}_\gamma^+(\mathbb{R}^{n+1})$ (the other \emph{parabolic
BMO space with time lag}).
\end{definition}
From the perspective of function spaces, the parabolic BMO space with
time lag can be viewed as a natural higher-dimensional counterpart of
the one-dimensional one-sided BMO space and exhibits analogous harmonic
analysis properties. On the other hand, motivated by the regularity
theory of solutions of \eqref{20251009.2225}, Kinnunen and Saari
\cite{ks(apde-2016)} and Kinnunen and Myyryl\"ainen \cite{km(am-2024)}
introduced the parabolic Muckenhoupt class with time
lag, also as a higher-dimensional extension of the one-dimensional
one-sided Muckenhoupt class. Let $\gamma\in[0,1)$ and $q\in[1,\infty]$.
The \emph{parabolic Muckenhoupt class $A_q^+(\gamma)$ with time lag} is
defined to be the set of all weights $\omega$ on $\mathbb{R}^{n+1}$ such that
\begin{align*}
[\omega]_{A_q^+(\gamma)}:=
\begin{cases}
\displaystyle
\sup_{R\in\mathcal{R}}\fint_{R^-(\gamma)}\omega
\left[\mathop\mathrm{ess\,inf}_{R^+(\gamma)}\omega\right]^{-1}<\infty
&\mathrm{if}\ q=1,\\
\displaystyle
\sup_{R\in\mathcal{R}}\fint_{R^-(\gamma)}\omega
\left[\fint_{R^+(\gamma)}\omega^\frac{1}{1-q}\right]^{q-1}<\infty
&\mathrm{if}\ q\in(1,\infty)
\end{cases}
\end{align*}
and $A_\infty^+(\gamma):=\bigcup_{q\in(1,\infty)}A_q^+(\gamma)$.
If the condition above holds with the time axis reversed, then $\omega\in
A_q^-(\gamma)$ (the other \emph{parabolic Muckenhoupt class
with time lag}). Analogous to \eqref{20251012.1456} and
\eqref{20250727.1611}, in the parabolic setting, all the functions in the
parabolic BMO space with time lag has the following exponential-logarithm
correspondence with all the parabolic Muckenhoupt weights with time lag:
for any given $\gamma\in(0,1)$ and any given $q\in(1,\infty]$,
\begin{align}\label{20251019.1524}
\mathrm{PBMO}_\gamma^-(\mathbb{R}^{n+1})=
\left\{\lambda\ln\omega:\omega\in
A_q^+(\gamma)\mbox{\ \ and\ \ }\lambda\in[0,\infty)\right\}
\end{align}
(see \cite[Lemma 7.4]{ks(apde-2016)}). This relation also fails to holds
for the endpoint case $q=1$ (see Remark \ref{r5.11}). Thus, as in the
$\mathrm{BMO}^+(\mathbb{R})$ case, this leads to a natural question:
\emph{Find a suitable alternative of
$\mathrm{PBMO}_\gamma^-(\mathbb{R}^{n+1})$ such that
\eqref{20251019.1524} holds with $q=1$.} One of the main purposes of this
section is to answer this question positively.

This section is organized as follows. In Subsection \ref{subsection5.1},
we introduce the parabolic BLO space
$\mathrm{PBLO}_\gamma^-(\mathbb{R}^{n+1})$ with time lag as a natural
higher-dimensional analogue of the one-dimensional one-sided BLO space
and characterize it in terms of $A_1^+(\gamma)$ and hence answer the
aforementioned question. Moreover, we prove that
this space is independent of the choice of the time lag. Subsection
\ref{subsection5.2} is devoted to providing a Bennett type
characterization of $\mathrm{PBLO}_\gamma^-(\mathbb{R}^{n+1})$.
Subsequently, in Subsection \ref{subsection5.3}, we characterize
$\mathrm{PBLO}_\gamma^-(\mathbb{R}^{n+1})$ functions in terms of a
Coifman--Rochberg type decomposition. In Subsection \ref{subsection5.4}, we establish the
relationships between $\mathrm{PBLO}_\gamma^-(\mathbb{R}^{n+1})$ and doubly nonlinear parabolic
equations, which is the second main purpose of this section. Finally,
in Subsection \ref{subsection5.5}, we provide a necessary
condition for the negative logarithm of the parabolic distance function to belong to
$\mathrm{PBLO}_\gamma^-(\mathbb{R}^{n+1})$ in terms of the parabolic
$\gamma$-FIT weak porosity of the set.

\subsection{Characterizations of
$\mathrm{PBLO}_\gamma^-(\mathbb{R}^{n+1})$ in Terms of
$A_1^+(\gamma)$}
\label{subsection5.1}

In this subsection, we first introduce the parabolic BLO space
$\mathrm{PBLO}_\gamma^-(\mathbb{R}^{n+1})$ with time lag. Then we show
its exponential-logarithm correspondence with $A_1^+(\gamma)$ (see
Theorem \ref{PBLO+ and A1+}) by applying a parabolic John--Nirenberg
inequality for $\mathrm{PBLO}_\gamma^-(\mathbb{R}^{n+1})$ (see Lemma
\ref{J-N PBLO+}). As a corollary, we prove the independence of
$\mathrm{PBLO}_\gamma^-(\mathbb{R}^{n+1})$
from the choice of the time lag (see Theorem \ref{PBLO+ time lag}).

\begin{definition}\label{PBLO+}
Let $\gamma\in[0,1)$. The \emph{parabolic $\mathrm{BLO}$ space
$\mathrm{PBLO}_\gamma^-(\mathbb{R}^{n+1})$ with time lag} is defined to
be the set of all $f\in L_{\mathrm{loc}}^1(\mathbb{R}^{n+1})$ such that
\begin{align*}
\|f\|_{\mathrm{PBLO}_\gamma^-(\mathbb{R}^{n+1})}:=\sup_{R\in\mathcal{R}}
\fint_{R^-(\gamma)}\left[f-\mathop\mathrm{ess\,inf}
_{R^+(\gamma)}f\right]_++\fint_{R^{++}(\gamma)}\left[f-
\mathop\mathrm{ess\,inf}_{R^+(\gamma)}f\right]_-<\infty.
\end{align*}
If the condition above holds with the time axis reversed, then
$f\in\mathrm{PBLO}_\gamma^+(\mathbb{R}^{n+1})$ (the other \emph{parabolic
BLO space with time lag}).
\end{definition}

Next, we characterize $\mathrm{PBLO}_\gamma^-(\mathbb{R}^{n+1})$ in
terms of $A_1^+(\gamma)$.

\begin{theorem}\label{PBLO+ and A1+}
Let $\gamma\in(0,1)$. Then there exists a positive constant $B$ such
that, for any $f\in\mathrm{PBLO}_\gamma^-(\mathbb{R}^{n+1})$ and
$\epsilon\in(0,\frac{B}{\|f\|_{\mathrm{PBLO}_
\gamma^-(\mathbb{R}^{n+1})}})$, $e^{\epsilon f}\in A_1^+(\gamma)$.
Conversely, for any $\omega\in A_1^+(\gamma)$, $\ln\omega\in
\mathrm{PBLO}_\gamma^-(\mathbb{R}^{n+1})$ and
$\|\ln\omega\|_{\mathrm{PBLO}_\gamma^-(\mathbb{R}^{n+1})}
\leq2\ln(1+[\omega]_{A_1^+(\gamma)})$. In particular,
\begin{align*}
\mathrm{PBLO}_\gamma^-(\mathbb{R}^{n+1})=
\left\{\lambda\ln\omega:\omega\in A_1^+(\gamma)\mbox{\ \
and\ \ }\lambda\in[0,\infty)\right\}.
\end{align*}
\end{theorem}

To show Theorem \ref{PBLO+ and A1+}, we need to establish a parabolic
John--Nirenberg inequality for the parabolic BLO space with time lag in
advance, which is a direct consequence of \cite[Corollary
3.2]{kmy(ma-2023)}.

\begin{lemma}\label{J-N PBLO+}
Let $0\leq\gamma<\alpha<1$ and $f\in
\mathrm{PBLO}_\gamma^-(\mathbb{R}^{n+1})$. Then there exist positive
constants $A$ and $B$ such that, for any $R\in\mathcal{R}$ and
$\lambda\in(0,\infty)$,
\begin{align*}
\left|R^{++}(\alpha)\cap\left\{\left[f-\mathop\mathrm{ess\,inf}
_{R^+(\gamma)}f\right]_->\lambda\right\}\right|\leq
Ae^{-\frac{B\lambda}{\|f\|_{\mathrm{PBLO}_\gamma^+(\mathbb{R}^{n+1})}}}
\left|R^{++}(\alpha)\right|
\end{align*}
and
\begin{align*}
\left|R^-(\alpha)\cap\left\{\left[f-\mathop\mathrm{ess\,inf}
_{R^+(\gamma)}f\right]_+>\lambda\right\}\right|\leq
Ae^{-\frac{B\lambda}{\|f\|_{\mathrm{PBLO}_\gamma^+(\mathbb{R}^{n+1})}}}
\left|R^-(\alpha)\right|
\end{align*}
\end{lemma}

Now, we are ready to prove Theorem \ref{PBLO+ and A1+}.

\begin{proof}[Proof of Theorem \ref{PBLO+ and A1+}]
Let $f\in\mathrm{PBLO}_\gamma^-(\mathbb{R}^{n+1})$ and define
$\alpha:=\frac{1+\gamma}{2}$. From Lemma \ref{J-N PBLO+} and Cavalieri's
principle, it follows that there exist positive constants $A$ and $B$
such that, for any $\epsilon\in(0,\frac{B}{\|f\|_{\mathrm{PBLO}_
\gamma^-(\mathbb{R}^{n+1})}})$ and $R\in\mathcal{R}$,
\begin{align}\label{20251018.1213}
\fint_{R^-(\alpha)}e^{\epsilon f}&\leq
e^{\epsilon\mathop\mathrm{ess\,inf}\limits_{R^+(\gamma)}f}
\fint_{R^-(\alpha)}
e^{\epsilon[f-\mathop\mathrm{ess\,inf}\limits_{R^+(\gamma)}f]_+}\notag\\
&\leq\left[1+\frac{A\epsilon\|f\|_{\mathrm{PBLO}_
\gamma^-(\mathbb{R}^{n+1})}}{B-\epsilon\|f\|_{\mathrm{PBLO}_
\gamma^-(\mathbb{R}^{n+1})}}\right]\mathop\mathrm{ess\,inf}
_{R^+(\gamma)}f\notag\\
&\leq\left[1+\frac{A\epsilon\|f\|_{\mathrm{PBLO}_
\gamma^-(\mathbb{R}^{n+1})}}{B-\epsilon\|f\|_{\mathrm{PBLO}_
\gamma^-(\mathbb{R}^{n+1})}}\right]\mathop\mathrm{ess\,inf}
_{R^+(\alpha)}f.
\end{align}
Taking the supremum over all $R\in\mathcal{R}$ and using \cite[Theorem
3.1]{km(am-2024)}, we obtain $e^{\epsilon f}\in A_1^+(\gamma)$.

Conversely, let $\omega\in A_1^+(\gamma)$. From Jensen's inequality, we
infer that, for any $R\in\mathcal{R}$,
\begin{align*}
\exp\left\{\fint_{R^-(\gamma)}\left[\ln\omega-\mathop\mathrm{ess\,inf}
_{R^+(\gamma)}\ln\omega\right]_+\right\}&\leq
\fint_{R^-(\gamma)}\exp\left\{\left[\ln\omega-
\mathop\mathrm{ess\,inf}_{R^+(\gamma)}\ln\omega\right]_+\right\}\\
&\leq1+\fint_{R^-(\gamma)}\omega\left[\mathop\mathrm{ess\,inf}
_{R^+(\gamma)}\omega\right]^{-1}\leq1+[\omega]_{A_1^+(\gamma)}
\end{align*}
and
\begin{align*}
\exp\left\{\fint_{R^{++}(\gamma)}\left[\ln\omega-\mathop\mathrm{ess\,inf}
_{R^+(\gamma)}\ln\omega\right]_-\right\}&\leq
\fint_{R^{++}(\gamma)}\exp\left\{\left[\ln\omega-
\mathop\mathrm{ess\,inf}_{R^+(\gamma)}\ln\omega\right]_-\right\}\\
&\leq1+\fint_{R^{++}(\gamma)}\frac{\mathop\mathrm{ess\,inf}
\limits_{R^+(\gamma)}\omega}{\omega}\\
&\leq1+\fint_{R^+(\gamma)}\omega\left[\mathop\mathrm{ess\,inf}
_{R^{++}(\gamma)}\omega\right]^{-1}\leq1+[\omega]_{A_1^+(\gamma)}.
\end{align*}
Taking the supremum over all $R\in\mathcal{R}$, we find that $\ln\omega
\in\mathrm{PBLO}_\gamma^-(\mathbb{R}^{n+1})$ and
\begin{align*}
\|\ln\omega\|
_{\mathrm{PBLO}_\gamma^-(\mathbb{R}^{n+1})}\leq
2\ln\left(1+[\omega]_{A_1^+(\gamma)}\right).
\end{align*}
This finishes the proof of Theorem \ref{PBLO+ and A1+}.
\end{proof}

\begin{remark}\label{Linfty PBLO- PBMO-}
Let $\gamma\in(0,1)$. Similar to Remark \ref{Linfty BLO+ BMO+},
we point out that
\begin{align*}
L^\infty(\mathbb{R}^{n+1})\subsetneqq\mathrm{PBLO}_\gamma^-
(\mathbb{R}^{n+1})\subsetneqq\mathrm{PBMO}_\gamma^-(\mathbb{R}^{n+1}).
\end{align*}
Indeed, from Definition \ref{PBLO+}, we
deduce that $L^\infty(\mathbb{R}^{n+1})\subset\mathrm{PBLO}_\gamma^-
(\mathbb{R}^{n+1})\subset\mathrm{PBMO}_\gamma^-(\mathbb{R}^{n+1})$ and,
for any $f\in L^\infty(\mathbb{R}^{n+1})$,
$\|f\|_{\mathrm{PBMO}_\gamma^-(\mathbb{R}^{n+1})}\lesssim
\|f\|_{\mathrm{PBLO}_\gamma^-(\mathbb{R}^{n+1})}\leq
4\|f\|_{L^\infty(\mathbb{R}^{n+1})}$, where the implicit positive constant is
independent of $f$. Using this and the characterization
of the null space of $\mathrm{PBMO}_\gamma^-(\mathbb{R}^{n+1})$ (see
\cite[Theorem 3.1]{kyyz(cvpde-2025)}), we conclude that, for any $f\in
L_{\mathrm{loc}}^1(\mathbb{R}^{n+1})$,
$\|f\|_{\mathrm{PBLO}_\gamma^-(\mathbb{R}^{n+1})}=0$
if and only if there exists a non-increasing function $g:\mathbb{R}\to
\mathbb{R}$ such that, for almost every $(x,t)\in\mathbb{R}^{n+1}$,
$f(x,t)=g(t)$. This further implies that
$[\mathrm{PBLO}_\gamma^-(\mathbb{R}^{n+1})\setminus
L^\infty(\mathbb{R}^{n+1})]\neq\emptyset$. We show
$[\mathrm{PBMO}_\gamma^-(\mathbb{R}^{n+1})\setminus
\mathrm{PBLO}_\gamma^-(\mathbb{R}^{n+1})]\neq\emptyset$ by contradiction.
If $\mathrm{PBMO}_\gamma^-(\mathbb{R}^{n+1})=
\mathrm{PBLO}_\gamma^-(\mathbb{R}^{n+1})$, then
$\mathrm{PBMO}_\gamma^+(\mathbb{R}^{n+1})=
\mathrm{PBLO}_\gamma^+(\mathbb{R}^{n+1})$, which, together with
\cite[Corollary 4.17(ii)]{kyyz(cvpde-2025)}, further implies that
\begin{align*}
\mathrm{PBMO}(\mathbb{R}^{n+1})=\mathrm{PBMO}_\gamma^+(\mathbb{R}^{n+1})
\cap\mathrm{PBMO}_\gamma^-(\mathbb{R}^{n+1})=
\mathrm{PBLO}_\gamma^+(\mathbb{R}^{n+1})\cap
\mathrm{PBLO}_\gamma^-(\mathbb{R}^{n+1}),
\end{align*}
where $\mathrm{PBMO}(\mathbb{R}^{n+1})$ is the \emph{parabolic
$\mathrm{BMO}$ space} which is defined to be the space of all $f\in
L_{\mathrm{loc}}^1(\mathbb{R}^{n+1})$ such that
\begin{align*}
\|f\|_{\mathrm{PBMO}(\mathbb{R}^{n+1})}:=\sup_{R\in\mathcal{R}}
\inf_{c\in\mathbb{R}}\fint_R|f-c|<\infty.
\end{align*}
From this and Theorem \ref{PBLO+ and A1+}, we infer that, for any
$f\in\mathrm{PBMO}(\mathbb{R}^{n+1})$, there exists
$\epsilon\in(0,\infty)$ such that $e^{\epsilon f}\in A_1^+(\gamma)\cap
A_1^-(\gamma)=A_1(\mathbb{R}^{n+1})$, where $A_1(\mathbb{R}^{n+1})$ is
the \emph{Muckenhoupt class related to parabolic rectangles} which
is defined to be the set of all weights $\omega$ on $\mathbb{R}^{n+1}$
such that
\begin{align*}
[\omega]_{A_1(\mathbb{R}^{n+1})}:=\sup_{R\in\mathcal{R}}\fint_R\omega
\left[\mathop\mathrm{ess\,inf}_R\omega\right]^{-1}<\infty.
\end{align*}
By this and \cite[Lemma 1]{cr(pams-1980)}, we find that
$f\in\mathrm{PBLO}(\mathbb{R}^{n+1})$ and hence
$\mathrm{PBMO}(\mathbb{R}^{n+1})=\mathrm{PBLO}(\mathbb{R}^{n+1})$, where
$\mathrm{PBLO}(\mathbb{R}^{n+1})$ is the \emph{parabolic $\mathrm{BLO}$
space} which is defined to be the set of all $f\in
L_{\mathrm{loc}}^1(\mathbb{R}^{n+1})$ such that
\begin{align*}
\|f\|_{\mathrm{PBLO}(\mathbb{R}^{n+1})}:=\sup_{R\in\mathcal{R}}
\fint_R\left|f-\mathop\mathrm{ess\,inf}_Rf\right|<\infty.
\end{align*}
This contradicts the classical result that
$\mathrm{PBLO}(\mathbb{R}^{n+1})$ is not a linear space (see, for
instance, \cite[p.\,553]{b(pams-1982)}). Thus,
$\mathrm{PBLO}_\gamma^-(\mathbb{R}^{n+1})\subsetneqq
\mathrm{PBMO}_\gamma^-(\mathbb{R}^{n+1})$. This further implies that
\eqref{20251019.1524} fails to hold for the endpoint case $q=1$.
\end{remark}

As a result of Theorem \ref{PBLO+ and A1+}, we establish the following
characterization of $\mathrm{PBLO}_\gamma^-(\mathbb{R}^{n+1})$ in terms
of the parabolic John--Nirenberg inequality.

\begin{theorem}\label{J-N PBLO+ 2}
Let $\gamma\in(0,1)$ and $f\in L_{\mathrm{loc}}^1(\mathbb{R}^{n+1})$.
Then the following statements are mutually equivalent.
\begin{enumerate}
\item[\rm(i)] $f\in\mathrm{PBLO}_\gamma^-(\mathbb{R}^{n+1})$.

\item[\rm(ii)] There exist positive constants $A$ and $B$ such that,
for any $R\in\mathcal{R}$ and $\lambda\in(0,\infty)$,
\begin{align*}
\left|R^-(\gamma)\cap\left\{\left[f-\mathop\mathrm{ess\,inf}
_{R^+(\gamma)}f\right]_+>\lambda\right\}\right|\leq
Ae^{-B\lambda}\left|R^-(\gamma)\right|
\end{align*}
and
\begin{align}\label{20251018.1228}
\left|R^{++}(\gamma)\cap\left\{\left[f-\mathop\mathrm{ess\,inf}
_{R^+(\gamma)}f\right]_->\lambda\right\}\right|\leq
Ae^{-B\lambda}\left|R^{++}(\gamma)\right|
\end{align}

\item[\rm(iii)] For any $q\in[1,\infty)$,
\begin{align*}
\|f\|_{\mathrm{PBLO}_\gamma^-(\mathbb{R}^{n+1}),q}:=
\sup_{R\in\mathcal{R}}\left\{\fint_{R^-(\gamma)}
\left[f-\mathop\mathrm{ess\,inf}_{R^+(\gamma)}f\right]_+^q+
\fint_{R^{++}(\gamma)}\left[f-\mathop\mathrm{ess\,inf}_{R^+(\gamma)}f
\right]_-^q\right\}^\frac{1}{q}<\infty.
\end{align*}
\end{enumerate}
Moreover, for any $q\in(1,\infty)$,
$\|f\|_{\mathrm{PBLO}_\gamma^-(\mathbb{R}^{n+1})}\sim
\|f\|_{\mathrm{PBLO}_\gamma^-(\mathbb{R}^{n+1}),q}$,
where the positive equivalence constants are independent of $f$.
\end{theorem}

To show Theorem \ref{J-N PBLO+ 2}, we first recall that the parabolic
Muckenhoupt class with time lag is independent of both the time lag and
the distance between the two rectangles, which is precisely \cite[Theorem
3.1]{km(am-2024)}. In what follows, for any subset
$E\subset\mathbb{R}^{n+1}$ and $a\in(0,\infty)$, let
\begin{align*}
E+(\mathbf{0},a):=\left\{(x,t+a)\in\mathbb{R}^{n+1}:(x,t)\in E\right\}.
\end{align*}

\begin{lemma}\label{Aq+ time lag}
Let $\gamma,\alpha\in(0,1)$, $q\in[1,\infty)$, and
$\tau\in(1-\alpha,\infty)$. Then $\omega\in A_q^+(\gamma)$ if and only if
\begin{align*}
\langle\omega\rangle_{A_q^+(\alpha)}:=
\begin{cases}
\displaystyle
\sup_{R\in\mathcal{R}}\fint_{R^-(\alpha)}\omega
\left[\mathop\mathrm{ess\,inf}_{S^+(\alpha)}\omega\right]^{-1}
<\infty&\mathrm{if}\ q=1,\\
\displaystyle
\sup_{R\in\mathcal{R}}\fint_{R^-(\alpha)}\omega\left[\fint_{S^+(\alpha)}
\omega^\frac{1}{1-q}\right]^{q-1}<\infty&\mathrm{if}\ q\in(1,\infty),
\end{cases}
\end{align*}
where, for any $R\in\mathcal{R}$, $S^+(\alpha):=R^-(\alpha)+
(\mathbf{0},\tau[l(R)]^p)$. Moreover, $[\omega]_{A_q^+(\gamma)}$ and
$\langle\omega\rangle_{A_q^+(\alpha)}$ only depend on each other.
\end{lemma}

We are ready to prove Theorem \ref{J-N PBLO+ 2}.

\begin{proof}[Proof of Theorem \ref{J-N PBLO+ 2}]
The proof that (iii) $\Longrightarrow$ (i) is obvious and the proof that
(ii) $\Longrightarrow$ (iii) follows from Cavalieri's principle. We only
need to show (i) $\Longrightarrow$ (ii). Indeed, let $B$ be the same as
in Lemma \ref{J-N PBLO+}. On the one hand, by Lemma \ref{Aq+ time lag}
with $\tau:=1+\alpha$ therein and \eqref{20251018.1213}, we conclude that
there exists a positive constant $A$ such that, for any $R\in\mathcal{R}$,
\begin{align*}
\fint_{R^-(\gamma)}
e^\frac{B[f-\mathop\mathrm{ess\,inf}\limits_{R^+(\gamma)}f]_+}
{2\|f\|_{\mathrm{PBLO}_\gamma^-(\mathbb{R}^{n+1})}}
&\leq1+\fint_{R^-(\gamma)}e^\frac{B[f-\mathop\mathrm{ess\,inf}
\limits_{R^+(\gamma)}f]}{2\|f\|_{\mathrm{PBLO}_\gamma^-(\mathbb{R}^{n+1})}}\\
&=1+\fint_{R^-(\gamma)}e^\frac{Bf}
{2\|f\|_{\mathrm{PBLO}_\gamma^-(\mathbb{R}^{n+1})}}
\left[\mathop\mathrm{ess\,inf}_{R^+(\gamma)}
e^\frac{Bf}{2\|f\|_{\mathrm{PBLO}_\gamma^-(\mathbb{R}^{n+1})}}\right]^{-1}\\
&\leq1+\left[e^\frac{Bf}
{2\|f\|_{\mathrm{PBLO}_\gamma^-(\mathbb{R}^{n+1})}}\right]
_{A_1^+(\gamma)}\leq A.
\end{align*}
This, together with Chebyshev's inequality, further implies that, for any
$R\in\mathcal{R}$ and $\lambda\in(0,\infty)$,
\begin{align*}
\left|R^-(\gamma)\cap\left\{\left[f-\mathop\mathrm{ess\,inf}_{R^+(\gamma)}f
\right]_+>\lambda\right\}\right|&=\left|R^-(\gamma)\cap
\left\{e^\frac{B[f-\mathop\mathrm{ess\,inf}\limits_{R^+(\gamma)}f]_+}
{2\|f\|_{\mathrm{PBLO}_\gamma^-(\mathbb{R}^{n+1})}}>e^\frac{B\lambda}
{2\|f\|_{\mathrm{PBLO}_\gamma^-(\mathbb{R}^{n+1})}}\right\}\right|\\
&\leq e^{-\frac{B\lambda}
{2\|f\|_{\mathrm{PBLO}_\gamma^-(\mathbb{R}^{n+1})}}}\int_{R^-(\gamma)}
e^\frac{B[f-\mathop\mathrm{ess\,inf}\limits_{R^+(\gamma)}f]_+}
{2\|f\|_{\mathrm{PBLO}_\gamma^-(\mathbb{R}^{n+1})}}\\
&\leq Ae^{-\frac{B\lambda}
{2\|f\|_{\mathrm{PBLO}_\gamma^-(\mathbb{R}^{n+1})}}}
\left|R^-(\gamma)\right|.
\end{align*}
On the other hand, from the same reason, we infer that
\begin{align*}
\fint_{R^{++}(\gamma)}e^\frac{B[f-\mathop\mathrm{ess\,inf}
\limits_{R^+(\gamma)}f]_-}{2\|f\|_{\mathrm{PBLO}_\gamma^-(\mathbb{R}^{n+1})}}
&\leq1+\fint_{R^{++}(\gamma)}
e^\frac{B[\mathop\mathrm{ess\,inf}\limits_{R^+(\gamma)}f-f]}
{2\|f\|_{\mathrm{PBLO}_\gamma^-(\mathbb{R}^{n+1})}}\\
&\leq1+\fint_{R^+(\gamma)}e^\frac{Bf}
{2\|f\|_{\mathrm{PBLO}_\gamma^-(\mathbb{R}^{n+1})}}
\left[\mathop\mathrm{ess\,inf}_{R^{++}(\gamma)}
e^\frac{Bf}{2\|f\|_{\mathrm{PBLO}_\gamma^-(\mathbb{R}^{n+1})}}\right]^{-1}\\
&\leq1+\left[e^\frac{Bf}
{2\|f\|_{\mathrm{PBLO}_\gamma^-(\mathbb{R}^{n+1})}}\right]
_{A_1^+(\gamma)}\leq A
\end{align*}
and hence \eqref{20251018.1228} holds for any $R\in\mathcal{R}$ and
$\lambda\in(0,\infty)$. This finishes the proof that (i) $\Longrightarrow$
(ii) and hence the proof of Theorem \ref{JN BLO+ 2}.
\end{proof}

At the end of this section, we prove that
$\mathrm{PBLO}_\gamma^-(\mathbb{R}^{n+1})$ is
independent of both the time lag and the distance between the three
rectangles, which plays a significant role in the proof of Proposition
\ref{N-:PBMO- to PBLO-}.

\begin{theorem}\label{PBLO+ time lag}
Let $\gamma,\alpha\in(0,1)$ and $\tau\in(1-\alpha,\infty)$. Then $f\in
\mathrm{PBLO}_\gamma^-(\mathbb{R}^{n+1})$ if and only if
$f\in L_{\mathrm{loc}}^1(\mathbb{R}^{n+1})$ and
\begin{align*}
\||f\||_{\mathrm{PBLO}_\alpha^-(\mathbb{R}^{n+1})}:=\sup_{R\in\mathcal{R}}
\fint_{R^-(\alpha)}\left[f-\mathop\mathrm{ess\,inf}_{S^+(\alpha)}f
\right]_++\fint_{S^{++}(\alpha)}\left[f-\mathop\mathrm{ess\,inf}
_{S^+(\alpha)}f\right]_-<\infty,
\end{align*}
where, for any $R\in\mathcal{R}$, $S^+(\alpha):=R^-(\alpha)+
(\mathbf{0},\tau[l(R)]^p)$ and $S^{++}(\alpha):=S^+(\alpha)+
(\mathbf{0},\tau[l(R)]^p)$. Moreover,
$\|f\|_{\mathrm{PBLO}_\gamma^-(\mathbb{R}^{n+1})}\sim
\||f\||_{\mathrm{PBLO}_\gamma^-(\mathbb{R}^{n+1})}$,
where the positive equivalence constants are independent of $f$.
\end{theorem}

\begin{proof}
We first show the necessity. Let $f\in
\mathrm{PBLO}_\gamma^-(\mathbb{R}^{n+1})$. Applying Theorem
\ref{PBLO+ and A1+} and Lemma \ref{Aq+ time lag}, we conclude that there
exist positive constants $B$ and $D$, independent of $f$, such that
\begin{align*}
\left\langle e^\frac{Bf}{2\|f\|_{\mathrm{PBLO}
_\gamma^-(\mathbb{R}^{n+1})}}\right\rangle_{A_1^+(\alpha)}
:=\sup_{R\in\mathcal{R}}\fint_{R^-(\alpha)}
e^\frac{Bf}{2\|f\|_{\mathrm{PBLO}_\gamma^-(\mathbb{R}^{n+1})}}
\left[\mathop\mathrm{ess\,inf}_{S^+(\alpha)}
e^\frac{Bf}{2\|f\|_{\mathrm{PBLO}_\gamma^-(\mathbb{R}^{n+1})}}
\right]^{-1}\leq D,
\end{align*}
which, together with an argument similar to that used in the proof of
Theorem \ref{PBLO+ and A1+} with $\omega$, $R^-(\gamma)$,
$R^+(\gamma)$, $R^{++}(\gamma)$, and $[\omega]_{A_1^+(\gamma)}$ therein
replaced, respectively, by $e^\frac{Bf}{2\|f\|_{\mathrm{PBLO}
_\gamma^-(\mathbb{R}^{n+1})}}$, $R^-(\alpha)$, $S^+(\alpha)$,
$S^{++}(\alpha)$, and $\langle e^\frac{Bf}{2\|f\|_{\mathrm{PBLO}
_\gamma^-(\mathbb{R}^{n+1})}}\rangle_{A_1^+(\alpha)}$, further implies that
\begin{align*}
\||f\||_{\mathrm{PBLO}_\alpha^-(\mathbb{R}^{n+1})}\leq\frac{4\ln(1+D)}{B}
\|f\|_{\mathrm{PBLO}_\gamma^-(\mathbb{R}^{n+1})}.
\end{align*}
This finishes the proof of the necessity.

Next, we prove the sufficiency. Suppose that
$\||f\||_{\mathrm{PBLO}_\alpha^-(\mathbb{R}^{n+1})}<\infty$. From an
argument similar to that used in the proof of Lemma \ref{J-N PBLO+}
with $R^-(\gamma)$, $R^+(\gamma)$, $R^{++}(\gamma)$,
and $\|f\|_{\mathrm{PBLO}_\gamma^-(\mathbb{R}^{n+1})}$ therein replaced,
respectively, by $R^-(\widetilde{\alpha})$, $S^+(\alpha)$,
$S^{++}(\widetilde{\alpha})$, and $\||f\||_{\mathrm{PBLO}
_\alpha^-(\mathbb{R}^{n+1})}$, it follows that
there exist positive constants $\widetilde{A}$ and $\widetilde{B}$,
independent of $f$, such that, for any $R\in\mathcal{R}$ and
$\lambda\in(0,\infty)$,
\begin{align*}
\left|S^{++}(\widetilde{\alpha})\cap\left\{\left[f-\mathop\mathrm{ess\,inf}
_{S^+(\alpha)}f\right]_->\lambda\right\}\right|\leq
\widetilde{A}e^{-\frac{\widetilde{B}\lambda}
{\||f\||_{\mathrm{PBLO}_\alpha^-(\mathbb{R}^{n+1})}}}
\left|S^{++}(\widetilde{\alpha})\right|
\end{align*}
and
\begin{align*}
\left|R^-(\widetilde{\alpha})\cap\left\{\left[f-\mathop\mathrm{ess\,inf}
_{S^+(\alpha)}f\right]_+>\lambda\right\}\right|\leq
\widetilde{A}e^{-\frac{\widetilde{B}\lambda}
{\||f\||_{\mathrm{PBLO}_\alpha^-(\mathbb{R}^{n+1})}}}
\left|R^-(\widetilde{\alpha})\right|,
\end{align*}
where $\widetilde{\alpha}:=\frac{1+\alpha}{2}$ and
$S^{++}(\widetilde{\alpha}):=R^-(\widetilde{\alpha})+
(\mathbf{0},(2\tau+\frac{1-\alpha}{2})[l(R)]^p)$.
Combining this and an argument similar to that used in
\eqref{20251018.1213} with $\epsilon$, $R^-(\alpha)$,
$R^+(\gamma)$, and $R^+(\alpha)$ therein replaced, respectively, by
$\frac{\widetilde{B}}{2\||f\||_{\mathrm{PBLO}
_\alpha^-(\mathbb{R}^{n+1})}}$, $R^-(\widetilde{\alpha})$,
$S^+(\alpha)$, and $S^{++}(\widetilde{\alpha})$, we find that, for any
$R\in\mathcal{R}$,
\begin{align*}
\fint_{R^-(\widetilde{\alpha})}
e^{\frac{\widetilde{B}f}{2\||f\||_{\mathrm{PBLO}
_\alpha^-(\mathbb{R}^{n+1})}}}\leq\left(1+\widetilde{A}\right)
\mathop\mathrm{ess\,inf}_{S^{++}(\widetilde{\alpha})}
e^{\frac{\widetilde{B}f}{2\||f\||_{\mathrm{PBLO}
_\alpha^-(\mathbb{R}^{n+1})}}}.
\end{align*}
Taking the supremum of all $R\in\mathcal{R}$ and applying Lemma
\ref{Aq+ time lag}, we obtain
$e^{\frac{\widetilde{B}}{2\||f\||_{\mathrm{PBLO}
_\alpha^-(\mathbb{R}^{n+1})}}}\in A_1^+(\gamma)$ and there exists a
positive constants $K$ (independent of $f$), such that
$[e^{\frac{\widetilde{B}}{2\||f\||_{\mathrm{PBLO}
_\alpha^-(\mathbb{R}^{n+1})}}}]_{A_1^+(\gamma)}\leq K$. This, together
with Theorem \ref{PBLO+ and A1+}, further implies that $f\in
\mathrm{PBLO}_\gamma^-(\mathbb{R}^{n+1})$ and there exists a positive
constant $K$, independent of $f$, such that $\|f\|
_{\mathrm{PBLO}_\gamma^-(\mathbb{R}^{n+1})}
\leq\frac{4}{\widetilde{B}}\ln(1+K)
\||f\||_{\mathrm{PBLO}_\alpha^-(\mathbb{R}^{n+1})}$. This finishes the
proof of the sufficiency and hence Theorem \ref{PBLO+ time lag}.
\end{proof}

\subsection{Bennett Type Characterization
of $\mathrm{PBLO}_\gamma^-(\mathbb{R}^{n+1})$ in Terms of
$\mathrm{PBMO}_\gamma^-(\mathbb{R}^{n+1})$}
\label{subsection5.2}

Let $\gamma\in[0,1)$. The \emph{uncentered parabolic natural maximal
operator $N^{\gamma-}$ with time lag} is defined by setting, for any
$f\in L_{\mathrm{loc}}^1(\mathbb{R}^{n+1})$ and
$(x,t)\in\mathbb{R}^{n+1}$,
\begin{align*}
N^{\gamma-}(f)(x,t):=\sup_{R\in\mathcal{R},\,(x,t)\in R^+(\gamma)}
\fint_{R^-(\gamma)}f.
\end{align*}
In this subsection, we first show that $N^{\gamma-}$ is bounded
from $\mathrm{PBMO}_\gamma^-(\mathbb{R}^{n+1})$ to
$\mathrm{PBLO}_\gamma^-(\mathbb{R}^{n+1})$ by using Theorem
\ref{PBLO+ time lag}, establishing a parabolic Calder\'on--Zygmund
decomposition, and applying the parabolic chaining lemma from
\cite[Remark 3.4]{s(ampa-2018)} (see the proof of Proposition
\ref{N-:PBMO- to PBLO-}). Using this, we in turn establish the Bennett
type characterization of $\mathrm{PBLO}_\gamma^-(\mathbb{R}^{n+1})$,
which indicates that $\mathrm{PBLO}_\gamma^-(\mathbb{R}^{n+1})$ coincides
exactly with
\begin{align*}
\left\{N^{\gamma-}(F):F\in\mathrm{PBMO}_\gamma^-(\mathbb{R}^{n+1})
\mbox{\ \ and\ \ }N^{\gamma-}(F)<\infty\mbox{\ almost\ everywhere}\right\}
\end{align*}
modulo bounded functions (see Theorem \ref{N-:PBMO- to PBLO- cor}),
analogous to Theorem \ref{N-:BMO+ to BLO+ cor}.

\begin{proposition}\label{N-:PBMO- to PBLO-}
Let $\gamma\in(0,1)$. Then there exists a positive constant $C$ such that, for
any $f\in\mathrm{PBMO}_\gamma^-(\mathbb{R}^{n+1})$ such that
$N^{\gamma-}(f)<\infty$ almost everywhere,
$N^{\gamma-}(f)\in\mathrm{PBLO}_\gamma^-(\mathbb{R}^{n+1})$ and
\begin{align*}
\left\|N^{\gamma-}(f)\right\|_{\mathrm{PBLO}_\gamma^-(\mathbb{R}^{n+1})}
\leq C\|f\|_{\mathrm{PBMO}_\gamma^-(\mathbb{R}^{n+1})}.
\end{align*}
\end{proposition}

To prove Proposition \ref{N-:PBMO- to PBLO-}, we need a parabolic
chaining lemma as a prerequisite, which is \cite[Remark
3.4]{s(ampa-2018)}.

\begin{lemma}\label{chain inequality}
Let $\gamma\in(0,1)$ and $E,F\subset\mathbb{R}^{n+1}$ be two bounded
Borel sets. If
\begin{align*}
&\inf\left\{t\in\mathbb{R}:\mbox{there\ exists\ }x\in\mathbb{R}^n\mbox{\
satisfying\ }(x,t)\in F\right\}\\
&\qquad-\sup\left\{t\in\mathbb{R}:\mbox{there\ exists\ }x\in\mathbb{R}^n
\mbox{\ satisfying\ }(x,t)\in E\right\}\\
&\quad\in(0,\infty),
\end{align*}
then there exists a positive constant $C$ such that, for any
$f\in\mathrm{PBMO}_\gamma^-(\mathbb{R}^{n+1})$,
\begin{align*}
(f_E-f_F)_+\leq C\|f\|_{\mathrm{PBMO}_\gamma^-(\mathbb{R}^{n+1})}.
\end{align*}
\end{lemma}

Now, we are ready to show Proposition \ref{N-:PBMO- to PBLO-}.

\begin{proof}[Proof of Proposition \ref{N-:PBMO- to PBLO-}]
Let $f\in\mathrm{PBMO}_\gamma^-(\mathbb{R}^{n+1})$ with
$N^{\gamma-}(f)<\infty$ almost everywhere and let $R\in\mathcal{R}$. Define
\begin{align*}
S^+(\gamma):=R^-(\gamma)+\left(\mathbf{0},2(1+2\gamma)^p[l(R)]^p\right),\
S^{++}(\gamma):=S^+(\gamma)+\left(\mathbf{0},(1+\gamma)[l(R)]^p\right),
\end{align*}
and
\begin{align*}
S^{+++}(\gamma):=S^{++}(\gamma)+\left(\mathbf{0},2(1+2\gamma)^p[l(R)]^p\right).
\end{align*}
For any $(x,t)\in R^-(\gamma)$, let
\begin{align*}
N^{\gamma-}_1(f)(x,t):=\sup\left\{\fint_{P^-(\gamma)}f:P\in\mathcal{R},\
(x,t)\in P^+(\gamma), \mbox{and}\ l(P)\leq\gamma l(R)\right\}
\end{align*}
and
\begin{align*}
N^{\gamma-}_2(f)(x,t):=\sup\left\{\fint_{P^-(\gamma)}f:P\in\mathcal{R},\
(x,t)\in P^+(\gamma), \mbox{and}\ l(P)\geq\gamma l(R)\right\}
\end{align*}
From Theorem \ref{PBLO+ time lag}, it follows that we only need to prove that
there exists a positive constant $C$, independent of $f$ and $R$, such that,
for any $j\in\{1,2\}$,
\begin{align}\label{20251004.0031}
\fint_{R^-(\gamma)}\left[N^{\gamma-}_j(f)-\mathop\mathrm{ess\,inf}
_{S^{++}(\gamma)}N^{\gamma-}(f)\right]_+\leq
C\|f\|_{\mathrm{PBMO}_\gamma^-(\mathbb{R}^{n+1})}
\end{align}
and
\begin{align}\label{20251004.0032}
\fint_{S^{+++}(\gamma)}\left[N^{\gamma-}_j(f)-\mathop\mathrm{ess\,inf}
_{S^{++}(\gamma)}N^{\gamma-}(f)\right]_-\leq
C\|f\|_{\mathrm{PBMO}_\gamma^-(\mathbb{R}^{n+1})}.
\end{align}
We first show \eqref{20251004.0031} for the case $j=1$. Let
\begin{align*}
\alpha:=1-\frac{(1-\gamma)\gamma^p+(1-\gamma)+2\gamma^p}{(1+2\gamma)^p}\in(0,1)
\end{align*}
and $\widetilde{R}\in \mathcal{R}$ be such that
$l(\widetilde{R})=(1+2\gamma)l(R)$ and
\begin{align*}
&\sup\left\{t\in\mathbb{R}:\mbox{there\ exists\ }x\in\mathbb{R}^n\mbox{\
satisfying\ }(x,t)\in\widetilde{R}^-(\alpha)\right\}\\
&\qquad\quad-\sup\left\{t\in\mathbb{R}:\mbox{there\ exists\ }x\in\mathbb{R}^n
\mbox{\ satisfying\ }(x,t)\in R^-(\gamma)\right\}\\
&\quad=(1-\gamma)\gamma^p[l(R)]^p.
\end{align*}
By some simple calculations, we find that the following statements hold:
\begin{enumerate}
\item[\rm(i)] $R^-(\gamma)\subsetneqq\widetilde{R}^-(\alpha)$ and
$|\widetilde{R}^-(\alpha)|=(1+2\gamma)^n(1+\gamma^p+
\frac{2\gamma^p}{1-\gamma})|R^-(\gamma)|$;

\item[\rm(ii)] $S^+(\gamma)$ is strictly higher than
$\widetilde{R}^+(\alpha)$. More precisely,
\begin{align*}
&\inf\left\{t\in\mathbb{R}:\mbox{there\ exists\ }x\in\mathbb{R}^n\mbox{\
satisfying\ }(x,t)\in S^+(\gamma)\right\}\\
&\qquad\quad-\sup\left\{t\in\mathbb{R}:\mbox{there\ exists\ }x\in\mathbb{R}^n
\mbox{\ satisfying\ }(x,t)\in \widetilde{R}^+(\alpha)\right\}\\
&\quad=2\gamma^p[l(R)]^p\in(0,\infty);
\end{align*}

\item[\rm(iii)] For any $(y,s)\in S^{++}(\gamma)$,
$N^{\gamma-}(f)(y,s)\geq f_{S^+(\gamma)}$.
\end{enumerate}
Next, we establish a parabolic Calder\'on--Zygmund decomposition of
$\widetilde{R}^-(\alpha)$. Given $\lambda\in
(f_{\widetilde{R}^+(\alpha)},\infty)$. Divide each spatial edge of
$\widetilde{R}^-(\alpha)$ into 2 equally long intervals and partition the
temporal edge of $\widetilde{R}^-(\alpha)$ into $\lceil2^p\rceil$ equally
long intervals. Then we obtain $2^n\lceil2^p\rceil$ subrectangles
$\{U_{1,i}^-\}_{i\in\mathbb{N}\cap[1,J_1]}$ of $\widetilde{R}^-(\alpha)$,
where $J_1:=2^n\lceil2^p\rceil$. For any $i\in\mathbb{N}\cap[1,J_1]$,
there exists a unique $R_{1,i}\in\mathcal{R}$ such that the bottoms of both
$U_{1,i}^-$ and $R_{1,i}$ coincide. We select all $U_{1,i}^-$ for which
$f_{R_{1,i}^+(\alpha)}>\lambda$. For those $U_{1,i}^-$ with
$f_{R_{1,i}^+(\alpha)}\leq\lambda$, we divide each spatial edge of $U_{1,i}^-$
into 2 equally long intervals. If
\begin{align*}
\frac{l_t(U_{1,i}^-)}{\lfloor2^p\rfloor}<\frac{l_t(P^-(\alpha))}{2^p},
\end{align*}
then we partition the temporal edge of $U_{1,i}^-$ into $\lfloor2^p\rfloor$
equally long intervals. Otherwise, if
\begin{align*}
\frac{l_t(U_{1,i}^-)}{\lfloor2^p\rfloor}\geq\frac{l_t(P^-(\alpha))}{2^p},
\end{align*}
then we partition the temporal edge of $U_{1,i}^-$ into $\lceil2^p\rceil$
equally long intervals. We obtain a sequence of subrectangles $\{U_{2,i}^-\}
_{i\in\mathbb{N}\cap[1,J_2]}$ of $\widetilde{R}^-(\alpha)$, where
$J_2\in\mathbb{N}$. For any $i\in\mathbb{N}\cap[1,J_2]$,
there exists a unique $R_{2,i}\in\mathcal{R}$ such that the bottoms of both
$U_{2,i}^-$ and $R_{2,i}$ coincide. We select all $U_{2,i}^-$ for which
$f_{R_{2,i}^+(\alpha)}>\lambda$. We continue the decomposition process
inductively. For any $j\in\mathbb{N}$, at the $j$-step, for those $U_{j,i}^-$
with $f_{R_{j,i}^+(\alpha)}\leq\lambda$, we divide each spatial edge of
$U_{j,i}^-$ into 2 equally long intervals. If
\begin{align*}
\frac{l_t(U_{j,i}^-)}{\lfloor2^p\rfloor}<\frac{l_t(P^-(\alpha))}{2^{pj}},
\end{align*}
then we partition the temporal edge of $U_{j,i}^-$ into $\lfloor2^p\rfloor$
equally long intervals. Otherwise, if
\begin{align*}
\frac{l_t(U_{j,i}^-)}{\lfloor2^p\rfloor}\geq\frac{l_t(P^-(\alpha))}{2^{pj}},
\end{align*}
then we partition the temporal edge of $U_{j,i}^-$ into $\lceil2^p\rceil$
equally long intervals. We obtain a sequence of subrectangles $\{U_{j+1,i}^-\}
_{i\in\mathbb{N}\cap[1,J_{j+1}]}$ of $\widetilde{R}^-(\alpha)$, where
$J_{j+1}\in\mathbb{N}$. For any $i\in\mathbb{N}\cap[1,J_{j+1}]$,
there exists a unique $R_{j+1,i}\in\mathcal{R}$ such that the bottoms of both
$U_{j+1,i}^-$ and $R_{j+1,i}$ coincide. We select all $U_{j+1,i}^-$ for which
$f_{R_{j+1,i}^+(\alpha)}>\lambda$. From a simple calculation, we deduce
that, for any $j\in\mathbb{N}$ and $i\in J_j$, $U_{j,i}^-\subset
R_{j,i}^-(\alpha)$. In conclusion, we obtain a sequence of pairwise
disjoint subrectangles $\{U_i^-\}_{i\in\mathbb{N}}$ of
$\widetilde{R}^-(\alpha)$ such that the following statements hold:
\begin{enumerate}
\item[\rm(iv)] For any $i\in\mathbb{N}$, there exists a unique
$R_i\in\mathcal{R}$ such that the bottoms of both $U_i^-$ and $R_i$ coincide,
$U_i^-\subset R_i^-(\alpha)$, and $f_{R_i^+(\alpha)}>\lambda$;

\item[\rm(v)] For any $i\in\mathbb{N}$, there exists a unique
subrectangle $U_{i^-}^-$ of $\widetilde{R}^-(\alpha)$ such that
$U_i^-$ is obtained by decomposing $U_{i^-}^-$ directly,
$2^n\lfloor2^p\rfloor|U_i^-|\leq|U_{i^-}^-|\leq2^n\lceil2^p\rceil|U_i^-|$,
and $f_{R_{i^-}^+(\alpha)}\leq\lambda$, where $R_{i^-}$ denotes the unique
parabolic rectangle whose bottom coincides with that of $U_{i^-}$.
\end{enumerate}

Now, let
\begin{align*}
\begin{cases}
\displaystyle
g:=f\boldsymbol{1}_{\widetilde{R}^-(\alpha)\setminus
\bigcup_{i\in\mathbb{N}}U_i^-}+
\sum\limits_{i\in\mathbb{N}}f_{R_{i^-}^+(\alpha)}\boldsymbol{1}_{U_i^-},\\
\displaystyle
b_i:=\left[f-f_{R_{i^-}^+(\alpha)}\right]\boldsymbol{1}_{U_i^-}
\mbox{\ for\ any\ }i\in\mathbb{N},\\
\displaystyle
b:=\sum\limits_{i\in\mathbb{N}}b_i.
\end{cases}
\end{align*}
Then $f\boldsymbol{1}_{\widetilde{R}^-(\alpha)}=g+b$ and, by the Lebesgue
differentiation theorem, $g\leq\lambda$ almost everywhere in
$\widetilde{R}^-(\alpha)$. In addition, for any $(x,t)\in R^-(\gamma)$,
\begin{align*}
N^{\gamma-}_1(f)(x,t)&=\sup\left\{\fint_{P^-(\gamma)}f
\boldsymbol{1}_{\widetilde{R}^-(\alpha)}:P\in\mathcal{R},\
(x,t)\in P^+(\gamma), \mbox{and}\ l(P)\leq\gamma l(R)\right\}\\
&\leq M^{\gamma-}(b_+)(x,t)+\lambda,
\end{align*}
where $M^{\gamma-}$ denotes the \emph{uncentered parabolic maximal
operator}, which is defined by setting, for any $f\in
L_{\mathrm{loc}}^1(\mathbb{R}^{n+1})$ and $(x,t)\in\mathbb{R}^{n+1}$,
\begin{align*}
M^{\gamma-}(f)(x,t):=\sup_{R\in\mathcal{R},\,(x,t)\in
R^+(\gamma)}\fint_{R^-(\gamma)}|f|.
\end{align*}
Combining this, (iii), (i), H\"older's inequality, the assumption that
$\lambda>f_{\widetilde{R}^+(\alpha)}$, the fact that $M^{\gamma-}$ is
bounded on $L^2(\mathbb{R}^{n+1})$, Lemma \ref{chain inequality} with
$E:=\widetilde{R}^+(\alpha)$ and $F:=S^+(\gamma)$ therein, and (ii),
we conclude that
\begin{align}\label{20251008.1704}
&\int_{R^-(\gamma)}\left[N^{\gamma-}_1(f)-\mathop\mathrm{ess\,inf}
_{S^{++}(\gamma)}N^{\gamma-}(f)\right]_+\notag\\
&\quad\leq\int_{R^-(\gamma)}\left[M^{\gamma-}(b_+)+\lambda
-\fint_{S^+(\gamma)}f\right]_+\notag\\
&\quad\leq\int_{\widetilde{R}^-(\alpha)}\left[M^{\gamma-}(b_+)+\lambda
-\fint_{\widetilde{R}^+(\alpha)}f\right]_++
\left[\fint_{\widetilde{R}^+(\alpha)}f
-\fint_{S^+(\gamma)}f\right]_+\left|R^-(\gamma)\right|\notag\\
&\quad\lesssim\|b_+\|_{L^2(\mathbb{R}^{n+1})}
\left|R^-(\gamma)\right|^\frac{1}{2}+\left[\lambda-
\fint_{\widetilde{R}^+(\alpha)}f\right]\left|R^-(\gamma)\right|+
\|f\|_{\mathrm{PBMO}_\gamma^-(\mathbb{R}^{n+1})}
\left|R^-(\gamma)\right|,
\end{align}
where the implicit positive constant is independent of $f$.
We then prove that $b_+\in L^2(\mathbb{R}^{n+1})$ and
\begin{align}\label{20251008.1657}
\|b_+\|_{L^2(\mathbb{R}^{n+1})}\lesssim\left|R^-(\gamma)\right|
^\frac{1}{2}\|f\|_{\mathrm{PBMO}_\gamma^-(\mathbb{R}^{n+1})},
\end{align}
where the implicit positive constant is independent of $f$. Indeed, from
the definition of $b$, the fact that $\{U_i^-\}_{i\in\mathbb{N}}$ are
pairwise disjoint, (iv), (v), and H\"older's inequality, we infer that,
for any $\{c_i\}_{i\in\mathbb{N}}$ in $\mathbb{R}$,
\begin{align}\label{20251007.2304}
\|b_+\|_{L^2(\mathbb{R}^{n+1})}^2&=\int_{\mathbb{R}}
\left\{\sum_{i\in\mathbb{N}}\left[f-f_{R_{i^-}^+(\alpha)}\right]
\boldsymbol{1}_{U_i^-}\right\}^2=\sum_{i\in\mathbb{N}}\int_{U_i^-}
\left[f-f_{R_{i^-}^+(\alpha)}\right]_+^2\notag\\
&\leq\sum_{i\in\mathbb{N}}\int_{\widetilde{R}_{i^-}^-(\alpha)}
\left[f-c_i+c_i-f_{R_{i^-}^+(\alpha)}\right]_+^2\notag\\
&\leq2\sum_{i\in\mathbb{N}}\left\{\int_{\widetilde{R}_{i^-}^-(\alpha)}
(f-c_i)_+^2+\int_{\widetilde{R}_{i^-}^-(\alpha)}
\left[\fint_{\widetilde{R}_{i^-}^+(\alpha)}(f-c_i)_-\right]^2\right\}\notag\\
&\leq2\sum_{i\in\mathbb{N}}\left\{\int_{\widetilde{R}_{i^-}^-(\alpha)}
(f-c_i)_+^2+\int_{\widetilde{R}_{i^-}^+(\alpha)}(f-c_i)_-^2\right\}.
\end{align}
Applying Cavaieri's principle and the parabolic John--Nirenberg inequality
for $\mathrm{PBMO}_\gamma^-(\mathbb{R}^{n+1})$ (see \cite[Theorem
4.1]{kmy(ma-2023)}), we find that there exist positive constants $A$ and
$B$, independent of $f$, such that, for any $i\in\mathbb{N}$, there
exists $c_i\in\mathbb{R}$ such that
\begin{align}\label{20251007.2215}
\int_{\widetilde{R}_{i^-}^-(\alpha)}(f-c_i)_+^2&=\int_0^\infty
\left|\widetilde{R}_{i^-}^-(\alpha)\cap\left\{(f-c_i)_+^2
>\lambda\right\}\right|\,d\lambda\notag\\
&=\int_0^\infty\left|\widetilde{R}_{i^-}^-(\alpha)\cap\left\{(f-c_i)_+
>\sqrt{\lambda}\right\}\right|\,d\lambda\notag\\
&\leq A\left|\widetilde{R}_{i^-}^-(\alpha)\right|\int_0^\infty
e^{-\frac{B\sqrt{\lambda}}{\|f\|_{\mathrm{PBMO}_\gamma^-
(\mathbb{R}^{n+1})}}}\,d\lambda
=\frac{2A\|f\|_{\mathrm{PBMO}_\gamma^-(\mathbb{R}^{n+1})}^2}{B}
\left|\widetilde{R}_{i^-}^-(\alpha)\right|.
\end{align}
Similarly, we can show that, for any $i\in\mathbb{N}$,
\begin{align*}
\int_{\widetilde{R}_{i^-}^+(\alpha)}(f-c_i)_-^2
\leq\frac{2A\|f\|_{\mathrm{PBMO}_\gamma^-(\mathbb{R}^{n+1})}^2}{B}
\left|\widetilde{R}_{i^-}^+(\alpha)\right|.
\end{align*}
This, together with \eqref{20251007.2304}, \eqref{20251007.2215}, (v),
and the fact that $\{U_i^-\}_{i\in\mathbb{N}}$ are pairwise disjoint,
further implies that
\begin{align*}
\|b_+\|_{L^2(\mathbb{R}^{n+1})}^2&\leq\frac{8A
\|f\|_{\mathrm{PBMO}_\gamma^-(\mathbb{R}^{n+1})}^2}{B}
\sum_{i\in\mathbb{N}}\left|\widetilde{R}_{i^-}^-(\alpha)\right|\\
&\lesssim\|f\|_{\mathrm{PBMO}_\gamma^-(\mathbb{R}^{n+1})}^2
\sum_{i\in\mathbb{N}}\left|U_i^-\right|\leq
\|f\|_{\mathrm{PBMO}_\gamma^-(\mathbb{R}^{n+1})}^2\left|R^-(\gamma)\right|
\end{align*}
and hence \eqref{20251008.1657} holds. Combining \eqref{20251008.1704}
and \eqref{20251008.1657} and letting $\lambda\to
f_{\widetilde{R}^+(\alpha)}$, we conclude that \eqref{20251004.0031}
holds for the case $j=1$.

Next, we prove \eqref{20251004.0031} for the case $j=2$. Notice that, for
any $P\in\mathcal{R}$ such that there exists $(x,t)\in R^-(\gamma)$
satisfying $(x,t)\in P^+(\gamma)$ and $l(P)\geq\gamma l(R)$,
\begin{align*}
&\sup\left\{t\in\mathbb{R}:\mbox{there\ exists\ }x\in\mathbb{R}^n\mbox{\
satisfying\ }(x,t)\in P^-(\gamma)\right\}\\
&\qquad\quad-\sup\left\{t\in\mathbb{R}:\mbox{there\ exists\ }x\in\mathbb{R}^n
\mbox{\ satisfying\ }(x,t)\in R^-(\gamma)\right\}\\
&\quad=(1-\gamma)[l(R)]^p-2\gamma[l(P)]^p
\leq\left(1-\gamma-2\gamma^{p+1}\right)[l(R)]^p\\
&\quad\leq\left[2(1+2\gamma)^p-(1-\gamma)\right][l(R)]^p\\
&\quad=\inf\left\{t\in\mathbb{R}:\mbox{there\ exists\ }
x\in\mathbb{R}^n\mbox{\ satisfying\ }(x,t)\in S^+(\gamma)\right\}\\
&\qquad\quad-\sup\left\{t\in\mathbb{R}:\mbox{there\ exists\ }
x\in\mathbb{R}^n\mbox{\ satisfying\ }(x,t)\in R^-(\gamma)\right\},
\end{align*}
which further implies that $S^+(\gamma)$ is strictly higher than
$P^-(\gamma)$. From this and Lemma \ref{chain inequality} with
$E:=P^-(\gamma)$ and $F:=S^+(\gamma)$ therein, it follows that, for any
$P\in\mathcal{R}$ such that there exists $(x,t)\in R^-(\gamma)$
satisfying $(x,t)\in P^+(\gamma)$ and $l(P)\geq\gamma l(R)$,
\begin{align*}
\left[f_{P^-(\gamma)}-f_{S^+(\gamma)}\right]_+\lesssim
\|f\|_{\mathrm{PBMO}_\gamma^-(\mathbb{R}^{n+1})},
\end{align*}
where the implicit positive constant is independent of $f$. Combining
this, (iii), and the definition of $N^{\gamma-}_2$, we find that
\eqref{20251004.0031} holds for the case $j=2$.

In conclusion, we show that \eqref{20251004.0031} holds for both
$j\in\{1,2\}$. By an argument similar to the above, we can prove that
\eqref{20251004.0032} holds for both $j\in\{1,2\}$, which completes
the proof of Proposition \ref{N-:PBMO- to PBLO-}.
\end{proof}

\begin{remark}\label{r5.11}
\begin{enumerate}
\item[\rm(i)] Similar to Remark \ref{N-:BMO+ to BLO+ remark}, we can
easily a construct example to illustrate that $M^{\gamma-}$ is unbounded
on $\mathrm{PBMO}_\gamma^-(\mathbb{R}^{n+1})$. Indeed, let $f$ be the same
as in Remark \ref{N-:BMO+ to BLO+ remark}(i) and, for any $(x,t)\in
\mathbb{R}^{n+1}$, define $F(x,t):=f(t)$. Then
$\|F\|_{\mathrm{PBMO}_\gamma^-(\mathbb{R}^{n+1})}=0$ and
$M^{\gamma-}(F)<\infty$ almost everywhere on $\mathbb{R}^{n+1}$. However,
for any $(x,t)\in\mathbb{R}^n\times(-\infty,0]$,
$M^{\gamma-}(F)(x,t)\leq1$ and, for any
$(x,t)\in\mathbb{R}^n\times[10,\infty)$,
\begin{align*}
M^{\gamma-}(F)(x,t)\geq\frac{1}{1-\gamma}\int_{t-(1+\gamma)}^{t-2\gamma}
2s\,ds=2t-3-\gamma\geq2t-4\geq16.
\end{align*}
This, together with the characterization of the null space of
$\mathrm{PBMO}_\gamma^-(\mathbb{R}^{n+1})$ (see \cite[Theorem
3.1]{kyyz(cvpde-2025)}), further implies that
$\|M^{\gamma}(F)\|_{\mathrm{PBMO}_\gamma^-(\mathbb{R}^{n+1})}
\in(0,\infty]$. This means that $M^{\gamma-}$ is unbounded
on $\mathrm{PBMO}_\gamma^-(\mathbb{R}^{n+1})$.

\item[\rm(ii)] Applying Remark \ref{Linfty BLO+ BMO+}, we find that
Proposition \ref{N-:PBMO- to PBLO-} is indeed a refinement of
\cite[Lemma 4.1]{s(ampa-2018)}, which states that, for any given
$\gamma\in(0,1)$ there exists a positive constant $C$ such that, for any
positive function $f\in\mathrm{PBMO}_\gamma^-(\mathbb{R}^{n+1})$
with $M^{\gamma-}(f)\in L_{\mathrm{loc}}^1(\mathbb{R}^{n+1})$,
\begin{align*}
\left\|M^{\gamma-}(f)\right\|_{\mathrm{PBMO}_\gamma^-(\mathbb{R}^{n+1})}
\leq C\|f\|_{\mathrm{PBMO}_\gamma^-(\mathbb{R}^{n+1})}.
\end{align*}
\end{enumerate}
\end{remark}

Applying Proposition \ref{N-:PBMO- to PBLO-}, we obtain
the Bennett type characterization of
$\mathrm{PBLO}_\gamma^-(\mathbb{R}^{n+1})$.

\begin{theorem}\label{N-:PBMO- to PBLO- cor}
Let $\gamma\in(0,1)$. Then $f\in\mathrm{PBLO}_\gamma^-(\mathbb{R}^{n+1})$
if and only if there exist $h\in L^\infty(\mathbb{R}^{n+1})$ and
$F\in\mathrm{PBMO}_\gamma^-(\mathbb{R}^{n+1})$ with
$N^{\gamma-}(F)<\infty$ almost everywhere on $\mathbb{R}^{n+1}$ such that
\begin{align}\label{20251008.1751}
f=N^{\gamma-}(F)+h.
\end{align}
Moreover, $\|f\|_{\mathrm{PBLO}_\gamma^-(\mathbb{R}^{n+1})}\sim
\inf\{\|F\|_{\mathrm{PBMO}_\gamma^-(\mathbb{R}^{n+1})}+
\|h\|_{L^\infty(\mathbb{R}^{n+1})}\}$, where the infimum is
taken over all decompositions of $f$ of the form \eqref{20251008.1751}
and the positive equivalence constants are independent of $f$.
\end{theorem}

\begin{proof}
Repeating the proof of Theorem \ref{N-:BMO+ to BLO+ cor} with $N^-$,
$\mathrm{BMO}^+(\mathbb{R})$, $\mathrm{BLO}^+(\mathbb{R})$,
$L^\infty(\mathbb{R})$, $I$, $I^-$, $I^+$, and Proposition
\ref{N-:BMO+ to BLO+} therein replaced,
respectively, by $N^{\gamma-}$,
$\mathrm{PBMO}_\gamma^-(\mathbb{R}^{n+1})$,
$\mathrm{PBLO}_\gamma^-(\mathbb{R}^{n+1})$, $L^\infty(\mathbb{R}^{n+1})$,
$R$, $R^-(\gamma)$, $R^+(\gamma)$, and Proposition
\ref{N-:PBMO- to PBLO-} then completes the proof of Theorem
\ref{N-:PBMO- to PBLO- cor}.
\end{proof}

\subsection{Coifman--Rochberg Type
Decomposition of $\mathrm{PBLO}_\gamma^-(\mathbb{R}^{n+1})$}
\label{subsection5.3}

In this subsection, we establish the Coifman--Rochberg type
decomposition of functions in $\mathrm{PBLO}_\gamma^-(\mathbb{R}^{n+1})$,
which also provides an equivalent norm of
$\mathrm{PBLO}_\gamma^-(\mathbb{R}^{n+1})$ (see Theorem \ref{C-R PBLO-}).
As corollaries, we prove that any
$\mathrm{PBMO}_\gamma^-(\mathbb{R}^{n+1})$ function can split into the
sum of two $\mathrm{PBLO}_\gamma^-(\mathbb{R}^{n+1})$ functions (see
Corollary \ref{PBMO-=PBLO--PBLO+}) and we give an explicit description of
the distance from functions in $\mathrm{PBLO}_\gamma^-(\mathbb{R}^{n+1})$
to $L^\infty(\mathbb{R}^{n+1})$ (see Corollary
\ref{distance PBLO- to Linfty}). We first present the following
Coifman--Rochberg type lemma of $A_1^+(\gamma)$, which is exactly
\cite[Theorem 7.2]{km(am-2024)} and is crucial in establishing the
main result of this subsection.

\begin{lemma}\label{C-R A1+gamma}
Let $\gamma\in(0,1)$.
\begin{enumerate}
\item[\rm(i)] Let $\delta\in(0,1)$ and $f\in
L_{\mathrm{loc}}^1(\mathbb{R}^{n+1})$ be such that
$M^{\gamma-}(f)<\infty$ almost everywhere on $\mathbb{R}^{n+1}$. Then
$[M^{\gamma-}(f)]^\delta\in A_1^+(\gamma)$ and
$[\{M^-(f)\}^\delta]_{A_1^+(\gamma)}$ depends only on $n$, $p$, $\gamma$,
and $\delta$.

\item[\rm(ii)] Let $\omega\in A_1^+(\gamma)$. Then there exists
$\delta\in(0,1)$, depending only on $n$, $p$, $\gamma$,
$[\omega]_{A_1^+(\gamma)}$, such that $\omega=b[M^{\gamma-}(f)]^\delta$,
where $f:=\omega^\frac{1}{\delta}\in
L_{\mathrm{loc}}^1(\mathbb{R}^{n+1})$ and
$b:=\frac{\omega}{[M^{\gamma-}(\omega^\frac{1}{\delta})]^\delta}$
satisfies $1\leq\|b\|_{L^\infty(\mathbb{R}^{n+1})}\leq C$ with $C$ being
a positive constant depending only on $n$, $p$, $\gamma$, and
$[\omega]_{A_1^+(\gamma)}$.
\end{enumerate}
\end{lemma}

The main result of this subsection, the Coifman--Rochberg type
decomposition of functions in $\mathrm{PBLO}_\gamma^-(\mathbb{R}^{n+1})$,
is stated as follows; since its proof is similar to that of Theorem
\ref{C-R BLO+} with Lemma \ref{C-R A1+}, Theorem
\ref{BLO+ and A1+}, and Remark \ref{Linfty BLO+ BMO+} therein replaced,
respectively, by Lemma \ref{C-R A1+gamma}, Theorem \ref{PBLO+ and A1+},
and Remark \ref{Linfty PBLO- PBMO-}, we omit the details.

\begin{theorem}\label{C-R PBLO-}
Let $\gamma\in(0,1)$. Then $f\in\mathrm{PBLO}_\gamma^-(\mathbb{R}^{n+1})$
if and only if there exist $\alpha\in(0,\infty)$, a nonnegative function
$g\in L_{\mathrm{loc}}^1(\mathbb{R}^{n+1})$, and $b\in
L^\infty(\mathbb{R}^{n+1})$ such that
\begin{align}\label{20251009.1018}
f=\alpha\ln\left(M^-(g)\right)+b.
\end{align}
Moreover,
\begin{align*}
\|f\|_{\mathrm{PBLO}_\gamma^-(\mathbb{R}^{n+1})}\sim
\inf\left\{\alpha+\|b\|_{L^\infty(\mathbb{R}^{n+1})}\right\}=:
\|f\|_{\mathrm{PBLO}_\gamma^-(\mathbb{R}^{n+1})}^*,
\end{align*}
where the infimum is taken over all decompositions of $f$ of the form
\eqref{20251009.1018} and the positive equivalence constants are
independent of $f$.
\end{theorem}

As corollaries, we obtain the Coifman--Rochberg type characterization
of $\mathrm{PBMO}_\gamma^-(\mathbb{R}^{n+1})$ in terms of the uncentered
parabolic maximal operator and hence any
$\mathrm{PBMO}_\gamma^-(\mathbb{R}^{n+1})$ function can be represented as
the sum of two $\mathrm{PBLO}_\gamma^-(\mathbb{R}^{n+1})$ functions.

\begin{corollary}\label{PBMO-=PBLO--PBLO+}
Let $\gamma\in(0,1)$. Then the following statements are mutually equivalent.
\begin{itemize}
\item[\rm(i)] $f\in\mathrm{PBMO}_\gamma^-(\mathbb{R}^{n+1})$.

\item[\rm(ii)] There exist $\alpha,\beta\in(0,\infty)$,
nonnegative functions $g,h\in L_{\mathrm{loc}}^1(\mathbb{R}^{n+1})$, and
$b\in L^\infty(\mathbb{R}^{n+1})$ such that
\begin{align*}
f=\alpha\ln\left(M^{\gamma-}(g)\right)
-\beta\ln\left(M^{\gamma+}(h)\right)+b.
\end{align*}

\item[\rm(iii)] There exist $g,h\in
\mathrm{PBLO}_\gamma^-(\mathbb{R}^{n+1})$ such that $f=g+h$.
\end{itemize}
\end{corollary}

\begin{proof}
We first prove (i) $\Longrightarrow$ (ii). Let
$f\in\mathrm{PBMO}_\gamma^-(\mathbb{R}^{n+1})$. From
\eqref{20251019.1524}, it follows that there exist $\lambda\in(0,\infty)$
and $\omega\in A_2^+(\gamma)$ such that $f=\lambda\ln\omega$. This,
together with the Jones factorization theorem for $A_2^+(\gamma)$ weights
(see, \cite[Theorem 7.1]{km(am-2024)}) and Lemma \ref{C-R A1+gamma}(ii),
further implies that there exist $\delta,\sigma\in(0,1)$, nonnegative
functions $g,h\in L_{\mathrm{loc}}^1(\mathbb{R}^{n+1})$, and
$\widetilde{b}\in L^\infty(\mathbb{R}^{n+1})$ such that
\begin{align*}
f=\lambda\delta\ln\left(M^{\gamma-}(g)\right)-
\lambda\sigma\ln\left(M^{\gamma+}(h)\right)
+\lambda\ln\widetilde{b}
\end{align*}
and hence (ii) holds with $\alpha:=\lambda\delta$, $\beta:=\lambda\sigma$,
and $b:=\lambda\ln\widetilde{b}$. Assertion (ii) $\Longrightarrow$ (iii)
is a direct consequence of Theorem \ref{C-R PBLO-} and assertion (iii)
$\Longrightarrow$ (i) can be deduced immediately from Remark
\ref{Linfty PBLO- PBMO-} and the subadditivity of
$\mathrm{PBMO}_\gamma^-(\mathbb{R}^{n+1})$. This finishes
the proof of Corollary \ref{PBMO-=PBLO--PBLO+}.
\end{proof}

At the end of this section, we provide an explicit quantification of the
distance from functions in $\mathrm{PBLO}_\gamma^-(\mathbb{R}^{n+1})$ to
$L^\infty(\mathbb{R}^{n+1})$ in
terms of the $\|\cdot\|_{\mathrm{PBLO}_\gamma^-(\mathbb{R}^{n+1})}^*$-norm.
This result is more refined than \cite[Theorem 5.1]{kyyz(cvpde-2025)}
by characterizing the equivalence therein via a strict equality.

\begin{corollary}\label{distance PBLO- to Linfty}
Let $\gamma\in(0,1)$ and $f\in\mathrm{PBLO}_\gamma^-(\mathbb{R}^{n+1})$. Then
\begin{align*}
\inf\left\{\|f-h\|_{\mathrm{PBLO}_\gamma^-(\mathbb{R}^{n+1})}^*:
h\in L^\infty(\mathbb{R}^{n+1})\right\}=
\sup\left\{\epsilon\in(0,\infty):e^{\epsilon f}\in A_1^+(\gamma)\right\}.
\end{align*}
\end{corollary}

\begin{proof}
By repeating the proof of Corollary \ref{distance BLO+ to Linfty} with
$\|\cdot\|_{\mathrm{BLO}^+(\mathbb{R})}^*$, $L^\infty(\mathbb{R})$,
$A_1^+(\mathbb{R})$, $L_{\mathrm{loc}}^1(\mathbb{R})$, $M^-$, and
\eqref{20250616.1638} therein replaced, respectively, by
$\|\cdot\|_{\mathrm{PBLO}_\gamma^-(\mathbb{R}^{n+1})}^*$,
$L^\infty(\mathbb{R}^{n+1})$,
$A_1^+(\gamma)$, $L_{\mathrm{loc}}^1(\mathbb{R}^{n+1})$, $M^{\gamma-}$,
and \eqref{20251009.1018}, we obtain the desired conclusion,
which then completes the proof of Corollary
\ref{distance PBLO- to Linfty}.
\end{proof}

\subsection{Applications to Doubly Nonlinear Parabolic Equations}
\label{subsection5.4}

In this subsection, we establish the relationships between
$\mathrm{PBLO}_\gamma^-(\mathbb{R}^{n+1})$ and parabolic partial
differential equations. Consider the following parabolic equation
\begin{align}\label{20251008.1048}
\frac{\partial(|u|^{p-2}u)}{\partial t}-\mathrm{div}\ A(x,t,u,\nabla u)=0,
\end{align}
where $A:\mathbb{R}^{n+1}\times\mathbb{R}^{n+1}\to\mathbb{R}^n$ is a
Carath\'eodory function and satisfies the standard structural conditions:
\begin{enumerate}
\item[\rm(i)] there exists a positive constant $C_1$ such that, for
almost every $(x,t)\in\mathbb{R}^{n+1}$,
\begin{align*}
A(x,t,u,\nabla u)\cdot\nabla u\geq C_1|\nabla u|^p;
\end{align*}

\item[\rm(ii)] there exists a positive constant $C_2$ such that, for
almost every $(x,t)\in\mathbb{R}^{n+1}$,
\begin{align*}
\left|A(x,t,u,\nabla u)\right|\leq C_2|\nabla u|^{p-1}.
\end{align*}
\end{enumerate}
In particular, if $A(x,t,u,\nabla u):=|\nabla u|^{p-2}\nabla u$, then
\eqref{20251008.1048} reduces to \eqref{20251009.2225}.
Recall that the \emph{local Sobolev space
$W_{\mathrm{loc}}^{1,p}(\mathbb{R}^n)$} is defined to be the space of
all functions $u$ on $\mathbb{R}^n$ such that, for any compact
subset $K\subset\mathbb{R}^n$,
\begin{align*}
\left(\int_K|u|^p\right)^\frac{1}{p}+\left(\int_K|\nabla
u|^p\right)^\frac{1}{p}<\infty.
\end{align*}
In addition, the \emph{Bochner--Lebesgue space
$L_{\mathrm{loc}}^p(\mathbb{R};W_{\mathrm{loc}}^{1,p}(\mathbb{R}^n))$}
is defined to be the space of all functions $u$ on $\mathbb{R}^{n+1}$ such
that, for any compact interval $I\subset\mathbb{R}$ and any compact
subset $K\subset\mathbb{R}^n$,
\begin{align*}
\int_I\int_K\left[|u(x,t)|^p+|\nabla u(x,t)|^p\right]\,dx\,dt<\infty.
\end{align*}
A function $u\in L_{\mathrm{loc}}^p(\mathbb{R};
W_{\mathrm{loc}}^{1,p}(\mathbb{R}^n))$ is said to be a \emph{weak
solution} to \eqref{20251008.1048} if, for any infinitely differentiable
function $\phi$ on $\mathbb{R}^{n+1}$ with compact support,
\begin{align*}
\int_{\mathbb{R}^{n+1}}\left[A(x,t,u,\nabla u)\cdot\nabla\phi(x,t)
-|u(x,t)|^{p-2}u(x,t)
\frac{\partial\phi}{\partial t}(x,t)\right]\,dx\,dt=0.
\end{align*}
It was showed in \cite[Theorem 2.1]{kk(ma-2007)} that, for any given
time lag $\gamma\in(0,1)$, there exists a positive constant $C$ such
that, for any nonnegative weak solution $u$ to \eqref{20251008.1048} and
for any $R\in\mathcal{R}$,
\begin{align*}
\mathop\mathrm{ess\,sup}_{R^-(\gamma)}u\leq
C\mathop\mathrm{ess\,inf}_{R^+(\gamma)}u
\end{align*}
(see also \cite[p.\,298]{ks(na-2016)}), which further implies that $u\in
A_1^+(\gamma)$. Combining this and Theorem \ref{PBLO+ and A1+}, we
immediately obtain the following relationships between
$\mathrm{PBLO}_\gamma^-(\mathbb{R}^{n+1})$ and nonnegative weak
solutions of \eqref{20251008.1048}.

\begin{proposition}\label{PBLO- and PDE}
Let $u$ be a nonnegative weak solution to \eqref{20251008.1048}. Then,
for any $\gamma\in(0,1)$, $u\in A_1^+(\gamma)$ and $\ln u\in
\mathrm{PBLO}_\gamma^-(\mathbb{R}^{n+1})$.
\end{proposition}

As an application of Propositions \ref{N-:PBMO- to PBLO-} and
\ref{PBLO- and PDE}, we further conclude the mapping property of
nonnegative weak solutions of \eqref{20251008.1048} under $N^{\gamma-}$.

\begin{corollary}\label{PBLO- and PDE cor}
Let $u$ be a nonnegative weak solution to \eqref{20251008.1048} and
$\gamma\in(0,1)$. If $N^{\gamma-}(\ln u)<\infty$ almost everywhere on
$\mathbb{R}^{n+1}$, then $N^{\gamma-}(\ln u)\in
\mathrm{PBLO}_\gamma^-(\mathbb{R}^{n+1})$.
\end{corollary}

Finally, as an application of Theorem \ref{C-R PBLO-} and Proposition
\ref{PBLO- and PDE}, we obtain a result which provides an explicit
description of nonnegative weak solutions of \eqref{20251008.1048}.

\begin{corollary}\label{PBLO- and PDE cor 2}
Let $u$ be a nonnegative weak solution to \eqref{20251008.1048} and
$\gamma\in(0,1)$. Then there exist $\alpha\in(0,\infty)$, a nonnegative
function $g\in L_{\mathrm{loc}}^1(\mathbb{R}^{n+1})$, and $b\in
L^\infty(\mathbb{R}^{n+1})$ such that
$u=b\left[M^{\gamma-}(g)\right]^\alpha$.
\end{corollary}

\subsection{Connections with Weakly Porous Sets}
\label{subsection5.5}

It is well known that $|x|^{-\alpha}\in A_1(\mathbb{R})$ if and only if
$0\leq\alpha<1$. Combining this and \eqref{20251012.1456}, we obtain the
classical result that $-\log\mathrm{dist}(x,\{0\})=-\log|x|
\in\mathrm{BLO}(\mathbb{R})$. For a general nonempty proper subset $E\subset
\mathbb{R}$, the BLO-property of the negative logarithm of the distance function
$-\log\mathrm{dist}(\cdot,E)$ has recently been characterized in terms of the
weak porosity of $E$  (see, for example, \cite{agg(2406.14369), aggm(jfa-2025),
almv(jfa-2024), klv(book-2021), m(prses-2025), pu(2507.21020), v(jam-2025)}).

Specifically, let $I\in\mathcal{I}$. A
subinterval $J\subset I$ is said to be \emph{$E$-free} if $J\cap E=\emptyset$.
Let $\mathcal{M}(I)\subset I$ be a longest $E$-free subinterval of $I$, that is,
$\mathcal{M}(I)\cap E=\emptyset$ and, for any $E$-free subinterval $J\subset
I$, $|J|\leq|\mathcal{M}(I)|$. $E$ is said to be \emph{weakly porous} if there
exist $c,\delta\in(0,1)$ such that, for any $I\in\mathcal{I}$, there exist
$N(I)\in\mathbb{N}$ and pairwise disjoint $E$-free subintervals
$\{J_i\}_{i\in\mathbb{N}\cap[1,N(I)]}$ of $I$ satisfying
\begin{enumerate}
\item[\rm(i)] for any $i\in\mathbb{N}\cap[1,N(I)]$,
$|J_i|\geq\delta|\mathcal{M}(I)|$;

\item[\rm(ii)] $\sum_{i\in\mathbb{N}\cap[1,N(I)]}|J_i|\geq c|I|$.
\end{enumerate}
Anderson et al. \cite[Theorem 1.1]{almv(jfa-2024)} showed that
$-\log\mathrm{dist}(\cdot,E)\in\mathrm{BLO}(\mathbb{R})$ if and only if $E$ is
weakly porous (see also \cite[Lemma 5]{v(jam-2025)}). Moreover, $E$ is said to
be \emph{right-sided weakly porous} if there exist
$c,\delta\in(0,1)$ such that, for any $I\in\mathcal{I}$, there exist
$N(I)\in\mathbb{N}$ and pairwise disjoint $E$-free subintervals
$\{J_i\}_{i\in\mathbb{N}\cap[1,N(I)]}$ of $I^-$ satisfying
\begin{enumerate}
\item[\rm(i)] for any $i\in\mathbb{N}\cap[1,N(I)]$,
$|J_i|\geq\delta|\mathcal{M}(I^+)|$;

\item[\rm(ii)] $\sum_{i\in\mathbb{N}\cap[1,N(I)]}|J_i|\geq c|I^-|$.
\end{enumerate}
Aimar et al. \cite[Theorem 4.5]{aggm(jfa-2025)} proved that there exists
$\alpha\in(0,\infty)$ such that $[\mathrm{dist}(\cdot, E)]^{-\alpha}\in
A_1^+(\mathbb{R})$ if and only if $E$ is right-sided weakly porous. From this
and Theorem \ref{BLO+ and A1+}, we infer that the result corresponding to
\cite[Theorem 1.1]{almv(jfa-2024)} holds in the one-dimensional one-sided case.

\begin{theorem}\label{BLO+ and porous}
Let $E\subsetneqq\mathbb{R}$ be nonempty. Then
$-\log\mathrm{dist}(x,E)\in\mathrm{BLO}^+(\mathbb{R})$ if and only if
$E$ is right-sided weakly porous.
\end{theorem}

Now, we turn to the higher-dimensional case.
Let $E\subsetneqq\mathbb{R}^{n+1}$ be nonempty. In this subsection,
we investigate the necessary geometric condition on $E$ under which
the negative logarithm of the parabolic distance function $-\log d_p(\cdot,E)\in
\mathrm{PBLO}_\gamma^-(\mathbb{R}^{n+1})$, where $d_p$ denotes the
\emph{parabolic distance} on $\mathbb{R}^{n+1}\times\mathbb{R}^{n+1}$ and is
defined by setting, for any $(x,t):=(x_1,\dots,x_n,t),(y,s):=(y_1,\dots,y_n,s)
\in\mathbb{R}^{n+1}$,
\begin{align*}
d_p((x,t),(y,s)):=\max\left\{|x_1-y_1|,\dots,|x_n-y_n|,\,
|t-s|^\frac{1}{p}\right\}.
\end{align*}
To this end, we begin by constructing a parabolic dyadic decomposition of
parabolic rectangles. Let $\gamma\in[0,1)$ and $R\in\mathcal{R}$ with edge
length $L\in(0,\infty)$. Define the \emph{parabolic dyadic decomposition} of
$R^-(\gamma)$ as follows. Divide each spatial edge of $R^-(\gamma)$ into 2
equally long intervals. If
\begin{align*}
\frac{l_t(R^-(\gamma))}{\lfloor2^p\rfloor}<\frac{(1-\gamma)L^p}{2^p},
\end{align*}
then partition the temporal edge of $R^-(\gamma)$ into $\lfloor2^p\rfloor$
equally long intervals. If
\begin{align*}
\frac{l_t(R^-(\gamma))}{\lfloor2^p\rfloor}\geq\frac{(1-\gamma)L^p}{2^p},
\end{align*}
then partition the temporal edge of $R^-(\gamma)$ into $\lceil2^p\rceil$
equally long intervals. We obtain $J_1$ pairwise disjoint half-open
subrectangles of $R^-(\gamma)$, where $J_1\in
\{2^n\lfloor2^p\rfloor,\,2^n\lceil2^p\rceil\}$. We
denote the set of these $J_1$ subrectangles by $\mathcal{D}_1(R^-(\gamma))$. We
continue the decomposition inductively. For any
$j\in\mathbb{N}\cap[2,\infty)$ and $P\in\mathcal{D}_{j-1}(R^-(\gamma))$, divide
each spatial edge of $P$ into 2 equally long intervals. If
\begin{align*}
\frac{l_t(P)}{\lfloor2^p\rfloor}<\frac{(1-\gamma)L^p}{2^{pj}},
\end{align*}
then partition the temporal edge of $P$ into $\lfloor2^p\rfloor$ equally
long intervals. If
\begin{align*}
\frac{l_t(P)}{\lfloor2^p\rfloor}\geq\frac{(1-\gamma)L^p}{2^{pj}},
\end{align*}
then partition the temporal edge of $P$ into into $\lceil2^p\rceil$
equally long intervals. We obtain $J_j$ pairwise disjoint half-open
subrectangles of $R^-(\gamma)$, where $J_j\in
\mathbb{N}\cap[(2^n\lfloor2^p\rfloor)^j,\,(2^n\lceil2^p\rceil)^j]$.
We denote the set of these $J_j$ subrectangles by $\mathcal{D}_j(R^-(\gamma))$.
Finally, let $\mathcal{D}_0(R):=\{R\}$ and
\begin{align*}
\mathcal{D}\left(R^-(\gamma)\right):=
\bigcup_{j\in\mathbb{Z}_+}\mathcal{D}_j\left(R^-(\gamma)\right).
\end{align*}
The rectangles in $\mathcal{D}(R^-(\gamma))$ are called \emph{parabolic dyadic
rectangles} (with respect to $R^-(\gamma)$). The following properties hold.
\begin{enumerate}
\item[(D0)] For any $j\in\mathbb{Z}_+$,
$R^-(\gamma)=\bigcup_{P\in\mathcal{D}_j(R^-(\gamma))}P$.

\item[(D1)] For any $j\in\mathbb{N}$ and $P\in\mathcal{D}_j(R^-(\gamma))$, there
exists a unique $\pi P\in\mathcal{D}_{j-1}(R^-(\gamma))$ such that $P\subset\pi P$.
$\pi P$ is called the \emph{parabolic dyadic parent} of $P$, while $P$ is
called the \emph{parabolic dyadic child} of $\pi P$.

\item[(D2)] For any $j\in\mathbb{Z}_+$ and $P\in\mathcal{D}_j(R^-(\gamma))$, the
number of the parabolic dyadic children of $P$ belongs to
$\{2^n\lfloor2^p\rfloor,\,2^n\lceil2^p\rceil\}$.

\item[(D3)] For any $P_1,P_2\in\mathcal{D}(R^-(\gamma))$, $P_1\cap P_2\in
\{P_1,\,P_2,\,\emptyset\}$.

\item[(D4)] For any $j\in\mathbb{Z}_+$ and $P\in\mathcal{D}_j(R^-(\gamma))$,
$l_x(P)=\frac{L}{2^j}$ and $l_t(P)\in
[\frac{1}{2}\frac{(1-\gamma)L^p}{2^{pj}},\frac{(1-\gamma)L^p}{2^{pj}}]$.
\end{enumerate}
We can construct the parabolic dyadic decomposition for $R^+(\gamma)$ in a completely analogous
manner. The collection of all parabolic dyadic rectangles with respect to $R^+(\gamma)$ is
denoted by $\mathcal{D}(R^+(\gamma))$.

After establishing the dyadic decomposition of parabolic rectangles, we then introduce the
concept of forward-in-time parabolic weak porosity, which can be regarded as a
higher-dimensional generalization of the concept of right-sided weak porosity.

\begin{definition}\label{parabolic porous}
Let $E\subset\mathbb{R}^{n+1}$ be nonempty and $\gamma\in[0,1)$.
\begin{enumerate}
\item[\rm(i)] Let $R\in\mathcal{R}$ and $P\in\mathcal{D}(R^+(\gamma))$. $P$ is said to be
\emph{$E$-free} if $P\cap E=\emptyset$. Denote by
$\mathcal{M}(R^+(\gamma))\in\mathcal{D}(R^+(\gamma))$ a largest
$E$-free parabolic dyadic rectangle of $R^+(\gamma)$, that is, if
$P\in\mathcal{D}(R^+(\gamma))$ is $E$-free, then $|P|\leq|\mathcal{M}(R^+(\gamma))|$.

\item[\rm(ii)] $E$ is said to be \emph{$\gamma$-FIT parabolic weakly porous
($\gamma$-forward-in-time parabolic weakly porous)} if there exist $c,\delta\in(0,1)$ such
that, for any $R\in\mathcal{R}$, there exist $N(R)\in\mathbb{N}$ and
$\{P_j\}_{j\in\mathbb{N}\cap[1,N(R)]}\subset\mathcal{D}(R^-(\gamma))$ satisfying
\begin{enumerate}
\item[\rm(a)] $\{P_j\}_{j\in\mathbb{N}\cap[1,N(R)]}$ are pairwise disjoint
and, for any $j\in\mathbb{N}\cap[1,N(R)]$, $P_j$ is $E$-free;

\item[\rm(b)] for any $j\in\mathbb{N}\cap[1,N(R)]$,
$|P_j|\geq\delta|\mathcal{M}(R^+(\gamma))|$;

\item[\rm(c)] $\sum_{i=1}^{N(R)}|P_j|\geq c|R^-(\gamma)|$.
\end{enumerate}
\end{enumerate}
\end{definition}

\begin{remark}\label{remark on parabolic porous}
Let all the symbols be the same as in Definition \ref{parabolic porous}.
\begin{enumerate}
\item[\rm(i)] We point out that $\mathcal{M}(R^+(\gamma))$ need not be unique,
but we fix one of them.

\item[\rm(ii)] We point out that $E$ is $\gamma$-FIT weakly porous if and only if
$\overline{E}$ is $\gamma$-FIT parabolic weakly porous.

\item[\rm(iii)] We point out that if $E$ is $\gamma$-FIT parabolic weakly porous,
then $|E|=0$. Indeed, from (ii), we may assume that $E$ is closed. By the $\gamma$-FIT
parabolic weak porosity of $E$, we find that there exists $c\in(0,1)$ such that, for any
$R\in\mathcal{R}$, there exist $N(R)\in\mathbb{N}$ and
$\{P_j\}_{j\in\mathbb{N}\cap[1,N(R)]}\subset\mathcal{D}(R^-(\gamma))$ satisfying
\begin{enumerate}
\item[\rm(i)] $\{P_j\}_{j\in\mathbb{N}\cap[1,N(R)]}$ are pairwise disjoint
and, for any $j\in\mathbb{N}\cap[1,N(R)]$, $P_j\cap E=\emptyset$;

\item[\rm(ii)] $\sum_{i=1}^{N(R)}|P_j|\geq c|R^-(\gamma)|$.
\end{enumerate}
Thus,
\begin{align*}
\frac{|E\cap R^-(\gamma)|}{|R^-(\gamma)|}\leq
\frac{|R^-(\gamma)\setminus\bigcup_{i=1}^{N(R)}P_j|}{|R^-(\gamma)|}\leq1-c<1.
\end{align*}
This, together with the arbitrariness of $R$ and the Lebesgue density theorem,
further implies that $|E|=0$.

\item[\rm(iv)] We provide an explicit example of a $\gamma$-FIT parabolic weakly
porous set. Let $E:=\{\Omega\times\{t_k\}:k\in\mathbb{Z}\}$, where $\Omega$ is a nonempty
measurable subset of $\mathbb{R}^n$ and
\begin{align*}
t_k:=\begin{cases}
k &\mathrm{if}\ k\in\mathbb{Z}_+,\\
-\frac{1}{2^k} &\mathrm{if}\ k\in\mathbb{Z}\setminus\mathbb{Z}_+.
\end{cases}
\end{align*}
Then $E$ is $\gamma$-FIT parabolic weakly porous with $\delta=\frac{1}{4}$ and $c=\frac{1}{2}$.
\end{enumerate}
\end{remark}

The main result of this subsection is sated as follows, which provides a necessary condition
for $-\log d_p(\cdot,E)\in\mathrm{PBLO}_\gamma^-(\mathbb{R}^{n+1})$.

\begin{theorem}\label{PBLO- imply porous}
Let $E\subsetneqq\mathbb{R}^{n+1}$ be nonempty and $\gamma\in(0,1)$. If
$-\log d_p(\cdot,E)\in\mathrm{PBLO}_\gamma^-(\mathbb{R}^{n+1})$, then $E$ is $\gamma$-FIT
parabolic weakly porous.
\end{theorem}

Theorem \ref{PBLO- imply porous} is a direct consequence of Theorem \ref{PBLO+ and A1+}
and the following result on the relationship between $A_1^+(\gamma)$ distance weights
and $\gamma$-FIT parabolic weakly porous sets.

\begin{lemma}\label{A1 imply porous}
Let $E\subsetneqq\mathbb{R}^{n+1}$ be nonempty and $\gamma\in[0,1)$.
If there exists $\alpha\in(0,\infty)$ such that
$[d_p(\cdot,E)]^{-\alpha}\in A_1^+(\gamma)$, then $E$ is $\gamma$-FIT
parabolic weakly porous.
\end{lemma}

\begin{proof}
By Remark \ref{remark on parabolic porous}(ii) and the fact that
$d_p(\cdot,E)=d_p(\cdot,\overline{E})$, we may assume that $E$ is closed.
Since $d_p(\cdot,E)^{-\alpha}\in A_1^+(\gamma)\subset L_{\mathrm{loc}}^1(\mathbb{R}^{n+1})$,
it follows that $|E|=0$. Let $R\in\mathcal{R}$ with edge length $L\in(0,\infty)$ and
$\delta\in(0,1)$ be chosen later. Define
\begin{align*}
\widehat{\mathcal{F}^{\gamma+}_\delta}(R):=
\left\{P\in\mathcal{D}\left(R^-(\gamma)\right):|P|\geq\delta
\left|\mathcal{M}\left(R^+(\gamma)\right)\right|\mbox{\ and\ }
P\cap E=\emptyset\right\}
\end{align*}
and let $\mathcal{F}^{\gamma+}_\delta(R)$ be the set of all maximal
rectangles in $\widehat{\mathcal{F}^{\gamma+}_\delta}(R)$, that is,
$P\in\mathcal{F}^{\gamma+}_\delta(R)$ if and only if
$P\in\widehat{\mathcal{F}^{\gamma+}_\delta}(R)$ and, for any
$\widetilde{P}\in\widehat{\mathcal{F}^{\gamma+}_\delta}(R)$,
$P\cap\widetilde{P}\in\{\widetilde{P},\,\emptyset\}$. By (D3), we find that
rectangles in $\mathcal{F}^{\gamma+}_\delta(R)$ are pairwise disjoint.
Moreover, applying Definition \ref{parabolic porous}(ii), we conclude
that, to show that $E$ is $\gamma$-FIT parabolic weakly porous, it suffices
to prove that there exists $c\in(0,1)$, independent of $R$, such that
\begin{align}\label{20251218.2213}
\sum_{P\in\mathcal{F}^{\gamma+}_\delta(R)}|P|\geq
c\left|R^-(\gamma)\right|.
\end{align}

Indeed, from the assumption that $E$ is closed, we deduce that, for
any $(x,t)\in R^-(\gamma)\setminus E$, $d_p((x,t),E)\in(0,\infty)$.
Therefore, there exists $P\in\mathcal{D}(R^-(\gamma))$ such that
$(x,t)\in P$ and $P\cap E=\emptyset$. This further implies that
$R^-(\gamma)\setminus E$ can be represented as a countable disjoint union
of maximal $E$-free rectangles in $\mathcal{D}(R^-(\gamma))$.
If $(R^-(\gamma)\setminus E)\setminus
\bigcup_{P\in\mathcal{F}^{\gamma+}_\delta(R)}P=\emptyset$, then
\begin{align*}
\sum_{P\in\mathcal{F}^{\gamma+}_\delta(R)}|P|=
\left|\bigcup_{P\in\mathcal{F}^{\gamma+}_\delta(R)}P\right|
=\left|R^-(\gamma)\right|
\end{align*}
and hence \eqref{20251218.2213} holds for any $c\in(0,1)$. Thus, without
loss of generality, we may assume that $(R^-(\gamma)\setminus E)\setminus
\bigcup_{P\in\mathcal{F}^{\gamma+}_\delta(R)}P\neq\emptyset$.
By the definition of $\mathcal{F}^{\gamma+}_\delta(R)$ and (D5), we obtain,
for any $(x,t)\in (R^-(\gamma)\setminus E)\setminus
\bigcup_{P\in\mathcal{F}^{\gamma+}_\delta(R)}P$, there exists
a maximal $E$-free rectangle $P\in\mathcal{D}(R^-(\gamma))$ such that
\begin{align}\label{20251220.1716}
|P|<\delta\left|\mathcal{M}\left(R^+(\gamma)\right)\right|\leq
\delta2^n(1-\gamma)l^{n+p},
\end{align}
where $l:=l_x(\mathcal{M}(R^+(\gamma))$. Notice that there exists
$j\in\mathbb{N}$ such that $P\in\mathcal{D}_j(R^-(\gamma))$. From
(D5), we infer that $l_x(P)=\frac{l(R)}{2^j}$
and $l_t(P)\geq\frac{1}{2}\frac{(1-\gamma)[l(R)]^p}{2^{pj}}$, which
further implies that
\begin{align*}
|P|=2^n\left[l_x(P)\right]^nl_t(P)\geq
\frac{2^{n-1}(1-\gamma)[l(R)]^{n+p}}{2^{(n+p)j}}.
\end{align*}
Combining this and \eqref{20251220.1716}, we conclude that
\begin{align}\label{20251220.1229}
\frac{l(R)}{2^j}<(2\delta)^\frac{1}{n+p}l.
\end{align}
In addition, since $P\in\mathcal{D}(R^-(\gamma))$ is maximal $E$-free, it
follows that $\pi P$ is not $E$-free and hence
\begin{align}\label{20251220.1728}
d_p((x,t),E)\leq\mathrm{diam}_p(\pi P):=\sup\left\{d_p(z_1,z_2):
z_1,z_2\in\pi P\right\}.
\end{align}
By (D1) and (D5), we obtain $\pi P\in\mathcal{D}_{j-1}(R^-(\gamma))$,
$l_x(\pi P)=\frac{l(R)}{2^{j-1}}$, and
$l_t(\pi P)\leq\frac{(1-\gamma)[l(R)]^p}{2^{p(j-1)}}$. Thus,
\begin{align*}
\mathrm{diam}_p(\pi P)=\max\left\{2l_x(\pi P),\,
\left[l_t(\pi P)\right]^\frac{1}{p}\right\}
\leq\max\left\{\frac{2l(R)}{2^{j-1}},\,
\frac{(1-\gamma)^\frac{1}{p}l(R)}{2^{j-1}}\right\}
=\frac{4l(R)}{2^j}.
\end{align*}
From this, \eqref{20251220.1229}, and \eqref{20251220.1728},
we deduce that, for any $(x,t)\in (R^-(\gamma)\setminus E)\setminus
\bigcup_{P\in\mathcal{F}^{\gamma+}_\delta(R)}P$,
\begin{align*}
d_p((x,t),E)<4(2\delta)^\frac{1}{n+p}l,
\end{align*}
which, together with the proven conclusion that $|E|=0$ and the condition
that $[d_p(\cdot,E)]^{-\alpha}\in A_1^+(\gamma)$ for some
$\alpha\in(0,\infty)$, further implies that
\begin{align}\label{20251220.1755}
l^{-\alpha}\frac{|R^-(\gamma)\setminus
\bigcup_{P\in\mathcal{F}^{\gamma+}_\delta(R)}P|}{|R^-(\gamma)|}&\leq
\frac{2^{\alpha(2+\frac{1}{n+p})}\delta^\frac{\alpha}{n+p}}{|R^-(\gamma)|}
\int_{(R^-(\gamma)\setminus E)\setminus
\bigcup_{P\in\mathcal{F}^{\gamma+}_\delta(R)}}
\left[d_p(\cdot,E)\right]^{-\alpha}\\
&\leq2^{\alpha(2+\frac{1}{n+p})}\delta^\frac{\alpha}{n+p}
\fint_{R^-(\gamma)}\left[d_p(\cdot,E)\right]^{-\alpha}\notag\\
&\leq2^{\alpha(2+\frac{1}{n+p})}
\left[\left\{d_p(\cdot,E)\right\}^{-\alpha}\right]_{A_1^+(\gamma)}
\delta^\frac{\alpha}{n+p}\mathop\mathrm{ess\,inf}_{R^+(\gamma)}
\left[d_p(\cdot,E)\right]^{-\alpha}.\notag
\end{align}
Let $(z,\tau)\in\mathbb{R}^{n+1}$ be the center of
$\mathcal{M}(R^+(\gamma))$. By (D5) again, we
find that $l_t(\mathcal{M}(R^+(\gamma)))\geq\frac{1-\gamma}{2}l^p$. Therefore,
\begin{align*}
\mathop\mathrm{ess\,inf}_{R^+(\gamma)}\left[d_p(\cdot,E)\right]^{-\alpha}
\leq\left[d_p((z,\tau),E)\right]^{-\alpha}=
\left[\min\left\{l,\,\left[\frac{1}{2}l_t\left(\mathcal{M}\left(R^+(\gamma)
\right)\right)\right]^\frac{1}{p}\right\}\right]^{-\alpha}
\leq\left(\frac{4}{1-\gamma}\right)^\frac{\alpha}{p}l^{-\alpha}.
\end{align*}
Combining this and \eqref{20251220.1755}, we obtain
\begin{align*}
\left|R^-(\gamma)\setminus
\bigcup_{P\in\mathcal{F}^{\gamma+}_\delta(R)}P\right|\leq
2^{\alpha(2+\frac{1}{n+p})}\left(\frac{4}{1-\gamma}\right)^\frac{\alpha}{p}
\left[\left\{d_p(\cdot,E)\right\}^{-\alpha}\right]_{A_1^+(\gamma)}
\delta^\frac{\alpha}{n+p}\left|R^-(\gamma)\right|.
\end{align*}
Finally, choose $\delta\in(0,1)$ such that
\begin{align*}
c:=2^{\alpha(2+\frac{1}{n+p})}\left(\frac{4}{1-\gamma}\right)^\frac{\alpha}{p}
\left[\left\{d_p(\cdot,E)\right\}^{-\alpha}\right]_{A_1^+(\gamma)}
\delta^\frac{\alpha}{n+p}\in(0,1),
\end{align*}
and then we complete the proof of \eqref{20251218.2213} and hence Lemma \ref{A1 imply porous}.
\end{proof}

\begin{remark}
The converse of Lemma \ref{A1 imply porous} is unknown, that is, whether the statement
``if $E\subsetneqq\mathbb{R}^{n+1}$ is a nonempty $\gamma$-FIT parabolic weakly porous set
with $\gamma\in[0,1)$, then there exists $\alpha\in(0,\infty)$ such that
$[d_p(\cdot,E)]^{-\alpha}\in A_1^+(\gamma)$" holds remains an open question.
\end{remark}

\medskip

\noindent\textbf{Acknowledgements}\quad Weiyi Kong would like to
express his deep thanks to Jin Tao and Chenfeng Zhu for some helpful
discussions.

\bigskip

\noindent Weiyi Kong, Dachun
Yang (Corresponding author) and Wen Yuan

\medskip

\noindent Laboratory of Mathematics
and Complex Systems (Ministry of Education of China),
School of Mathematical Sciences,
Beijing Normal University,
Beijing 100875,
The People's Republic of China

\smallskip

\noindent{\it E-mails:} \texttt{weiyikong@mail.bnu.edu.cn} (W. Kong)

\noindent\phantom{{\it E-mails:} }\texttt{dcyang@bnu.edu.cn} (D. Yang)

\noindent\phantom{{\it E-mails:} }\texttt{wenyuan@bnu.edu.cn} (W. Yuan)
\end{document}